\documentclass[a4paper,final]{amsart}

\usepackage{hyperref}
\usepackage{amsxtra,amssymb}
\usepackage[czech,english]{babel}
\usepackage{mdwtab}
\usepackage{mathtools}
\usepackage[OT1]{fontenc}
\usepackage{ifthen}

\newcommand{\seq}[1]{\pmb{\mathsf{#1}}}
\newcommand{\Xseq}[1]{\mathsf{#1}}

\newcommand{\lift}[2][]{{\rm L}_{#1}(#2)}
\newcommand{\ball}[2][]{{\rm N}_{#2}\ifthenelse{\equal{A#1}{A}}{}{^{#1}}}

\newcommand{\bbbn}{\mathbb{N}}
\newcommand{\bbbz}{\mathbb{Z}}
\newcommand{\bbbr}{\mathbb{R}}

\usepackage{color}
\definecolor{shadecolor}{gray}{.85}%
\definecolor{tintedcolor}{gray}{.80}%
\usepackage{framed}%
\definecolor{mytintedcolor}{gray}{.95}%
\makeatletter
\newdimen\svparindent
\newenvironment{mytinted}{%
  \MakeFramed {\FrameRestore}}%
{\endMakeFramed}
\newenvironment{great}{%
\fboxsep=12pt\relax

        \vbox\bgroup\begin{mytinted}%
        \list{}{\leftmargin=12pt\rightmargin=2\leftmargin\leftmargin=\z@\topsep=\z@\relax}%
        \expandafter\item\parindent=\svparindent
        \relax}%
{\endlist\end{mytinted}\egroup}
\makeatother

\theoremstyle{plain}
\newtheorem{theorem}{Theorem}

\newtheorem{lemma}{Lemma}
\newtheorem{corollary}{Corollary}
\theoremstyle{definition}

\newtheorem{definition}{Definition}

\newtheorem{example}{Example}

\theoremstyle{remark}

\newtheorem{fact}{Fact}

\newtheorem{remark}{Remark}
\newcounter{tmpthm}

\title{Cluster Analysis of Local Convergent Sequences of Structures}
\thanks{Supported by grant ERCCZ LL-1201 
and CE-ITI, and by the European Associated Laboratory ``Structures in
Combinatorics'' (LEA STRUCO) P202/12/G061}
\author{Jaroslav Ne{\v s}et{\v r}il}
\address{Jaroslav Ne{\v s}et{\v r}il\\
Computer Science Institute of Charles University (IUUK and ITI)\\
   Malostransk\' e n\' am.25, 11800 Praha 1, Czech Republic}
\email{nesetril@kam.ms.mff.cuni.cz}
\author{Patrice~Ossona~de~Mendez}
\address{Patrice~Ossona~de~Mendez\\
Centre d'Analyse et de Math\'ematiques Sociales (CNRS, UMR 8557)\\
  190-198 avenue de France, 75013 Paris, France
  and
     Computer Science Institute of Charles University (IUUK)\\
   Malostransk\' e n\' am.25, 11800 Praha 1, Czech Republic}
 \email{pom@ehess.fr}
 \date{\today}
 
\subjclass[2010]{Primary  03C13 (Finite structures), 03C98 (Applications of model theory), 05C99 (Graph theory),  06E15 (Stone spaces and related structures), Secondary 28C05 (Integration theory via linear functionals)}

 \keywords{Graph \and Relational structure \and Graph limits \and Structural limits \and Radon measures \and Stone space \and Model theory \and First-order logic \and Measurable graph}

\begin{document}
 \begin{abstract}
 The cluster analysis of very large objects is an important problem, which spans several theoretical as well as applied branches of mathematics and computer science.
 Here we suggest a novel approach: under assumption of local convergence of a sequence of finite structures we derive an asymptotic clustering. This is achieved by a blend of analytic and geometric techniques, and particularly by a new interpretation of the authors' representation theorem for limits of local convergent sequences, which serves as a guidance for the whole process. Our study may be seen as an effort to describe connectivity structure at the limit (without having a defined explicit limit structure) and to pull this connectivity structure 
 back to the finite structures in the sequence in a continuous way. 
 \end{abstract}
 \maketitle
 \tableofcontents
  \section{Introduction}
  Cluster analysis (being established part of statistics, computer science and mathematics) is a core method for database mining. It
initiated in the thirties in social sciences, particularly in anthropology and psychology. 
 While the abstract notion of a cluster is somehow vague, some canonical types of cluster models have been considered, which allow to construct meaningful partitions of large data sets. Among these models, let us mention two principal extreme models:
{\em density models} --- where clusters correspond to connected dense regions, and {\em distribution models} --- where clusters are defined by means of statistical distributions.
For a comprehensive review of cluster analysis, we refer the reader to~\cite{Everitt2011}.

In this paper --- which extends and precise some ideas introduced by the authors in \cite{modeling} to study structural limits of trees ---
we propose a novel approach based on an interplay of these two models: knowing a limit statistical distribution associated to structures in a convergent sequence,
we compute the parameters driving a density clustering 
of each of the structures in the sequence, in a seemingly ``continuous'' way. 
 We believe that the cluster analysis presented here has a broader impact than the analysis of structural limits (which was our original motivation), and that it 
highlights a duality of the density and distribution models.
Our analysis found immediate applications to the study of structural limits and we hope that more will come.

The convergence notion we use is the convergence of the distribution of the local properties of random vectors of elements. The limit distribution is used to drive a segmentation process, which can be seen as a marking of the elements of each structure in the sequence. The consistency of these markings is ensured by the requirement that the sequence of marked structures is still local convergent (see Fig~\ref{fig:segment} for a schematic visualization of this segmentation method).

\begin{figure}[ht]
	\begin{center}
		\includegraphics[width=\textwidth]{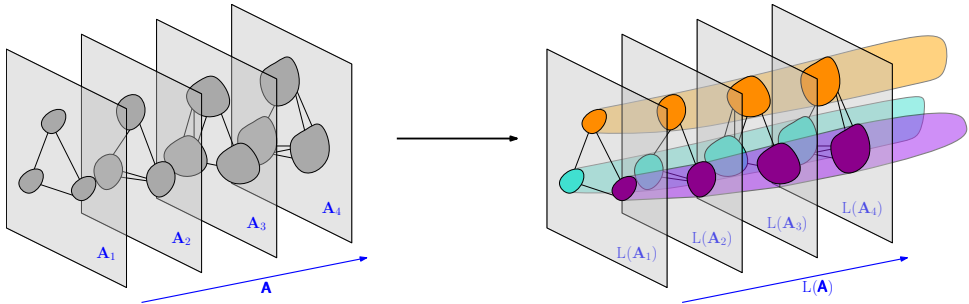}
	\end{center}
	\caption{Segmentation of structures in a convergent sequence based on cluster analysis}
	\label{fig:segment}
\end{figure}

Our approach is a natural one: if instead of considering a single snapshot of an evolving system we consider a significant part of the full movie, then clusters appear in a more obvious way, and meaningful parameters are much more easily defined and estimated. However the details are involved and lead to a new taxonomy.

Note that the notion of convergence considered here is a generalization  of the notion of {\em local convergence} introduced by Benjamini and Schramm for graphs with bounded degrees \cite{Benjamini2001}. 
In the general structural setting, introduced by the authors in \cite{CMUC}, there is no restriction on the degrees of the considered graphs or structures. Informally, a sequence $\seq{A}=(\mathbf A_n)_{n\in\bbbn}$ of structures is {\em local convergent} if the probability 
$\langle\phi, \mathbf A_n\rangle$ (the {\em Stone pairing} of $\phi$ and $\mathbf{A}_n$)
of satisfaction of every local first-order formula $\phi$  in structure $\mathbf A_n$ (for a random assignment of the free variables) converges as $n$ grows to infinity. (Recall that a {\em local formula} is a formula whose satisfaction only depends on a bounded neighborhood of its free variables.)
The limit of a local convergent sequence can thus be described by the (infinite) vectors of limit satisfaction probabilities $\lim_{n\rightarrow\infty}\langle\phi, \mathbf A_n\rangle$ indexed by all local first-order formulas $\phi$. This can also be represented as a probability measure, as stated in the general representation theorem (Theorem~\ref{thm:rep}), in a way extending Aldous-Hoover representation of left limits of dense graphs by infinite exchangeable graphs \cite{Aldous1981,Hoover1979} and
Benjamini-Schramm representation of local limits of graphs with bounded degree by an unimodular distribution on rooted connected countable graphs \cite{Benjamini2001}.

Our cluster analysis allows to meaningfully partition the structures in a local convergent sequence into dense connected clusters (plus an additional residual sparse cluster). It also show 
how this clustering is related to an imaginary connectivity structure of the limit (although no {\em bona fide} limit structure is generally available). More: our cluster analysis will be a central tool to construct limit structures for sequences of graphs with locally few cycles (meaning that the number of cycles in the $d$-neighborhood of every vertex in every graph in the sequence is bounded by some fixed function of $d$). This will be the subject of a forthcoming paper \cite{modeling}.
 
 \begin{figure}[ht]
 	\begin{center}
\includegraphics[width=\textwidth]{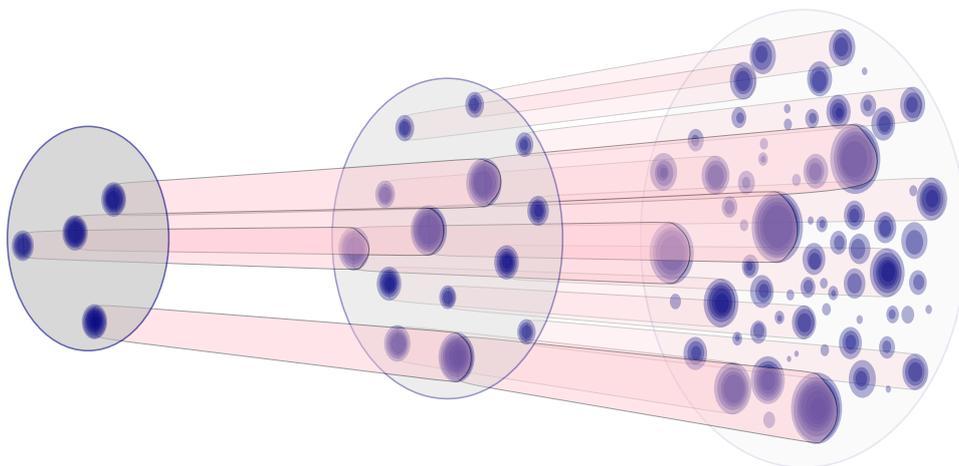}
\caption{Typical shape of a structure  continuously segmented by a clustering: dense spots correspond to globular clusters, and the background to the residual cluster.
Biggest globular clusters appear first and then move apart from each other, while new (smaller) globular clusters appear and residual cluster becomes sparser and sparser.}
\label{fig:milk}
 	\end{center}
 \end{figure}
 
 Let us take time for a more detailed description both of our main result (Theorem~\ref{thm:main}) and of the main difficulties that we have to overcome to prove it. The first (surprising, at least at first glance) aspect, which already appears when considering Benjamini-Schramm limit of connected graphs with bounded degrees, is that the limit of a sequence of connected graphs needs not to be connected: if $\seq{G}=(\mathbf G_n)_{i\in\bbbn}$ is a local convergent sequence of finite connected graphs with degree at most $D$ and with orders growing to infinity, then for every integers $k,r$ the probability that a random subset of $k$ vertices contains two vertices at distance at most $r$ tends to $0$, which ultimately shows that the limit cannot have finitely many connected components. Actually every limit {\em graphing} will have uncountably many connected components. 
  
 When considering general local convergent sequences of finite structures, even if we don't have a  limit structure, it makes sense to talk about the limit connected components and some of their properties. For instance, we prove that it is possible to determine the measure of all the limit connected components and, for those with non-zero measures, their associated statistics. This is basically done by using  Fourier analysis. Using this information, we prove that it is possible to track the limit connected components back to the structures of a local convergent sequence, by marking consistently the elements of all the structures in the sequence (see also schematic Fig.~\ref{fig:milk}).  
  The component structure of the limit is very complex and it has been repeatedly asked as a problem (by Lov\'asz and others) how sequences of connected structures disconnect at the limit.
      Here we solve this problem at a general level, by showing that we can trace limit connected components with positive measure back in the sequence and how they gradually disconnect themselves from the remaining of the structures.

 This analysis leads to interesting new notions (see Fig~\ref{fig:thesaurus}):
  {\em globular cluster} (corresponding to a limit non-zero measure connected component), {\em residual cluster} (corresponding to all the zero-measure connected components taken as a whole), and {\em negligible cluster} (corresponding to the stretched part connecting the other clusters, which eventually disappears at the limit). The marking of each of all these types of clusters will be explained in the second part of the paper. But let us mention that 
  the main issue here is that we require that the marking of all these (countably many) clusters should  preserve local convergence. This means that even if we consider local formulas using these marks, the satisfaction probabilities will still converge.
  
  \begin{figure}
  \begin{center}
	\includegraphics[width=\textwidth]{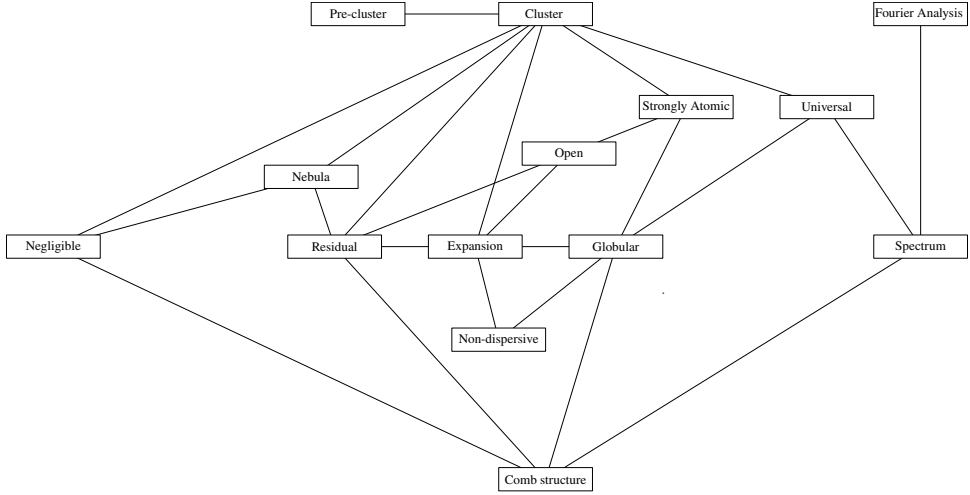}
\end{center}
  \caption{Semantic connections of new notions considered in this paper.}
  	\label{fig:thesaurus}
  \end{figure}
 
 The main result of this paper reads as follows:
 \begin{great}
 \begin{theorem}
 	\label{thm:main}
 	Let $\seq{A}$ be a local convergent sequence of $\sigma$-structures. Then there exists a signature
  $\sigma^+$ obtained from $\sigma$ by the addition of countably many unary symbols $M_R$ and $M_{i,j}$ ($i\in\bbbn$, $1\leq j\leq N_i$) and a clustering $\seq{A}^+$ of $\seq{A}$ with the following properties:
 	\begin{itemize}
 		\item For every $i\in\bbbn$, $\bigl(\bigcup_{j=1}^{N_i}M_{i,j}(\mathbf A_n^+)\bigr)_{n\in\bbbn}$  is a universal cluster; 
 		\item For every $i\in\bbbn$ and every $1\leq j\leq N_i$, $\bigl(M_{i,j}(\mathbf A_n^+)\bigr)_{n\in\bbbn}$  is a globular cluster; 
 		\item Two clusters $\bigl(M_{i,j}(\mathbf A_n^+)\bigr)_{n\in\bbbn}$ and $\bigl(M_{i',j'}(\mathbf A_n^+)\bigr)_{n\in\bbbn}$ are interweaving if and only  if $i=i'$;
 		\item $\bigl(M_R(\mathbf A_n^+)\bigr)_{n\in\bbbn}$  is a residual cluster.
 	\end{itemize}
(all undefined notions are explained below.)
 \end{theorem}	
 \end{great}

The paper is organized in three parts, each subdivided into sections: 

In the first part of the paper we introduce (in Section~\ref{sec:prelim}) the main definitions and notations used in this paper, and present (in Section~\ref{sec:local})
a reduction
argument showing  that considering {\em strongly local formulas} (that is local formulas whose satisfaction requires that all the free variables are assigned to vertices which are close) is sufficient to compute (exact) statistics component-wise, possibly after deletion of some set with negligible impact.  For this purpose, a ``weak algebra'' of strongly formulas is developed. 

 The second part of the paper is devoted to the theoretical study of an abstract notion of ``cluster''.
 Section~\ref{sec:neg} is devoted to the study of
  sets with negligible impact, called {\em negligible sets} and to sequence of more and more negligible sets, called {\em negligible sequences}. Deletion of subsets forming a negligible sequence does not change the limit statistics of a local convergent sequence. Ultimately, our goal is thus to consider a negligible sequence, whose deletion will disconnect the graphs in the sequence into clusters. The formal notion of a {\em cluster}  is discussed in Section~\ref{sec:cluster}.
For us, a cluster will be a ``continuous'' sequence of subsets that correspond to a ``stable entity'' and that is ``well separated'' from the rest of the structure. This is expressed by the property that marking a cluster (formalized by considering a lift) preserves the local convergence, and that the frontier of a cluster forms a negligible sequence. 
	Several types of clusters are defined and discussed in this section, in particular {\em universal} clusters and {\em strongly atomic} clusters. These last clusters corresponding to expanding parts of the structures in a local convergent sequence, and their properties are close to those of expander graphs.
 	It follows from the definition of a cluster that iteratively marking finitely many clusters preserves local convergence. However, if we want to mark countably many clusters then the situation becomes more tricky. The conditions under which countably many clusters can be marked, sometimes modulo a limited modification, is discussed in Section~\ref{sec:comb}, and is the purpose of the Cluster Comb Lemma (Lemma~\ref{lem:comb}).

Particular clusters are intrinsically defined by the local convergence, which allow to mark dense spots in the structures of a local convergent sequence. These clusters, called {\em globular clusters}, ultimately represent the non-zero measure imaginary connected components of the limit. To the opposite, a  {\em residual} cluster represents a group of zero-measure imaginary connected components.
 	
The third part of the paper is devoted to effective density clustering
into countably many clusters and a residual cluster.
In Section~\ref{sec:rep} we review the general representation theorem for limits of convergent sequences, and prove a general random rooting theorem using Fourier analysis.	This result allows us to compute the {\em spectrum} of the sequence, from which we derive the asymptotic measures of the globular clusters. Using these informations, the actual computation of the clustering is done, and we deduce a complete characterization of all the globular clusters of the sequence, the computed globular clusters serving as a ``globular basis''.
in Section~\ref{sec:glob}.

\part{Preliminaries}
  \section{Basic Definitions and Notations}. 
  \label{sec:prelim}
 The theory of graph (and structure) convergence gained recently a substantial attention. Various notions of convergence were proposed, adapted to different contexts. Let us mention:
 \begin{itemize}
 \item the theory of dense graph limits \cite{Borgs20081801,Lov'asz2006} based on the notion of {\em left convergence},
 \item the theory of bounded degree graph limits 
 \cite{Benjamini2001} based on the notion of {\em local convergence}.
 \end{itemize}
 These approaches have been (partly) unified by the authors in the setting of 
 {\em structural limits} \cite{CMUC}. This last approach relies on a
 balance of model theoretic and functional analysis aspects.
 For a signature $\sigma$ and a fragment $X$ of the set of first-order formulas over the language generated by $\sigma$, we define for a finite $\sigma$-structure $\mathbf A$ and a formula $\phi\in X$ with free variables $x_1,\dots,x_p$ the {\em Stone pairing} of $\phi$ and $\mathbf A$ as
 $$
 \langle\phi,\mathbf A\rangle=\frac{|\phi(\mathbf A)|}{|A|^p},
 $$
 where $\phi(\mathbf A)=\{(v_1,\dots,v_p)\in A^p:\ \mathbf{A}\models\phi(v_1,\dots,v_p)\}$. In other words, $\langle\phi,\mathbf{A}\rangle$ is the probability that $\phi$ is satisfied in $\mathbf A$ for a random (uniform independent) assignment of the free variables $x_1,\dots,x_p$ to elements of $A$. 
 
 The above setting naturally extends to the case where a structure
 $\mathbf A$ is equipped with a probability measure $\nu_{\mathbf A}$ on its domain. In this case, we define the Stone pairing as
 $$
 \langle\phi,\mathbf A\rangle=\nu_{\mathbf A}^{\otimes p}(\mathbf A),
 $$
 where $\nu_{\mathbf A}^{\otimes p}$ stands for the product measure on $A^p$. 
 In this paper we deal with finite structures endowed with a probability measure (which we briefly call {\em structures} for the sake of simplicity); when not defined, the probability measure considered on a finite structure is meant to be the uniform measure. The class of all the finite structures with signature $\sigma$ will be denoted by ${\rm Rel}(\sigma)$.
 
 \begin{table}
 \begin{tabular}{|c|l|}
 \hlx{hv}
Symbol&\multicolumn{1}{c|}{Signification}\\
\hlx{vhv}
${\rm Rel}(\sigma)$&Set of all finite $\sigma$-structures (endowed with probability measure)\\
$\mathbf{A}$&A structure\\
$A$&The domain of structure $\mathbf A$\\
$\nu_{\mathbf A}$&Probability measure on the domain $A$ of $\mathbf A$\\
$X,Y$&Subsets of $A$\\
$X\cap Y, X\cup Y, X\setminus Y$&Set operations\\
$X\subseteq Y$&Set inclusion\\
\hlx{vhv}
$\phi(\mathbf A)$&Set of tuples satisfying $\phi$ in $\mathbf A$\\
$\langle\phi,\mathbf A\rangle$&Stone pairing of $\phi$ and $\mathbf A$\\
$\mathbf A[X]$&Substructure of $\mathbf A$ induced by $X$\\
$\mathbf A-X$&Substructure of $\mathbf A$ induced by $A\setminus X$\\
$\ball[d]{\mathbf A}(X)$&Closed $d$-neighborhood of subset $X$ in $\mathbf A$\\
$\partial_{\mathbf A}(X)$&Outer boundary of $X$ in $\mathbf A$: $\partial_{\mathbf A}(X)=\ball{\mathbf A}(X)\setminus X$\\
\hlx{vhhv}
$\seq{A}$&Sequence $(\mathbf A_n)_{n\in\bbbn}$ of structures\\
$\Xseq{A}$&Sequence $(A_n)_{n\in\bbbn}$ of the domains of structures in $\seq{A}$\\
$\Xseq{X}$&Sequence $(X_n)_{n\in\bbbn}$ of subsets, cluster\\
$\Xseq{0}$&Sequence of empty sets: sequence $\Xseq{X}$, where $X_n=\emptyset$\\
$\seq{A}_f$&Subsequence $(\mathbf A_{f(n)})_{n\in\bbbn}$ of $\seq{A}$\\
$\Xseq{X}_f$&Subsequence $(X_{f(n)})_{n\in\bbbn}$ of $\Xseq{X}$\\
\hlx{vhv}
$\nu_{\seq{A}}(\Xseq{X})$&Sequence $(\nu_{\mathbf A_n}(X_n))_{n\in\bbbn}$ of
measures of subsets\\
$\Xseq{X}\cap\Xseq{Y}, \Xseq{X}\cup\Xseq{Y}, \Xseq{X}\setminus\Xseq{Y}$&Sequences $(X_n\cap Y_n)_{n\in\bbbn}, (X_n\cup Y_n)_{n\in\bbbn}$, and  $(X_n\setminus Y_n)_{n\in\bbbn}$\\
$\Xseq{X}\subseteq \Xseq{Y}$&Pointwise sequence inclusion: $\Xseq{X}\subseteq \Xseq{Y}\ \iff\ (\forall n)\ X_n\subseteq Y_n$\\
$\seq{A}[\Xseq{X}]$&Sequence $(\mathbf A_n[X_n])_{n\in\bbbn}$ of induced substructures\\
$\seq{A}-\Xseq{X}$&Sequence $(\mathbf A_n[A_n\setminus X_n])_{n\in\bbbn}$ of induced substructures\\
$\ball[d]{\seq{A}}(\Xseq{X})$&Sequence $(\ball[d]{\mathbf A_n}(X_n))_{n\in\bbbn}$ of closed $d$-neighborhoods\\
$\partial_{\seq{A}}\Xseq{X}$&Sequence $(\partial_{\mathbf A_n}(X_n))_{n\in\bbbn}$ of outer boundaries\\
$\phi(\seq{A})$&Sequence $(\phi(\mathbf A_n))_{n\in\bbbn}$ of satisfaction sets of $\phi$\\
$\lift[\Xseq{X}]{\seq{A}}$&Lifted sequence obtained by marking $\Xseq{X}$ in $\seq{A}$ \\
\hlx{vhhv}
$\lim\seq{A}$&Limit of $\seq{A}$ (as an abstract object)\\
$\langle\phi,\lim\seq{A}\rangle$&Limit Stone pairing: $\langle\phi,\lim\seq{A}\rangle=\lim_{n\rightarrow\infty}\langle\phi,\mathbf A_n\rangle$\\
\hlx{vhhv[1]}
\multicolumn{2}{|c|}{Introduced in Section~\ref{sec:rep}}\\
\hlx{v[1]hv}
$S_\sigma$&Stone space associated to $\sigma$-structures\\
$P(S)$&Space of probability measures on space $S$\\
$\mathfrak M_\sigma$&Closure of the space of representation measures of finite\\
&$\sigma$-structures in $P(S_\sigma)$\\
$\mu_{\mathbf A}$&Representation measure of structure $\mathbf A$\\
$\mu_{\lim\seq{A}}$&Representation measure of the limit of sequence $\seq{A}$\\
$k(\phi)$&Function representing $\phi$, s.t. $\langle\phi,\mathbf A\rangle=\int k(\phi)\,{\rm d}\mu_{\mathbf A}$\\
\hlx{vhhv[1]}
\multicolumn{2}{|c|}{Introduced in Section~\ref{sec:local}}\\
\hlx{v[1]hv}
$\phi\oplus\psi$&addition: $\phi\vee\psi$, defined if $\phi\wedge\psi=0$\\
$\phi\ominus\psi$&subtraction: $\phi\wedge\neg\psi$, defined if $\phi\rightarrow\psi$\\
$\phi\otimes\psi$&free product of $\phi$ and $\psi$\\
\hlx{vhhv[1]}
\multicolumn{2}{|c|}{Introduced in Section~\ref{sec:neg}}\\
\hlx{v[1]hv}
$\Xseq{X}\approx\Xseq{Y}$&Equivalent sequences ($\Xseq{X}\Delta\Xseq{Y}$ negligible)\\
\hlx{vhhv[1]}
\multicolumn{2}{|c|}{Introduced in Section~\ref{sec:cluster}}\\
\hlx{v[1]hv}
$\Xseq{X}\between\Xseq{Y}$&Interweaving clusters ($\lim\lift[\Xseq{X}]{\seq{A}}=\lim\lift[\Xseq{Y}]{\seq{A}}$)\\
\hlx{vhss}
\end{tabular}
 	\caption{Main symbols and notations of this paper}
 	\label{tab:not}
 \end{table}
 
In the following, we shall use the following convention (see Table~\ref{tab:not}):
\begin{itemize}
\item Structures are denoted by boldface capital letters $\mathbf{A}$;
\item Sets are denoted by plain roman capital letters $X,Y$;
\item Sequences of structures are denoted by boldface capital sans serif letter $\seq{A}=(\mathbf{A}_n)_{n\in\bbbn}$;
\item Sequence of sets are denoted by plain capital sans serif letter $\Xseq{X}=(X_n)_{n\in\bbbn}$.
\end{itemize} 

Let $\mathbf A$ be a $\sigma$-structure and let $X$ be a subset of the domain $A$ of $\mathbf A$.
If $\nu_{\mathbf A}(X)>0$ we define $\mathbf A[X]$ as the substructure of $\mathbf A$ induced by $X$ endowed with probability measure defined by $\nu_{\mathbf A[X]}(Y)=\nu_{\mathbf A}(Y)/\nu_{\mathbf A}(X)$ (for every $Y\subseteq X$). 
If $\nu_{\mathbf A}(X)<1$ we define 
$\mathbf A-X=\mathbf A[A\setminus X]$. 
 We denote by ${\rm Gaif}(\mathbf{A})$ the Gaifman graph of $\mathbf A$. The {\em distance} between two elements of a structure will always refer to the graph distance in the Gaifman graph of the structure. 

For a subset $X\subseteq A$,  the {\em closed neighborhood} of $X$ in $\mathbf{A}$ is $\ball{\mathbf{A}}(X)$. Consequently the set of all elements of $A$ at distance at most $d$ from an element of $X$ is
$\ball[d]{\mathbf{A}}(X)$.
The {\em outer vertex boundary} (or simply the {\em outer boundary}) of $X$ in $\mathbf{A}$ 
is the set of vertices in $A\setminus X$ with at least one neighbor in $X$ \cite{bobkov2000vertex}:
$$\partial_{\mathbf{A}}X=\ball{\mathbf{A}}(X)\setminus X.$$
Note, in particular, that if $X$ is the domain of a union of connected components of $\mathbf{A}$, then $\partial_{\mathbf{A}}X=\emptyset$.

Furthermore, we extend all operations defined on structures and subsets to sequences coordinate-wise:
The sequence $\seq{A}$ has domain $\Xseq{A}$ (meaning $\mathbf{A}_n$ has domain $A_n$); for
$\Xseq{X}\subseteq \Xseq{A}$ (meaning $X_n\subseteq A_n$) we denote by $\seq{A}[\Xseq{X}]$ the sequence $(\mathbf{A}_n[X_n])_{n\in\bbbn}$, by $\partial_{\seq{A}}\Xseq{X}$ the sequence $(\partial_{\mathbf{A}_n}X_n)_{n\in\bbbn}$, 
by $\Xseq{X}\subseteq\Xseq{Y}$ the inclusions $X_n\subseteq Y_n$, 
by $\phi(\seq{A})$ the sequence of the sets $\phi(\mathbf{A}_n)$, 
by $\nu_{\seq{A}}(\Xseq{X})$ the sequence of the measures
$\nu_{\mathbf A_n}(X_n)$,
etc.
Also, for increasing $f:\bbbn\rightarrow\bbbn$ we denote by 
$\seq{A}_f$ the subsequence $(\mathbf{A}_{f(n)})_{n\in\bbbn}$ of $\seq{A}$ and by $\Xseq{X}_f$ the subsequence $(X_{f(n)})_{n\in\bbbn}$ of $\Xseq{X}$. 	
For instance,  $\seq{A}$ denotes a sequence of structures whose $n$th term is
$\mathbf{A}_n$, and $\Xseq{A}$ denotes the sequence of the domains $A_n$ of the structures $\mathbf{A}_n$.

\begin{definition}
For $\sigma$-structures $\mathbf A_1,\mathbf A_2, \dots$
and not negative reals $\lambda_1,\lambda_2, \dots$
with sum $1$ we define $\sum_i \lambda_i\mathbf A_i$ as the $\sigma$-structure
$\mathbf A$ obtained by endowing the disjoint union of the 
$\sigma$-structures $\mathbf A_1,\mathbf A_2, \dots$ with the probability measure
$\nu_{\mathbf A}=\sum_i \lambda_i \nu_{\mathbf A_i}$.
\end{definition}

Note that although this allows us to define $\langle\phi,\sum_i \lambda_i\mathbf A_i\rangle$, in general we have
$\langle\phi,\sum_i \lambda_i\mathbf A_i\rangle\neq\sum_i \lambda_i\langle\phi,\mathbf A_i\rangle$. However, equality holds in the very particular case where $\phi$ is a local formula with a single free variable. When $\phi$ is a general local formula with $p$ free variables, it is possible to express $\langle\phi,\sum_i \lambda_i\mathbf A_i\rangle$ as a polynomial of degree at most $p$ in terms of the form $\langle\phi_j,\mathbf A_i\rangle$, for some strongly local formulas $\phi_j$ depending on $\phi$ (see Corollary~\ref{cor:slocP}).

To deal marking we introduce the following notion of lift:
\begin{definition}
Let $\sigma\subset\sigma^+$ be countable signatures,
let $\seq{A}$ be a sequence of $\sigma$-structures, and let $\seq{B}$ be a sequence of $\sigma^+$-structures.

The sequence $\seq{A}$ is the {\em shadow} of the sequence $\seq{B}$ if, for each $n\in\bbbn$, the structure $\mathbf A_n$ is the structure obtained from $\mathbf B_n$ by ``forgetting'' about all relations not in $\sigma$.
Conversely, the sequence $\seq{B}$ is a {\em lift} of the sequence $\seq{A}$ if $\seq{A}$ is the shadow of $\seq{B}$.
The sequence  $\seq{B}$ is a {\em conservative lift} of the sequence $\seq{A}$ if, for each $n\in\bbbn$, the structures $\mathbf A_n$ and $\mathbf B_n$ have the same Gaifman graph.
\end{definition}
In this paper, a lift of 
a sequence $\seq{A}$ will usually be denoted by 
$\lift{\seq{A}}$, with possibly adding some subscripts to differentiate different lifts of a same sequence. In particular, if $\Xseq{X}$ is a sequence of subsets of $\seq{A}$ (i.e. $X_n\subseteq A_n$) and $\sigma^+$ is the signature obtained from $\sigma$ by adding a single unary symbol $M$, we shall denote by $\lift[\Xseq{X}]{\seq{A}}$ the lift  of $\seq{A}$ such that $M(\lift[\Xseq{X}]{\seq{A}})=\Xseq{X}$.

For the benefit of the reader we included in Table~\ref{tab:not} a list of the main symbols and notations used throughout this paper.
\section{Reduction from Local Formulas to Strongly Local Formulas}
\label{sec:local}
Recall that a first-order formula $\phi$ is {\em local} if there is some integer $r$ such that the satisfaction of $\phi$ only depends on the distance $r$ neighborhood of its free variables. 
Let ${\rm FO}^{\rm local}(\sigma)$ be the fragment of local
first-order formulas (for given signature $\sigma$).
The following is the key definition.

\begin{definition}
A sequence $\seq{A}=(\mathbf{A}_n)_{n\in\bbbn}$ of $\sigma$-structure is
{\em local-convergent} if $(\langle\phi,\mathbf A_n\rangle)_{n\in\bbbn}$ convergences for every local formula $\phi$. 
\end{definition}
Note that for bounded degree graphs our notion of local convergence is equivalent to the notion of local convergence introduced in \cite{Benjamini2001} (see \cite{CMUC}).  For general graphs (or regular hypergraphs), local convergence is stronger than the left convergence considered by \cite{Lov'asz2006,ElekSze}.

Before discussing the notion of local convergence in greater detail, we take time for few definitions.

\begin{fact}
\label{fact:bord}
Let $X,Y$ be  subsets of the domain $A$ of a structure $\mathbf{A}$, let $d$ be an integer and let $Z$ be any of 
$X\cap Y$, $X\cup Y$, $X\setminus Y$, $Y\setminus X$, $X\,\Delta\,Y$, and their complements in $A$. Then
it holds
$$
\ball[d]{\mathbf{A}}(\partial_{\mathbf{A}}Z)\subseteq
\ball[d+1]{\mathbf{A}}(\partial_{\mathbf{A}}X)\cup
\ball[d+1]{\mathbf{A}}(\partial_{\mathbf{A}}Y).
$$
\end{fact}

 A first-order formula $\phi$ with free variables $x_1,\dots,x_p$ is {\em $r$-local} if, for every structure $\mathbf{A}$ and elements $v_1,\dots,v_p\in A$ it holds
$$
\mathbf{A}\models \phi(v_1,\dots,v_p)\quad\iff\quad \mathbf{A}[\ball[r]{\mathbf{A}}(\{v_1,\dots,v_p\})]\models \phi(v_1,\dots,v_p).
$$
A formula is called {\em local} if it is $r$-local for some $r$. The set of all local first-order formulas (in the language of the considered signature) is denoted by ${\rm FO}^{\rm local}$, and we simply use the term of {\em local convergence} for  ${\rm FO}^{\rm local}$-convergence. Note that 
this notion of convergence extends Benjamini-Schramm's notion of local convergence:
\begin{fact}[\cite{CMUC}]
A sequence $(G_n)_{n\in\bbbn}$ of graphs  with maximum degree at most $D$ is local-convergent (in the sense of Benjamini-Schramm) if and only if it is local-convergent (in the sense of ${\rm FO}^{\rm local}$-convergence).
\end{fact}
Indeed, by Gaifman locality theorem \cite{Gaifman1982}, for every local formula $\phi$ with $p$ free variables there exist $p$ formulas $\psi_1,\dots,\psi_p$ with single free variable, such that for every graph $G$ with bounded degrees it holds
$$
\langle\phi,G\rangle=\prod_{i=1}^p\langle\psi_i,G\rangle+o(1).
$$

The main interest of our definition of local-convergence is that it does not need any restriction on the degrees.

The notion of local formula can be strengthened by requiring that all the free variables are at bounded distance from each other. Precisely,
 a first-order formula $\phi$ with free variables $x_1,\dots,x_p$ is {\em strongly $r$-local} if it is $r$-local and the following entailment holds:
 $$
 \phi\vdash\bigwedge_{i=1}^p ({\rm dist}(x_i,x_j)\leq r).
 $$

A formula is called {\em strongly  local} if it is strongly $r$-local for some $r$. 

We now introduce a notion of "weak algebra" of formulas.
In the following definition, a formula $\phi$ is {\em packed} if its free variables are $x_1,\dots,x_p$ (for some $p\in\bbbn$). For a formula $\phi$ with free variables $x_{i_1},\dots,x_{i_p}$ and an injection $\iota:\bbbn\rightarrow\bbbn$, $\iota(\phi)$ denotes the formula $\phi$ where all the  occurrences of $x_j$ are replaced by $x_{\iota(j)}$. We denote by $\tau$ be the injection $i\mapsto i+1$.
\begin{definition}
A {\em weak algebra} of formula is a set $\mathcal S$ of 
(logical equivalence class of) formulas which is closed under the following (partially defined) operations:
\begin{enumerate}
\item If $\phi,\psi\in\mathcal{S}$ and $\phi\wedge\psi=0$ then
$\phi\oplus\psi:=\phi\vee\psi$ belongs to $\mathcal{S}$;
\item if $\phi,\psi\in\mathcal{S}$ and $\phi\rightarrow\psi$ then
$\phi\ominus\psi:=\phi\wedge\neg\psi$ belongs to $\mathcal{S}$;
\item if $\iota:\bbbn\rightarrow\bbbn$ is an injection and $\phi\in\mathcal{S}$ then $\iota(\phi)\in\mathcal{S}$;
\item if $\phi,\psi\in\mathcal{S}$ are packed and $\phi$ has $p$ free variables, then $\phi\otimes\psi:=\phi\wedge\tau^p(\psi)$ belongs to $\mathcal S$.
\end{enumerate}
\end{definition}

Note that for every $\phi,\psi$, we have:
\begin{itemize}
\item If $\phi\oplus\psi$ is defined then for every structure $\mathbf{A}$ it holds
\begin{align*}
(\phi\oplus\psi)(\mathbf A)&\equiv\phi(\mathbf{A})\cup\psi(\mathbf{A})\\
\langle\phi\oplus\psi,\mathbf A\rangle&=\langle\phi,\mathbf A\rangle+\langle\psi,\mathbf A\rangle
\end{align*}
\item If $\phi\ominus\psi$ is defined then for every structure $\mathbf{A}$ it holds
\begin{align*}
(\phi\ominus\psi)(\mathbf A)&\equiv\phi(\mathbf{A})\setminus\psi(\mathbf{A})\\
\langle\phi\ominus\psi,\mathbf A\rangle&=\langle\phi,\mathbf A\rangle-\langle\psi,\mathbf A\rangle
\end{align*}
\item If $\phi\otimes\psi$ is defined then for every structure $\mathbf{A}$ it holds
\begin{align*}
(\phi\otimes\psi)(\mathbf A)&=\phi(\mathbf{A})\times\psi(\mathbf{A})\\
\langle\phi\otimes\psi,\mathbf A\rangle&=\langle\phi,\mathbf A\rangle\,.\,\langle\psi,\mathbf A\rangle
\end{align*}
\end{itemize}

Here, the equivalence $X\equiv Y$ (with respect to domain $A$) means, for $X\subseteq A^p$ and $Y\subseteq A^q$ that there exist a permutation $\iota$ of $[p+q]$ such that
$$\iota(X\times A^q)=Y\times A^p.$$
Also, note that if $\phi\oplus\psi$ is defined then
$\phi=(\phi\oplus\psi)\ominus\psi$.
\begin{great}
\begin{theorem}
\label{thm:walg}
The smallest weak algebra containing all strongly local formulas is the weak algebra of all local formulas.
\end{theorem}
\end{great}
\begin{proof}
One direction is obvious (as local formulas form a weak algebra).
For the other direction, consider an $r$-local formula $\phi$ with free variables $x_i$ for $i\in I$.
Let $\mathfrak{F}_I$ be the set of all graphs with vertex set $I$.
Obviously it holds
$$
\phi=\bigoplus_{F\in\mathfrak{F}_I} \left(
\bigwedge_{ij\in E(F)}({\rm dist}(x_i,x_j)\leq 2r)\,\wedge\,\bigwedge_{ij\notin E(F)}({\rm dist}(x_i,x_j)>2r)\,\wedge\,\phi\right).
$$
It follows that we can restrict our attention to formulas of the form
used in the above sum for some $F$. Moreover, we can assume that
$I=[p]$ and that  the vertex sets $I_1,\dots,I_k$ of the 
connected components $F_1,\dots,F_k$ of $F$
are intervals of $[p]$. 
By locality property, we further assume that $\phi$ has the following form:
$$
\phi=\bigwedge_{ij\in E(F)}({\rm dist}(x_i,x_j)\leq 2r)\,\wedge\,\bigwedge_{ij\notin E(F)}({\rm dist}(x_i,x_j)>2r)\,\wedge\,\bigwedge_{z=1}^k \rho_z,
$$
where $\rho_z$ is $r$-local with set of free variables $\{x_i: i\in I_z\}$.

 We proceed by induction on $k$.
If $k=1$ then the formula is $2pr$-strongly local so the lemma holds.
Assume that $k>1$ and that the statement holds for less than $k$ connected components. 
For $1\leq z\leq k$ define
$$\phi_z= \bigwedge_{ij\in E(F_z)}({\rm dist}(x_i,x_j)\leq 2r)\,\wedge\,\rho_z.$$
Then it holds
$$
\bigotimes_{z=1}^k\phi_z=\bigoplus_{H\in\mathfrak{F}'}\left(\bigwedge_{ij\in E(H)}({\rm dist}(x_i,x_j)\leq 2r)\,\wedge\,\bigwedge_{ij\notin E(H)}({\rm dist}(x_i,x_j)>2r)\,\wedge\,\bigwedge_{z=1}^k \rho_z\right),
$$
where $\mathfrak{F}'$ is the set of all graphs $H$ with vertex set $[p]$ such that $H[I_z]=F_z$ for every $1\leq z\leq k$. Note that $\mathfrak{F}'$ contains exactly one graph with $k$ connected components, namely $F$, all the other ones containing strictly less than $k$ connected components. Thus, if we denote
$$\psi=\bigoplus_{H\in\mathfrak{F}'\setminus\{F\}}\left(\bigwedge_{ij\in E(H)}({\rm dist}(x_i,x_j)\leq 2r)\,\wedge\,\bigwedge_{ij\notin E(H)}({\rm dist}(x_i,x_j)>2r)\,\wedge\,\bigwedge_{z=1}^k \rho_z\right),
$$
it holds
$$
\phi=\bigotimes_{z=1}^k\phi_z\ominus\psi.
$$
By induction, $\psi$ belongs to the weak algebra generated by strongly local formulas, hence so does $\phi$. 
\end{proof}
We now take time for three important corollaries.

\begin{corollary}
\label{cor:slocP}
For every $r$-local formula $\phi$ with $p$ free variables, there exist finitely many $(2pr)$-strongly local formulas $\phi_{i}$ ($1\leq i\leq N$) and a polynomial
$P\in\bbbz[X_1,\dots,X_N]$ of degree at most $p$ such that for every structure $\mathbf{A}$ it holds
$$
\langle\phi,\mathbf{A}\rangle=P(\langle\phi_1,\mathbf{A}\rangle,\dots,\langle\phi_N,\mathbf{A}\rangle).
$$
\end{corollary}

Note also the following corollary, which allows to check that first-order definable subsets of (the power of) a measurable structure are measurable (with respect to product measure) by reduction to sets definable by strongly local formulas.

\begin{corollary}
Assume $\mathbf A$ is an infinite structure, whose domain is a measurable space. If, for every strongly local formula $\phi$ the
set $\phi(\mathbf{A})$ is measurable (with respect to product measure) then for every first-order formula $\psi$ the
set $\psi(\mathbf{A})$ is measurable (with respect to product measure).
\end{corollary}

\begin{proof}
This is a direct consequence of Theorem~\ref{thm:walg} and Gaifman locality theorem \cite{Gaifman1982}.
\end{proof}
\begin{corollary}
\label{cor:sloc}
A sequence $\seq{A}$ is local-convergent if and only if it is strong-local-convergent.
\end{corollary}

\part{Clustering Local Convergent Sequences}

The notion of clustering we develop in this part is based on the stability of the convergence of a sequence when marking certain subsets of the domains of the structures in the sequence. This justifies to relate it to the notion of lift introduced in Section~\ref{sec:prelim}.

\section{Negligible Sets and Sequences}
\label{sec:neg}
The following notion of negligible set  corresponds intuitively to parts of the graph  one can  remove, without a great modification of the statistics of the graph.

\begin{definition}
Let $\mathbf{A}$ be a structure, let $d\in\bbbn$ and let $\epsilon>0$.
A subset $X\subset A$ of elements of $\mathbf{A}$ is {\em $(d,\epsilon)$-negligible} in $\mathbf{A}$ if
 $$\nu_{\mathbf A}(\ball[d]{\mathbf{A}}(X))<\epsilon.$$
\end{definition}

The main property of $(d,\epsilon)$-negligible sets is the following:
\begin{great}
\begin{lemma}
\label{lem:approx0}
Let $\phi\in {\rm FO}_p^{\rm local}$ be $r$-local with $r<d$, and let $X$ be a 
$(d,\epsilon)$-negligible set of a structure $\mathbf{A}$. Then
$$
|\langle\phi,\mathbf{A}\rangle-\langle\phi,\mathbf{A}-X\rangle|<2p\epsilon.
$$
Moreover, suppose $\mathbf{B}$ is a structure with same domain  as $\mathbf{A}$ such that $\mathbf{A}-X=\mathbf{B}-X$ then
$$
|\langle\phi,\mathbf{A}\rangle-\langle\phi,\mathbf{B}\rangle|<p\epsilon.
$$
\end{lemma}
\end{great}
\begin{proof}
We first prove the second inequality.

Consider the lift  $\lift{\mathbf{A}}$ (resp. $\lift{\mathbf{B}}$) of $\mathbf{A}$ (resp. $\mathbf{B}$) where all elements in $\ball[d]{\mathbf{A}}(X)$ (resp. $\ball[d]{\mathbf{A}}(X)$) are marked with new unary relation $M$.
Let $\psi(x_1,\dots,x_p):=\bigvee M(x_i)$.
Then
it holds
\begin{align*}
0\leq\langle\phi,\lift{\mathbf{A}}\rangle-\langle\phi\wedge\neg\psi,\lift{\mathbf{A}}\rangle&<\langle \psi,\lift{\mathbf{A}}\rangle<1-(1-\epsilon)^p<p\epsilon \\
0\leq\langle\phi,\lift{\mathbf{B}}\rangle-\langle\phi\wedge\neg\psi,\lift{\mathbf{B}}\rangle&<\langle \psi,\lift{\mathbf{B}}\rangle<p\epsilon
\intertext{Thus, as $\phi(\lift{\mathbf A})=\phi(\mathbf A)$, 
$\phi(\lift{\mathbf B})=\phi(\mathbf B)$ and
$(\phi\wedge\neg\psi)(\lift{\mathbf A})=(\phi\wedge\neg\psi)(\lift{\mathbf B})$ it holds}
|\langle\phi,\mathbf A\rangle-\langle\phi,\mathbf B\rangle|<p\epsilon
\end{align*}

For the second inequality, we have likewise
$$
0\leq \langle\phi,\lift{\mathbf{A}}-X\rangle-\langle\phi\wedge\neg\psi,\lift{\mathbf{A}}-X\rangle<\langle \psi,\lift{\mathbf{A}}-X\rangle<\frac{p\epsilon}{\nu_{\mathbf A}(A\setminus X)^p}\\ 
$$
Moreover, as $\phi$ is $r$-local, it holds $(\phi\wedge\lnot\psi)(\lift{\mathbf{A}})=(\phi\wedge\lnot\psi)(\lift{\mathbf{A}}-X)$ hence
$$
\langle\phi\wedge\neg\psi,\lift{\mathbf{A}}\rangle=\nu_{\mathbf A}(A\setminus X)^p
\langle\phi\wedge\neg\psi,\lift{\mathbf{A}}-X\rangle.
$$
Thus
$$
|\langle\phi,\lift{\mathbf{A}}\rangle-\nu_{\mathbf A}(A\setminus X)^p\langle\phi,\lift{\mathbf{A}}-X\rangle|<p\epsilon
$$

Hence
\begin{align*}
	|\langle\phi,\lift{\mathbf A}\rangle-\langle\phi,\lift{\mathbf A}-X\rangle|
	&\leq |\langle\phi,\lift{\mathbf A}\rangle-\nu_{\mathbf A}(A\setminus X)^p\langle\phi,\lift{\mathbf A}-X\rangle|+1-\nu_{\mathbf A}(A\setminus X)^p\\
	&<p\epsilon+1-(1-\epsilon)^p\\
	&<2p\epsilon.
\end{align*}

hence the result, as $\langle\phi,\lift{\mathbf{A}}\rangle=\langle\phi,\mathbf{A}\rangle$ and $\langle\phi,\lift{\mathbf{A}}-X\rangle=\langle\phi, \mathbf{A}-X\rangle$.
\end{proof}

\begin{definition}
A {\em $(d,\epsilon)$-fragmentation} of a structure $\mathbf{A}$ is a (at most) countable partition $(S,X_1,X_2,\dots)$ of $A$ such that no element in $X_i$ has a neighbor in $X_j$ for $i\neq j$ and $S$ is $(d,\epsilon)$-negligible in $\mathbf{A}$.
\end{definition}
\begin{great}
\begin{lemma}
\label{lem:approx1}
Assume $(S,X_1,X_2,\dots)$ is a $(d,\epsilon)$-fragmentation of $\mathbf{A}$ and let $\phi$ be a strongly $r$-local formula ($r\leq d$) with free variables $x_1,\dots,x_p$.
Then
$$
\Bigl|\langle\phi,\mathbf{A}\rangle-
\sum_{i\geq 1}\nu_{\mathbf A}(X_i)^p\langle\phi,\mathbf{A}[X_i]\rangle\Bigr|
<2p\epsilon.
$$
\end{lemma}
\end{great}
\begin{proof}
This follows from Lemma~\ref{lem:approx0} and the fact that
$\phi(\mathbf{A}-S)$ is the disjoint union of the structures $\phi(\mathbf{A}[X_i])$.
\end{proof}

We now consider how the notion of $(d,\epsilon)$-negligible subset of a structure allows to define negligible sequences of subsets and equivalence of sequences.

\begin{definition}
Let $\seq{A}$ be a local-convergent sequence of structures.
A sequence $\Xseq{X}\subseteq\Xseq{A}$ is
{\em negligible} in $\seq{A}$ if
$$
\forall d\in\bbbn:\quad \limsup_{n\rightarrow\infty}\nu_{\mathbf A_n}(\ball[d]{\mathbf{A}_n}(X_n))=0,
$$
what we simply formulate as
$$
\forall d\in\bbbn:\quad \limsup\nu_{\seq{A}}(\ball[d]{\seq{A}}(\Xseq{X}))=0.
$$

Two sequences $\Xseq{X}$ and $\Xseq{Y}$ of subsets are {\em equivalent}
in $\seq{A}$ (and we note $\Xseq{X}\approx \Xseq{Y}$ if the sequence $\Xseq{X}\,\Delta\,\Xseq{Y}=(X_n\,\Delta\,Y_n)_{n\in\bbbn}$ is negligible in $\seq{A}$.

We denote by $\Xseq{0}$ the sequence of empty subsets. Hence $\Xseq{X}\approx\Xseq{0}$ is equivalent to the property that $\Xseq{X}$ is negligible.

We further define a partial order on sequences of subsets by $\Xseq{X}\preceq \Xseq{Y}$ if the sequence 
$\Xseq{X}\setminus \Xseq{Y}=(X_n\setminus Y_n)_{n\in\bbbn}$ is negligible in $\seq{A}$. Hence $\preceq$ has $\Xseq{0}$ for its minimum.

Two sequences  $\seq{A}$ and
$\seq{B}$  of structures are {\em equivalent} if there exists
a negligible sequence $\Xseq{X}$ of $\seq{A}$ and a
negligible sequence  $\Xseq{Y}$ of $\seq{B}$
 such that $\mathbf{A}_n-X_n$ is isomorphic to $\mathbf{B}_n-Y_n$ for every $n\in\bbbn$. 
\end{definition}

The following lemma is a straightforward consequence
of Lemma~\ref{lem:approx0}
but we think it nevertheless deserves to be explicitly stated.
\begin{lemma}
\label{lem:eqlim}
Let  $\seq{A}$ and
$\seq{B}$ be equivalent sequences of structures. 

Then 
$\seq{A}$ is local-convergent if and only if 
$\seq{B}$ is local-convergent. In this case, they have the same limit.
\end{lemma}

\section{Clusters and Pre-Clusters}
\label{sec:cluster}
The notion of cluster of a local-convergent sequence we introduce now is a weak analog of the notion of union of connected components, or more precisely of the topological notion of ``clopen subset''.

\subsection{Clusters}

In our setting, where clustering is performed on a local convergent sequence $\seq{A}$, the term of ``cluster'', which we will now define, will be used to qualify a sequence $\Xseq{X}$ of sets, with $X_n\subseteq A_n$.

\begin{definition}
Let  $\seq{A}$ be a local-convergent sequence of structures.

A sequence $\Xseq{X}\subseteq\Xseq{A}$  is a 
{\em cluster} of  $\seq{A}$ if the following conditions hold:
\begin{enumerate}
\item the lifted sequence $\lift[\Xseq{X}]{\seq{A}}$ obtained by marking set $X_n$ in $\mathbf{A}_n$ by a new unary relation $M_{\Xseq{X}}$
is local-convergent;
\item the sequence $\partial_{\seq{A}}\Xseq{X}$ is negligible in $\seq{A}$.
\end{enumerate}

Condition (1) can be seen as a continuity requirement for the subset selection.
Condition (2) is stronger than the usual requirement that there are not too many connections leaving the cluster. We intuitively require that the (asymptotically arbitrarily large) ring around a cluster is very sparse zone. Note that no minimality nor connectivity assumption is made at this point.

We start our ``cluster analysis'' by means of the following notions (more will follow, see Fig.~\ref{fig:thesaurus}):
A cluster $\Xseq{X}$ is
{\em atomic} if, for every cluster $\Xseq{Y}$ of $\seq{A}$ such that $\Xseq{Y}\preceq \Xseq{X}$ either $\Xseq{Y}\approx\Xseq{0}$ or $\Xseq{Y}\approx \Xseq{X}$; the cluster $\Xseq{X}$ is {\em strongly atomic} if $\Xseq{X}_f$ is an atomic cluster of $\seq{A}_f$ for every increasing function $f:\bbbn\rightarrow\bbbn$.
To the opposite, the cluster $\Xseq{X}$ is a {\em nebula}
if, for every increasing function $f:\bbbn\rightarrow\bbbn$, every atomic cluster $\Xseq{Y}_f$ of $\seq{A}_f$ with $\Xseq{Y}_f\subseteq \Xseq{X}_f$ is trivial.
Finally,
a cluster $\Xseq{X}$ is {\em universal} for  $\seq{A}$ if $\Xseq{X}$ is a cluster of every conservative lift of $\seq{A}$. 
\end{definition}

\begin{lemma}
\label{lem:eqcore}
Let $\Xseq{X}$ be a cluster of $\seq{A}$ and let $\Xseq{Y}$ be a sequence of subsets. Then  $\Xseq{X}\approx\Xseq{Y}$ in $\seq{A}$ if and only if $\Xseq{Y}$ is a cluster of $\seq{A}$  and
$$\limsup\nu_{\seq{A}}(\Xseq{X}\,\Delta\,\Xseq{Y})=0.$$ 
\end{lemma}
\begin{proof}
Assume $\Xseq{Y}$ is a cluster and 
$\limsup\nu_{\seq{A}}(\Xseq{X}\,\Delta\,\Xseq{Y})=0$.
For every integer $d$ it holds
\begin{align*}
\ball[d]{\mathbf{A}_n}(X_n\,\Delta\,Y_n)&\subseteq (X_n\,\Delta\,Y_n)\cup
\ball[d]{\mathbf{A}_n}(\partial_{\mathbf{A}_n}(X_n\,\Delta Y_n))\\
&\subseteq (X_n\,\Delta\,Y_n)\cup
\ball[d]{\mathbf{A}_n}(\partial_{\mathbf{A}_n}X_n)\cup 
\ball[d]{\mathbf{A}_n}(\partial_{\mathbf{A}_n}Y_n)
\end{align*}
Thus $\Xseq{X}\,\Delta\,\Xseq{Y}$ is negligible in $\seq{A}$, that is
$\Xseq{X}\approx\Xseq{Y}$.

Conversely, assume $\Xseq{X}\approx\Xseq{Y}$. Then obviously
$\limsup\nu_{\seq{A}}(\Xseq{X}\,\Delta\,\Xseq{Y})=0$. 
As $\partial_{\mathbf{A}_n}Y_n\subseteq \partial_{\mathbf{A}_n}X_n\cup \ball{\mathbf{A}_n}(X_n\,\Delta\, Y_n)$, and as
$X\Delta Y$ is negligible since $X\approx Y$,
the sequence
$\partial_{\seq{A}}\Xseq{Y}$ is negligible. Moreover, as
$\lift[\Xseq{X}]{\seq{A}}\approx \lift[\Xseq{Y}]{\seq{A}}$ (considering we use the same symbol for both lifts), we deduce that
$\Xseq{Y}$ is a cluster of $\seq{A}$.
\end{proof}
In particular, if $\Xseq{X}$ is a cluster and $\Xseq{Y}\approx\Xseq{X}$ then
$\Xseq{Y}$ is a cluster.

We have the following alternative characterization of clusters:

\begin{lemma}
\label{lem:cluster}
Let  $\seq{A}$ be a local-convergent sequence of structures.

A sequence $\Xseq{X}\subseteq\Xseq{A}$  is a 
cluster of  $\seq{A}$ if either $\Xseq{X}\approx\Xseq{0}$ or
the following conditions hold:
\begin{enumerate}
\item the sequence $\seq{A}[\Xseq{X}]$ is local-convergent;
\item the limit $\lim \nu_{\seq{A}}(\Xseq{X})$  and is strictly positive;
\item the sequence $\partial_{\seq{A}}\Xseq{X}$ is negligible in $\seq{A}$.
\end{enumerate}
\end{lemma}
\begin{proof}
Assume $\Xseq{X}$ is a cluster of $\seq{A}$, and let
$\alpha=\langle M_{\Xseq{X}},\lim\lift[\Xseq{X}]{\seq{A}}\rangle$.
If $\alpha=0$ then $\Xseq{X}$ is negligible in $\seq{A}$ as for every $d,n\in\bbbn$ it holds $\ball[d]{\mathbf{A}_n}(X_n)= X_n\cup 
\ball[d-1]{\mathbf{A}_n}(\partial_{\mathbf{A}_n}X_n)$ and thus 
$$\limsup\nu_{\seq{A}}(\ball[d]{\seq{A}}(\Xseq{X}))\leq \lim\nu_{\seq{A}}(\Xseq{X})+\limsup\nu_{\seq{A}}(\ball[d-1]{\seq{A}}(\partial_{\seq{A}}\Xseq{X})=\alpha=0.$$
Otherwise, $\alpha>0$. To every local formula $\phi$ we associate the local formula $\phi\vert M_{\seq{X}}$ conditioning every variables with $M_{\seq{X}}$. Then it holds
\begin{align*}
\langle\phi,\mathbf{A}_n[X_n]\rangle
&=\langle\phi\vert M_{\seq{X}},\lift[\Xseq{X}]{\mathbf{A}_n}[X_n]\rangle\\
&=\alpha^{-p}\langle\phi\vert M_{\seq{X}},\lift[\Xseq{X}]{\mathbf{A}_n}-\partial_{\mathbf A_n}(X_n)\rangle+o(1)\\
&=\alpha^{-p}\langle\phi\vert M_{\seq{X}},\lift[\Xseq{X}]{\mathbf{A}_n}\rangle+o(1)
\end{align*}
It follows that the sequence $\seq{A}[\Xseq{X}]$ is local-convergent.

Conversely,   let $\Xseq{X}$ be a sequence satisfying the conditions of the lemma. Then either $\Xseq{X}\approx\Xseq{0}$ and $\Xseq{X}$ is a cluster (according to Lemma~\ref{lem:eqcore}), or 
$\seq{A}[\Xseq{X}]$ is local convergent, $\lim\nu_{\seq{A}}(\Xseq{X})>0$, and $\partial_{\seq{A}}\Xseq{X}\approx\Xseq{0}$.
Then, denoting $\alpha=\lim\nu_{\seq{A}}(\Xseq{X})$ it holds for every local formula $\phi$ (with respect to the language of
$\lift[\Xseq{X}]{\seq{A}}$), denoting $\phi^+$ (resp. $\phi^-$) the formula where $M_{\seq{X}}$ is replaced by true (resp. false) it holds:
\begin{align*}
\langle\phi,\lift[\Xseq{X}]{\mathbf{A}_n}\rangle
&=\langle\phi,\lift[\Xseq{X}]{\mathbf{A}_n}-\partial_{\mathbf{A}_n}X_n\rangle+o(1)\\
&=\alpha^p\langle\phi,\lift[\Xseq{X}]{\mathbf{A}_n}[X_n]\rangle
+(1-\alpha)^p\langle\phi,\lift[\Xseq{X}]{\mathbf{A}_n}[A_n\setminus X_n\setminus \partial_{\mathbf{A}_n}X_n]\rangle
+o(1)\\
&=\alpha^p\langle\phi^+,\lift[\Xseq{X}]{\mathbf{A}_n}[X_n]\rangle
+(1-\alpha)^p\langle\phi^-,\lift[\Xseq{X}]{\mathbf{A}_n}[A_n\setminus X_n\setminus \partial_{\mathbf{A}_n}X_n]\rangle
+o(1)\\
&=\alpha^p\langle\phi^+,\mathbf{A}_n[X_n]\rangle
+(1-\alpha)^p\langle\phi^-,\mathbf{A}_n[A_n\setminus X_n\setminus \partial_{\mathbf{A}_n}X_n]\rangle
+o(1)\\
&=\alpha^p\langle\phi^+,\mathbf{A}_n[X_n]\rangle
+\langle\phi^-,\mathbf{A}_n\rangle-\alpha^p\langle\phi^-,\mathbf{A}_n[X_n]\rangle
+o(1)
\end{align*}
Hence $\lift[\Xseq{X}]{\seq{A}}$ is local convergent.
\end{proof}

\begin{definition}
Two clusters $\Xseq{X}$ and $\Xseq{Y}$ of a local-convergent
sequence $\seq{A}$ are {\em interweaving}, and we note
$\Xseq{X}\between \Xseq{Y}$ if 
every sequence $\Xseq{Z}$ with $Z_n\in\{X_n,Y_n\}$ is a cluster of 
$\seq{A}$.
\end{definition}

Interweaving clusters allow to build many new clusters by weaving (hence the name ``interweaving''). Interweaving clusters have the following handy characterization:
\begin{lemma}
Let $\Xseq{X}$ and $\Xseq{Y}$ be two clusters of a local-convergent
sequence $\seq{A}$. The following are equivalent:
\begin{enumerate}
	\item $\Xseq{X}$ and $\Xseq{Y}$ are interweaving;
	\item $\lim\lift[\Xseq{X}]{\seq{A}}=\lim\lift[\Xseq{Y}]{\seq{A}}$;
	\item either $\Xseq{X}\approx\Xseq{Y}\approx\Xseq{0}$ or
the following two conditions hold:
\begin{enumerate}
\item $\lim\seq{A}[\Xseq{X}]=\lim\seq{A}[\Xseq{Y}]$;
\item $\lim\nu_{\seq{A}}(\Xseq{X})=\lim\nu_{\seq{A}}(\Xseq{Y})$.
\end{enumerate}
\end{enumerate}
\end{lemma}
\begin{proof}
Let us prove (1)$\Rightarrow$(2)$\Rightarrow$(3)$\Rightarrow$(1).
\begin{itemize}
	\item[(1)$\Rightarrow$(2):] Let $Z_n$ be $X_n$ if $n$ is odd and $Y_n$ if $n$ is even. As $\Xseq{X}\between\Xseq{Y}$, the sequence $\Xseq{Z}$ is a cluster and
	(by considering the common subsequences) it holds
	$\lim\lift[\Xseq{X}]{\seq{A}}=\lim\lift[\Xseq{Z}]{\seq{A}}=\lim\lift[\Xseq{Y}]{\seq{A}}$.
	\item[(2)$\Rightarrow$(3):] Let $\alpha=\lim\langle M(x_1),\lift[X_n]{\mathbf{A}_n}\rangle$. Then either 
	$\alpha=0$ and $\Xseq{X}\approx\Xseq{Y}\approx\Xseq{0}$ or
	$\lim\nu_{\seq{A}}(\Xseq{X})=\lim\nu_{\seq{A}}(\Xseq{Y})\alpha>0$. In the later case,
	for any local formula $\phi$, let $\widetilde{\phi}$ be the formula where all the variables (free or bound) are constrained to belong to relation $M$. For sufficiently large $n$ (so that $\langle M(x_1),\lift[X_n]{\mathbf{A}_n}\rangle>0$ and $\langle M(x_1),\lift[Y_n]{\mathbf{A}_n}\rangle>0$) it holds
	$$
	\langle\phi,\mathbf A_n[X_n]\rangle=\frac{\langle\widetilde\phi,\lift[X_n]{\mathbf{A}_n}\rangle}{\langle M(x_1),\lift[X_n]{\mathbf{A}_n}\rangle}\text{ and }
	\langle\phi,\mathbf A_n[Y_n]\rangle=\frac{\langle\widetilde\phi,\lift[Y_n]{\mathbf{A}_n}\rangle}{\langle M(x_1),\lift[Y_n]{\mathbf{A}_n}\rangle}.
$$
	Thus $\lim\seq{A}[\Xseq{X}]=\lim\seq{A}[\Xseq{Y}]$.
	\item[(3)$\Rightarrow$(1):]
	If $\Xseq{X}\approx\Xseq{Y}\approx\Xseq{0}$ then obviously $\Xseq{X}\between\Xseq{Y}$. Otherwise, consider arbitrary $\Xseq{Z}$ with $Z_n\in\{X_n,Y_n\}$. Assume for contradiction that $\seq{A}[\Xseq{Z}$ does not converge. Then we can extract
	two subsequences $\Xseq{Z}_f$ and $\Xseq{Z}_g$ such that
	$\lim\seq{A}_f[\Xseq{Z}_f\neq \lim\seq{A}_g[\Xseq{Z}_g]$. By taking further subsequences if necessary we can assume that $\Xseq{Z}_f$ and $\Xseq{Z}_g$ are each
	either a subsequence of $\Xseq{X}$ or $\Xseq{Y}$. As 
	$\lim\seq{A}[\Xseq{X}]=\lim\seq{A}[\Xseq{Y}]$ we get a contradiction, so $\seq{A}[\Xseq{Z}$ converges. Similarly, $\nu_{\seq{A}}(\Xseq{Z})$ converges.
	As $\partial_{\mathbf A_n}(Z_n)\subseteq \partial_{\mathbf A_n}(X_n)\cup \partial_{\mathbf A_n}(Y_n)$ and
	$\partial_{\seq{A}}\Xseq{X}\approx\partial_{\seq{A}}\Xseq{Y}\approx\Xseq{0}$ we get $\partial_{\seq{A}}\Xseq{Z}\approx\Xseq{0}$. Altogether, this means that $\Xseq{Z}$ is a cluster of $\seq{A}$.
\end{itemize}
\end{proof}

Obviously, interweaving is a limit for the possibility to track clusters in a local-convergent sequence. In some sense, interweaving clusters cannot be distinguished.

We now prove that the families of all clusters of a local convergent sequence
is closed under complementation.

\begin{great}
\begin{lemma}
\label{lem:corecompl}
Let  $\seq{A}$ be a local-convergent sequence, and let $\Xseq{X}$ be a cluster of  $\seq{A}$.

Then $\Xseq{Y}=\Xseq{A}\setminus \Xseq{X}$ is a cluster of   $\seq{A}$.
\end{lemma}
\end{great}
\begin{proof}
First notice that for every integer $d$ it holds
$$\ball[d]{\seq{A}}(\partial_{\seq{A}}(\Xseq{Y}))\subseteq \ball[d+1]{\seq{A}}(\partial_{\seq{A}}(\Xseq{X}))$$
thus
$\lim\nu_{\seq{A}}(\ball[d]{\seq{A}}(\partial_{\seq{A}}\Xseq{Y}))=0$, that is $\Xseq{Y}\approx\Xseq{0}$.
As $L_{\Xseq{Y}}(\seq{A})$ can be obtained from $L_{\Xseq{X}}(\seq{A})$ by taking for $M_{\Xseq{Y}}$ the negation of $M_{\Xseq{X}}$ it is clear that $L_{\Xseq{Y}}(\seq{A})$ is local convergent.
\end{proof}

To the opposite, if $\Xseq{X}$ and $\Xseq{Y}$ are clusters, none of $\Xseq{X}\cap\Xseq{Y},
\Xseq{X}\setminus\Xseq{Y},\Xseq{Y}\setminus\Xseq{X},\Xseq{X}\,\Delta\,\Xseq{Y}, \Xseq{X}\cup\Xseq{Y}$ and their complements needs to be a cluster. For that consider the following:
\begin{example}
Consider a local-convergent sequence $\seq{E}$ of connected expanders, where $|E_n|=cn(1+o(1))$.
Define the sequence $\seq{A}$ as follows:
$$
\mathbf{A}_n=\begin{cases}
	\mathbf{E}_{5n}\cup\mathbf{E}_{6n}\cup\mathbf{E}_{8n}&\text{if $n$ is odd}\\
	\mathbf{E}_{2n}\cup\mathbf{E}_{3n}\cup\mathbf{E}_{4n}\cup\mathbf{E}_{10n}&\text{if $n$ is even}
\end{cases}
$$	
Then it is easily checked that $\seq{A}$ is local convergent, and that the only clusters of $\seq{A}$ are (up to equivalence) $\Xseq{0},\Xseq{X},\Xseq{Y},\Xseq{A\setminus X},\Xseq{A\setminus Y}$, and
$\Xseq{A}$, where
\begin{align*}
	{X}_n&=\begin{cases}
	{E}_{5n}&\text{if $n$ is odd}\\
	{E}_{2n}\cup {E}_{3n}&\text{if $n$ is even}
	\end{cases}\\
	{Y}_n&=\begin{cases}
	{E}_{6n}&\text{if $n$ is odd}\\
	{E}_{2n}\cup {E}_{4n}&\text{if $n$ is even}
	\end{cases}
\end{align*}
Also notice that the graphs $\mathbf{A}_n$ could be made connected by linking connected components using paths of lengths $\sqrt{n}$ without changing the conclusion.
\end{example}

Nevertheless, a simple necessary and sufficient condition for a family of clusters to generate a Boolean algebra of clusters can be given. 

\begin{great}
\begin{lemma}
	\label{lem:BAclust}
	Let $\seq{A}$ be a local convergent sequence, and let $\Xseq{C}^1,\dots,\Xseq{C}^i,\dots$ be countably many clusters of $\seq{A}$. Then the following conditions are equivalent:
	\begin{enumerate}
		\item\label{it:BA1} The lifted sequence $\lift{\seq{A}}$ defined by marking elements in $\Xseq{C}^i$ by mark $M_i$ is local convergent;
		\item\label{it:BA2} The clusters $\Xseq{C}^i$ generate a Boolean algebra of clusters, that is: every finite Boolean combination of $\Xseq{C}^i$'s is a cluster;
		\item\label{it:BA3} Every finite intersection of $\Xseq{C}^i$'s is a cluster.
	\end{enumerate}
\end{lemma}
 \end{great}
\begin{proof}
	We proceed by proving that \eqref{it:BA1} and \eqref{it:BA3} are  both equivalent to \eqref{it:BA2}.

	That \eqref{it:BA1}$\Rightarrow$\eqref{it:BA2} is obvious as every finite Boolean combination of $\Xseq{C}^i$'s is the set of solutions of the corresponding Boolean combination of $M_i$'s. 
	Let us now prove \eqref{it:BA2}$\Rightarrow$\eqref{it:BA1}. 
	According to Corollary~\ref{cor:sloc}, in order to prove \eqref{it:BA1} it is sufficient to prove that for every strongly local formula $\phi$ (with some $p$ free variables) the sequence $\langle\phi,\lift{\seq{A}}\rangle$ converges. Let $N$ be the maximum index of a mark symbol used in $\phi$. For $I\subseteq [N]$
	denote by $\phi^I$ the formula where every term
	$M_i(x)$ is replaced by {\tt true} if $i\in I$ and  {\tt false} if $i\notin I$. 
	Define $\theta_I$ as the formula $\bigwedge_{i\in I}M_i(x_1)\wedge\bigwedge_{i\notin I}\neg M_i(x_1)$. Let $S=\bigcup_{I\subseteq [N]}\partial_{\seq{A}}\theta_I(\seq{A})$. As each $\theta_I(\seq{A})$ defines a cluster, $S$ is negligible. Thus
	Then it holds
	\begin{align*}
		\langle\phi,\lift{\mathbf{A}_n}\rangle&=\langle\phi,\lift{\mathbf{A}_n}-S\rangle+o(1)\\
		&=\sum_{I\subseteq [N]}\nu_{\mathbf A_n}(\theta_I(\mathbf A_n))^p\,\langle \phi^I,\mathbf A_n[\theta_I(\mathbf A_n)]\rangle+o(1)
	\end{align*}
	thus $\langle\phi,\lift{\mathbf{A}_n}\rangle$ converges as $n\rightarrow\infty$. As this holds for every strongly local formula, we deduce that $\lift{\seq{A}}$ is local convergent.
	 
	That \eqref{it:BA2}$\Rightarrow$\eqref{it:BA3} is trivial. Let us now prove \eqref{it:BA3}$\Rightarrow$\eqref{it:BA2}. By following an easy induction and using the fact that the complement of a cluster is a cluster (Lemma~\ref{lem:corecompl}) we reduce easily the implication to the following statement to be proved: if $\Xseq{X}, \Xseq{Y}$, and $\Xseq{X}\cap  \Xseq{Y}$ are all clusters of $\seq{A}$ then so is $\Xseq{X}\setminus\Xseq{Y}$. To prove this, let $\Xseq{S}=\partial_{\seq{A}}\Xseq{X}\cup
	\partial_{\seq{A}}(\Xseq{X}\cap\Xseq{Y})$. Then $S$ is negligible and thus
     for every strongly local formula $\phi$ (with $p$ free variables) it holds
	\begin{align*}
		\langle\phi,\lift[X\setminus Y]{\mathbf{A}_n}\rangle&=\langle\phi^0,
		\mathbf A_,\rangle+\nu_{\mathbf A_n}(X_n)^p\, \bigl(\langle\phi^1,\mathbf A_n[X_n]\rangle-\langle\phi^0,\mathbf A_n[X_n]\rangle\bigr)\\
		&\quad+\nu_{\mathbf A_n}(X_n\cap Y_n)^p\, \bigl(\langle\phi^0,\mathbf A_n[X_n\cap Y_n]\rangle-\langle\phi^1,\mathbf A_n[X_n\cap Y_n]\rangle\bigr)+o(1),
	\end{align*}
	where $\phi^0$ (resp. $\phi^1$) stands for the formula obtained from $\phi$ by
	replacing each term of the form $M(x)$ by {\tt false} (resp. {\tt true}).
	Hence $\lift[X\setminus Y]{\mathbf{A}_n}$ is local convergent, and as 
	$\partial_{\seq{A}}(\Xseq{X}\setminus\Xseq{Y})\subseteq S$ is negligible, it follows that $\Xseq{X}\setminus\Xseq{Y}$ is a cluster.
\end{proof}
\begin{corollary}
	Let $\seq{A}$ be a local convergent sequence, and let $\Xseq{C}^1,\dots,\Xseq{C}^i,\dots$ be countably many weakly disjoint clusters of $\seq{A}$.
	
	Then the lifted sequence $\lift{\seq{A}}$ defined by marking elements in $\Xseq{C}^i$ by mark $M_i$ is local convergent.
\end{corollary}

The ultimate goal would be to extend Lemma~\ref{lem:BAclust} to the $\sigma$-algebra generated by the $\Xseq{C}^i$'s. However, we do not expect this will be always the case, and we expect that some further conditions will be required.

For instance, in order to guarantee that countable unions will be clusters, it is natural to require that there is a negligible sequence including all the possible frontiers of countable Boolean combinations. Also, we shall need some ``continuity'' property for countable Boolean combinations. The simplest form for these conditions can be given when we further assume that the clusters $\Xseq{C}^i$'s are pairwise weakly disjoint, namely:
\begin{enumerate}
	\item The sequence $\bigcup_i\partial_{\seq{A}}\Xseq{C}^i$ is negligible;
	\item The clusters $\Xseq{C}^i$ form a stable partition of $\Xseq{A}$ in the sense that $\sum_i\lim\nu_{\seq{A}}(\Xseq{C}^i)=1$.
\end{enumerate}
\begin{lemma}
	Let $\seq{A}$ be a local convergent sequence, and let $\Xseq{C}^1,\dots,\Xseq{C}^i,\dots$ be countably many weakly disjoint clusters of $\seq{A}$.
	Assume  $\bigcup_i\partial_{\seq{A}}\Xseq{C}^i$ is negligible and
	$\sum_{i}\lim \nu_{\seq{A}}(\Xseq{C}^i)=1$. Then for every $I\subseteq\bbbn$, the sequence $\bigcup_{i\in I}\Xseq{C}^i$ is a cluster. In other words, the collection of all unions of clusters among the $\Xseq{C}^i$'s forms a $\sigma$-algebra of clusters.
\end{lemma}
\begin{proof}

Let $S=\bigcup_i\partial_{\seq{A}}\Xseq{C}^i$.
Let $\phi$ be an $r$-local strongly local formula with $p$ free variables. Then for every positive real $\epsilon>0$ there exists $n_0\in\bbbn$ such that for every $n\geq n_0$ it holds $\nu_{\seq{A}}(\ball[r+1]{\seq{A}}(S))<\epsilon/8p$. 
Let $I\subseteq\bbbn$ and let $\lift[I]{\seq{A}}$ be the sequence obtained from $\seq{A}$ by marking elements of $\bigcup_{i\in I}\Xseq{C}^i$ by a new mark $M$.
Let $\psi$ be an $r$-local strongly local formula in the extended signature, and let $\psi^0$ (resp. $\psi^1$) be the formula obtained from $\psi$ by replacing each term of the form $M(x)$ by {\tt false} (resp. {\tt true}). 
According to Lemma~\ref{lem:approx1} we deduce that for every $n\geq n_0$ it holds
$$
\Bigl|\langle\psi,\lift[I]{\mathbf{A}_n}\rangle-
\sum_{i\notin I}\nu_{\mathbf A_n}(C^i_n)^p\langle\psi^0,\mathbf{A}_n[C_n^i]\rangle
-\sum_{i\in I}\nu_{\mathbf A_n}(C^i_n)^p\langle\psi^1,\mathbf{A}_n[C_n^i]\rangle\Bigr|
<\epsilon/4.
$$
As $\sum_{i\geq 1}\lim \nu_{\seq{A}}(\Xseq{C}^i)=1$ there exists $i_0\in\bbbn$ such that $\sum_{i>i_0}\lim \nu_{\seq{A}}(\Xseq{C}^i)<\epsilon/8$ thus
$$
\sum_{i\notin I, i>i_0}\bigl(\lim\nu_{\seq{A}}(\Xseq{C}^i)\bigr)^p\,\lim\langle\psi^0,\seq{A}[\Xseq{C}^i]\rangle
+\sum_{i\in I, i>i_0}\bigl(\lim\nu_{\seq{A}}(\Xseq{C}^i)\bigr)^p\,\lim\langle\psi^1,\seq{A}[\Xseq{C}^i]\rangle
<\epsilon/8.$$
Moreover, the exists $n_1\geq n_0$ such that for every $n\geq n_1$, every $1\leq i\leq i_0$ and every $k\in\{0,1\}$ it holds
$$
\Bigl|\nu_{\mathbf A_n}(C^i_n)^p\langle\psi^k,\mathbf{A}_n[C_n^i]\rangle-\bigl(\lim\nu_{\seq{A}}(\Xseq{C}^i)\bigr)^p\,\lim\langle\psi^k,\seq{A}[\Xseq{C}^i]\rangle\Bigr|<\epsilon/4i_0
$$
and
$$
\bigl|\nu_{\mathbf A_n}(C^i_n)-\lim\nu_{\seq{A}}(\Xseq{C}^i)\bigr|<\epsilon/4i_0.
$$
We deduce that

$$
\sum_{i>i_0}\nu_{\mathbf A_n}(C^i_n)=1-\sum_{i=1}^{i_0}\nu_{\mathbf A_n}(C^i_n)
<1-\sum_{i=1}^{i_0}\lim\nu_{\seq{A}}(\Xseq{C^i})+\epsilon/4=\sum_{i>i_0}\lim\nu_{\seq{A}}(\Xseq{C^i})+\epsilon/4<3\epsilon/8.
$$
From this follows that 
$$\Bigl|\langle\psi,\lift[I]{\mathbf{A}_n}\rangle-
\sum_{i\notin I}\bigl(\lim\nu_{\seq{A}}(\Xseq{C}^i)\bigr)^p\,\lim\langle\psi^0,\seq{A}[\Xseq{C}^i]\rangle
-\sum_{i\in I}\bigl(\lim\nu_{\seq{A}}(\Xseq{C}^i)\bigr)^p\,\lim\langle\psi^1,\seq{A}[\Xseq{C}^i]\rangle
\Bigr|<\epsilon.$$
It follows that $\lift[I]{\seq{A}}$ is local convergent.
As $\partial_{\seq{A}}\Bigl(\bigcup_{i\in I}\Xseq{C}^i\Bigr)\subseteq 
\bigcup_{i}\partial_{\seq{A}}\Xseq{C}^i$ is negligible by assumption, we deduce that $\bigcup_{i\in I}\Xseq{C}^i$ is a cluster.
\end{proof}

Note that when we consider the complete Boolean algebra generated by non weakly-disjoint clusters $\Xseq{C}^i$ the situation is less clear.
\subsection{Universal Clusters}

The next lemma states that the cluster of a sequence remain the same when marking a universal cluster.

\begin{lemma}
Let $\Xseq{C}$ be a universal cluster of a local convergent sequence $\seq{A}$, and let $\lift[\Xseq{C}]{\seq{A}}$ be the lift of $\seq{A}$ obtained by marking $\Xseq{C}$ by a new unary symbol $M_{\Xseq{C}}$.

Then, a sequence $\Xseq{X}$ is a cluster of $\seq{A}$ if and only if it is a cluster  of $\lift[\Xseq{C}]{\seq{A}}$.
\end{lemma}
\begin{proof}
	Of course, every cluster of $\lift[\Xseq{C}]{\seq{A}}$ is a cluster of $\seq{A}$.
	
	Assume $\Xseq{X}$ is a cluster of $\seq{A}$. Then, by definition, the sequence $\lift[\Xseq{X}]{\seq{A}}$ is a local convergent lift of $\seq{A}$. As $\seq{C}$ is universal, it is a cluster of $\lift[\Xseq{X}]{\seq{A}}$ hence the sequence
	$\lift[\Xseq{C}]{\lift[\Xseq{X}]{\seq{A}}}$ is local convergent.
	As $\lift[\Xseq{C}]{\lift[\Xseq{X}]{\seq{A}}}=\lift[\Xseq{X}]{\lift[\Xseq{C}]{\seq{A}}}$ we deduce that $\Xseq{X}$ is a cluster of $\lift[\Xseq{C}]{\seq{A}}$.
\end{proof}

Also, marking a universal cluster preserves universal clusters (but new universal cluster may appear).

\begin{remark}
Let $\Xseq{C}$ be a universal cluster of a local convergent sequence $\seq{A}$, and let $\lift[\Xseq{C}]{\seq{A}}$ be the lift of $\seq{A}$ obtained by marking $\Xseq{C}$ by a new unary symbol $M_{\Xseq{C}}$.

	Then, as every conservative lift of $\lift[\Xseq{C}]{\seq{A}}$ is a conservative lift of $\seq{A}$, it follows that every universal cluster of $\seq{A}$ is a universal cluster of $\lift[\Xseq{C}]{\seq{A}}$.
\end{remark}

The universal clusters of $\seq{A}$ are of a particular interest, as they form (as we shall prove in the next two lemmas) a Boolean algebra of clusters preserved by conservative lifts, which includes all definable clusters of $\seq{A}$.

\begin{lemma}
	Let $\seq{A}$ be a local convergent sequence and let $\phi$ be a local formula with single free variable $x_1$. 
	
	Then the following conditions are equivalent:
	\begin{enumerate}
	\item 	$\partial_{\seq{A}}\phi(\seq{A})\approx 0$;
		\item $\phi(\seq{A})$ is a cluster of $\seq{A}$;
		\item $\phi(\seq{A})$ is a universal cluster of $\seq{A}$.
	\end{enumerate}
\end{lemma}
\begin{proof}
	If $\phi(\seq{A})$ is a (universal) cluster of $\seq{A}$ then 
	$\partial_{\seq{A}}\phi(\seq{A})\approx 0$ (by definition of a cluster).
	
	Conversely, if $\partial_{\seq{A}}\phi(\seq{A})\approx 0$ then either $\phi(\seq{A})$ is negligible (thus $\phi(\seq{A})$ is a cluster) or $\seq{A}[\phi(\seq{A})]$ is local convergent: for every local formula $\psi$ with free variables $x_1,\dots,x_p$, denoting $\hat\psi$ the formula obtained by replacing terms
	$(\exists y)\theta$ by $(\exists y) (\phi(y)\wedge\theta)$ and terms $(\forall y)\theta$ by $(\forall y) (\phi(y)\rightarrow\theta)$ and denoting $\widetilde{\psi}$ the formula
	$\hat{\psi}\wedge\bigwedge_{i=1}^p\phi(x_i)$ we get
	$\widetilde{\psi}(\seq{A}\setminus\partial_{\seq{A}}\phi(\seq{A}))=\psi(\seq{A}[\phi(\seq{A})])$ hence
	$$
	\langle\psi,\mathbf{A}_n[\phi(\mathbf{A}_n)]\rangle=
	\langle\widetilde{\psi},\mathbf{A}_n\rangle+o(1).
	$$
	It follows that $\phi(\seq{A})$ is a cluster of $\seq{A}$.
	As condition (1) holds as well in every conservative lift of $\seq{A}$, it follows that $\phi(\seq{A})$ is a universal cluster of $\seq{A}$ as well.
\end{proof}
\begin{great}
\begin{lemma}
\label{lem:BAuniv}
	Let $\seq{A}$ be a local convergent sequence. Then the equivalence classes of universal clusters of $\seq{A}$ form a Boolean algebra.
\end{lemma}
\end{great}
\begin{proof}
	Let $\Xseq{X},\Xseq{Y}$ be universal clusters of $\seq{A}$, and
	let $\lift{\seq{A}}$ be a local convergent conservative lift of $\seq{A}$. Then the sequence $\lift{\lift[\Xseq{X}]{\lift[\Xseq{Y}]{\seq{A}}}}$ is local convergent. 
	If follows, by considering formulas 
	$\lnot M_{\Xseq{X}}, M_{\Xseq{X}}\vee M_{\Xseq{Y}}$, and $M_{\Xseq{X}}\wedge M_{\Xseq{Y}}$ that $\Xseq{A}\setminus\Xseq{X}$, $\Xseq{X}\cup\Xseq{Y}$ and $\Xseq{X}\cap\Xseq{Y}$ are clusters of $\lift{\lift[\Xseq{X}]{\lift[\Xseq{Y}]{\seq{A}}}}$ hence of $\lift{\seq{A}}$. It follows that
	$\Xseq{A}\setminus\Xseq{X}$, $\Xseq{X}\cup\Xseq{Y}$ and $\Xseq{X}\cap\Xseq{Y}$ are universal clusters of $\seq{A}$.
\end{proof}

\subsection{Pre-Clusters}
\begin{definition}
A sequence $\Xseq{X}\not\approx\Xseq{0}$ is a {\em pre-cluster} of
$\seq{A}$ if $\Xseq{X}\approx\Xseq{0}$ or if it holds
\begin{enumerate}
\item the sequence $\seq{A}[\Xseq{X}]$ is local-convergent;
\item the limit $\lim\nu_{\seq{A}}(\Xseq{X})$ and is strictly positive;
\item for every integer $d$ it holds
$$\limsup\nu_{\seq{A}}(\ball[d]{\seq{A}}(\Xseq{X})\setminus \Xseq{X})=0.$$
\end{enumerate}
\end{definition}

The definition of pre-clusters of $\seq{A}$ is consistent with the notion of equivalence of sequence of  subsets:

\begin{lemma}
Let $\Xseq{X}$ be a pre-cluster of
$\seq{A}$ and let $\Xseq{Y}\approx\Xseq{X}$ in $\seq{A}$.
Then $\Xseq{Y}$ is a pre-cluster of $\seq{A}$.
\end{lemma}
\begin{proof}
That $\seq{A}[\Xseq{Y}]$ is local-convergent follows from 
Lemma~\ref{lem:eqlim}. Also, it is immediate that
$\lim\nu_{\seq{A}}(\Xseq{Y})$ exists and that 
$\lim\nu_{\seq{A}}(\Xseq{Y})=\lim\nu_{\seq{A}}(\Xseq{X})$.

Let $\Xseq{Z}=\Xseq{X}\,\Delta\,\Xseq{Y}$. By assumption, $\Xseq{Z}$ is negligible in $\seq{A}$.

Assume  $\Xseq{X}$  is a pre-cluster. Let $d\in\bbbn$.
Then 
\begin{align*}
\ball[d]{\seq{A}}(\Xseq{Y})\setminus \Xseq{Y} &\subseteq
(\ball[d]{\seq{A}}(\Xseq{X})\cup \ball[d]{\seq{A}}(\Xseq{Z}))\setminus (\Xseq{X}\setminus \Xseq{Z})\\
&\subseteq (\ball[d]{\seq{A}}(\Xseq{X})\setminus \Xseq{X})\cup \ball[d]{\seq{A}}(\Xseq{Z})
\end{align*}
It follows that
$\limsup\nu_{\seq{A}}(\ball[d]{\seq{A}}(\Xseq{Y})\setminus \Xseq{Y})=0$ hence $\Xseq{Y}$ is a pre-cluster.
\end{proof}

\begin{lemma}
Every cluster is a pre-cluster. 
\end{lemma}
\begin{proof}
This follows from the fact that 
$\ball[d]{\seq{A}}(\Xseq{X})\setminus \Xseq{X}\subseteq \ball[d]{\seq{A}}(\partial_{\seq{A}_n}\Xseq{X})$.
\end{proof}

We now define a standard construction of a cluster from a pre-cluster.
\begin{definition}
Let  $\Xseq{X}$  be a pre-cluster of a local-convergent sequence $\seq{A}$.

The {\em wrapping} of $\Xseq{X}$ in $\seq{A}$ is the sequence $\Xseq{W}$ obtained as follows:

For every $n\in\bbbn$, let $D(n)\in\bbbn\cup\{\infty\}$ be the supremum of integers $d$
such that for every $n'\geq n$ it holds
$\nu_{\mathbf A_{n'}}(\ball[2d+1]{\mathbf{A}_{n'}}(X_{n'})\setminus X_{n'})<1/d$.
Then we define $W_n=\ball[D(n)]{\mathbf{A}_n}(X_n)$.
\end{definition}
Note that $D(n)$ is non-decreasing and unbounded.

\begin{lemma}
For every pre-cluster $\Xseq{X}$  of
$\seq{A}$,
 the wrapping
$\Xseq{W}$ of $\Xseq{X}$ in 
$\seq{A}$ is (up to equivalence) the unique cluster such that
$\Xseq{X}\subseteq \Xseq{W}$ and 
$$\limsup\nu_{\seq{A}}(\Xseq{W}\setminus \Xseq{X})=0.$$
\end{lemma}
\begin{proof}
For every $d\in\bbbn$ there exists $T(d)$ such that for every 
$n\geq T(d)$ it holds
$\nu_{\mathbf A_n}(\ball[2d_1]{\mathbf{A}_n}(X_n)\setminus X_n)<1/d$.
For $T(d)\leq n<T(d+1)$ we have $W_n=\ball[d]{\mathbf{A}_n}(X_n)$.
Thus, for every $d\in\bbbn$ and every  $T(d')\leq n< T(d'+1)$ (with $d'\geq d$) it holds
$$
\nu_{\mathbf A_n}(\ball[d]{\mathbf{A}_n}(\partial_{\mathbf{A}_n}W_n))
\leq \nu_{\mathbf A_n}(\ball[2d'+1]{\mathbf{A}_n}(X_n)\setminus X_n)<1/d'.
$$
Thus $\partial_{\seq{A}}\Xseq{W}$ is negligible
in $\seq{A}$ hence 
$\Xseq{W}$ is a cluster of
$\seq{A}$.

Moreover, for every $n\geq N(d)$ it holds
$\nu_{\mathbf A_n}(W_n\setminus X_n)<1/d$.

Assume that a cluster $\Xseq{Y}$ of $\seq{A}$ as the same properties. 
Then 
$$
\limsup\nu_{\seq{A}}(\Xseq{W}\,\Delta\,\Xseq{Y})\leq
\limsup\nu_{\seq{A}}(\Xseq{W}\setminus\Xseq{X})+
\limsup\nu_{\seq{A}}(\Xseq{Y}\setminus\Xseq{X})=0.
$$
Hence, according to Lemma~\ref{lem:eqcore}, $\Xseq{W}$ and $\Xseq{Y}$ are equivalent in $\seq{A}$.
\end{proof}

\subsection{Expanding Clusters}
Here we introduce a sequential version of expansion property.
\begin{definition}
A structure $\mathbf{A}$ is {\em $(d,\epsilon,\delta)$-expanding} if, for every $X\subset A$ it holds
$$
\epsilon<\nu_{\mathbf A}(X)<1-\epsilon\quad\Longrightarrow\quad
\nu_{\mathbf A}(\ball[d]{\mathbf{A}}(X))>(1+\delta)\nu_{\mathbf A}(X),
$$
that is
$$
\inf\biggl\{\frac{\nu_{\mathbf A}(\ball[d]{\mathbf{A}}[X]\setminus X)}{\nu_{\mathbf A}(X)}:\quad
\epsilon<\nu_{\mathbf A}(X)<1-\epsilon\biggr\}>\delta.
$$
Note that the left hand size of the above inequality
is similar to the {\em magnification} introduced in \cite{Alon1986}, which is the isoperimetric constant $h_{\rm out}$ defined by
$$
h_{\rm out}=\inf\left\{\frac{|\ball{\mathbf{A}}[X]\setminus X|}{|X|}:\quad 0<\frac{|X|}{|A|}<1/2\right\}.
$$
\end{definition}

\begin{lemma}
\label{lem:cleanexp}
Let $0<\epsilon<1/6$ and
	let $\mathbf A$ be a $(d,\epsilon,\delta)$-expanding structure. Then there exists a subset $Y\subseteq A$ of measure  $\nu_{\mathbf A}(Y)\leq\epsilon$ such that, denoting $\mathbf A'=\mathbf A-Y$, it holds
	$$
	\forall X\subseteq A'\qquad \nu_{\mathbf A'}(X)\leq 1/2\ \Longrightarrow \nu_{\mathbf A'}(\ball[d]{\mathbf A'}(X)\setminus X)\geq \delta\nu_{\mathbf A'}(X).
	$$
\end{lemma}
\begin{proof}
	Let $Y\subseteq A$ be maximal (for inclusion) with the property that $\nu_{\mathbf A}(Y)\leq 1-2\epsilon$ and
	$\nu_{\mathbf A}(\ball[d]{\mathbf A}(Y)\setminus Y)<\delta \nu_{\mathbf A}(Y)$. First note that $\nu_{\mathbf A}(Y)\leq \epsilon$ as $\mathbf A$ is $(d,\epsilon,\delta)$-expanding. Let $\mathbf A'=\mathbf A-Y$. 
	
	Assume for contradiction that there exists
	$Z\subset A'$ is such that $\nu_{\mathbf A'}(Z)\leq 1/2$
	and $\nu_{\mathbf A'}(\ball[d]{\mathbf A'}(Z)\setminus Z)<\delta \nu_{\mathbf A'}(Z)$. (Note that $\nu_{\mathbf A}(Z)\leq nu_{\mathbf A'}(Z)\leq 1/2$.)
	As $\nu_{\mathbf A'}$ is proportional to $\nu_{\mathbf A}$ it also holds 
	$\nu_{\mathbf A}(\ball[d]{\mathbf A'}(Z)\setminus Z)<\delta \nu_{\mathbf A}(Z)$.
	Moreover it obviously holds $\ball[d]{\mathbf A}(Y\cup Z)\subseteq\ball[d]{\mathbf A}(Y)\cup \ball[d]{\mathbf A'}(Z)$ thus 
	\begin{align*}
		\nu_{\mathbf A}(\ball[d]{\mathbf A}(Y\cup Z))&\leq 
		\nu_{\mathbf A}(\ball[d]{\mathbf A}(Y))+
		\nu_{\mathbf A}(\ball[d]{\mathbf A'}(Z))\\
		&<(1+\delta)(\nu_{\mathbf A}(Y)+\nu_{\mathbf A}(Z))
		=(1+\delta)(\nu_{\mathbf A}(Y\cup Z))
	\end{align*}
	Hence $\nu_{\mathbf A}(\ball[d]{\mathbf A}(Y\cup Z)\setminus (Y\cup Z))<\delta\nu_{\mathbf A}(Y\cup Z)$. However	$\nu_{\mathbf A}(Y\cup Z)=\nu_{\mathbf A}(Y)+\nu_{\mathbf A}(Z)\leq \epsilon+1/2<1-2\epsilon$, what contradicts the maximality of $Y$.
\end{proof}

This lemma brings us even closer to the definition of the magnification.
The main difference now stands in the existence of the parameter $d$. For graphs and $d=2$, the sequence of stars shows that the concepts differ. Actually, for graphs, $(d,\epsilon,\delta)$-expansion means that the $d$th power of the graph
(after deletion of a subset of vertices of measure at most $\epsilon$) has magnification at least $\delta$. In the very special (but standard) case of graphs with maximum degree at most $\Delta$ we recover the standard definition of expansion:
\begin{lemma}
\label{lem:expand}
	Let $0<\epsilon<1/6$ and let $G$ be a $(d,\epsilon,\delta)$-expanding graph with degree at most $\Delta$. 
	Then there $G$ has a subset $Y$ of size at most $\epsilon|G|$ such that $h_{\rm out}(G-Y)\geq \delta/(\Delta-1)^{d}$.
\end{lemma}
\begin{proof}
	We consider the uniform probability measure on $G$.
	Then the lemma follows from  Lemma~\ref{lem:cleanexp} and the simple fact that if $G$ has maximum degree at most $\Delta$ then for every subset $X$ of vertices and for every integer $k\geq 1$ it holds
	$|\ball[k+1]{G}(X)\setminus\ball[k]{G}(X)|\leq (\Delta-1)|\ball[k]{G}(X)\setminus\ball[k-1]{G}(X)|$, where
	we define $\ball[0]{G}(X)=X$. Hence
	$|\ball[d]{G}(X)\setminus X|\leq (1+\dots+(\Delta-1)^{d-1})
	|\ball{G}(X)\setminus X|$.
\end{proof}

\begin{definition}
A local-convergent sequence $\seq{A}$ is {\em expanding}
if, for every $\epsilon>0$ there exist $d,t\in\bbbn$ and $\delta>0$ such that every $\mathbf{A}_n$ with $n\geq t$ is $(d,\epsilon,\delta)$-expanding.
\end{definition}
We have the following equivalent formulations of this concept:

\begin{lemma}
\label{lem:satom}
Let $\Xseq{X}\not\approx\Xseq{0}$ be a cluster of a local convergent sequence $\seq{A}$.
The following conditions are equivalent:
\begin{enumerate}
\item the sequence  $\seq{A}[\Xseq{X}]$ is expanding;
\item for every $\epsilon>0$ there exists $d,t\in\bbbn$ such that for every $\Xseq{Z}\subseteq\Xseq{X}$ with $\nu_{\mathbf A_n}(Z_n)>\epsilon\nu_{\mathbf A_n}(X_n)$ and every $n\geq t$ it holds 
$$\nu_{\mathbf A_n}(\ball[d]{\mathbf{A}_n}(Z_n))>(1-\epsilon)\nu_{\mathbf A_n}(X_n);$$
\item the sequence $\Xseq{X}$ is a strongly atomic cluster of $\seq{A}$;
\item for every $\epsilon>0$ there exists no $\Xseq{Y}\subseteq\Xseq{X}$ such that $\partial_{\seq{A}}\Xseq{Y}\approx\Xseq{0}$ and
$$
\epsilon<\liminf \nu_{\seq{A}}(\Xseq{Y})<\lim \nu_{\seq{A}}(\Xseq{X})-\epsilon.
$$
\end{enumerate}
\end{lemma}
\begin{proof}
First assume that  $\seq{A}[\Xseq{X}]$ is expanding, and assume for contradiction that $\Xseq{X}$ is not a strongly atomic cluster of $\seq{A}$. Then there exists some increasing function $f:\bbbn\rightarrow\bbbn$ such that $\Xseq{Y}_f$ is a non-trivial cluster of $\seq{A}_f$, $\Xseq{Y}_f\subseteq \Xseq{X}_f$ and $\Xseq{Y}_f\not\approx \Xseq{X}_f$. Then $\alpha=\lim \nu_{\seq{A}_f}(\Xseq{Y}_{f})/\nu_{\seq{A}_f}(\Xseq{X}_{f})$ is bounded away from $0$ and $1$. Thus 
there exists $\delta>0$ and $d\in\bbbn$ such that
$$\liminf\frac{\nu_{\seq{A}_f}(\ball[d]{\seq{A}_{f}[\Xseq{X}_{f}]}(\Xseq{Y}_{f}))}{\nu_{\seq{A}_f}(\Xseq{Y}_{f})}>1+\delta,$$
what contradicts the property that $\Xseq{Y}_f$ is a cluster.

Now assume that $\Xseq{X}$ is a strongly atomic cluster of $\seq{A}$ and assume for contradiction that $\seq{A}[\Xseq{X}]$ is not expanding.
Then there exists $\epsilon>0$ such that for every $d\in\bbbn$ 
it holds
$$\liminf_{n\rightarrow\infty}\inf_{Y_n}\frac{\nu_{\mathbf A_{n}}(\ball[d]{\mathbf{A}_{n}[X_{n}]}(Y_{n}))}{\nu_{\mathbf A_{n}}(Y_{n})}=1,$$
where infimum is on subsets $Y_n\subset X_n$ with
$\epsilon<\nu_{\mathbf A_{n}}(Y_n)/\nu_{\mathbf A_{n}}(X_n)<1-\epsilon$.
We inductively construct an increasing function $f:\bbbn\rightarrow\bbbn$ and subsets $Y_{f(n)}\subset X_{f(n)}$ as follows:
$f(1)$ is the minimum integer $n$ such that there exists
$Y_{n}\subset X_{n}$ with $\epsilon<\nu_{\mathbf A_{n}}(Y_{n})/\nu_{\mathbf A_{n}}(X_{n})<1-\epsilon$ and $\nu_{\mathbf A_{n}}(\ball{\mathbf{A}_{n}[X_{n}]}(Y_{n}))<2 \nu_{\mathbf A_{n}}(Y_{n})$ and
(for $d\geq 1$) $f(d+1)$ is the minimum integer $n>f(d)$ such that there exists
$Y_{n}\subset X_{n}$ with $\epsilon<\nu_{\mathbf A_{n}}(Y_{n})/\nu_{\mathbf A_{n}}(X_{n})<1-\epsilon$ and $\nu_{\mathbf A_{n}}(\ball[d+1]{\mathbf{A}_{n}[X_{n}]}(Y_{n}))<\frac{d+2}{d+1} \nu_{\mathbf A_{n}}(Y_{n})$.
It is easily checked that $(Y_{f(n)})$ is such that for every integer $d$ it holds
$$
\limsup\nu_{\seq{A}_f}(\ball[d]{\seq{A}_f}(\Xseq{Y}_f)\setminus \Xseq{Y}_f)=0.
$$
We can further consider a subsequence $\Xseq{Y}_{f\circ g}$ of
$\Xseq{Y}_f$ such that $\seq{A}_{f\circ g}[\Xseq{Y}_{f\circ g}]$ is local convergent and $\nu_{\seq{A}_{f\circ g}}(\Xseq{Y}_{f\circ g})$ converges. It follows that
$\Xseq{Y}_{f\circ g}$ is a pre-cluster. Let 
$(\Xseq{\hat Y}_{f\circ g}$ be the wrapping of $\Xseq{Y}_{f\circ g}$ in
$\seq{A}_{f\circ g})$. Then $\Xseq{\hat Y}_{f\circ g}$ is a cluster, 
$\Xseq{\hat Y}_{f\circ g}\preceq \Xseq{X}_{f\circ g}$ and 
$\Xseq{\hat Y}_{f\circ g}\not\approx \Xseq{X}_{f\circ g}$, what contradicts the assumption that $\Xseq{X}$ is a strongly atomic cluster.
\end{proof}

A stronger form of expanding property is the non-dispersive property.

\begin{definition}
A local-convergent sequence $\seq{A}$ is {\em non-dispersive} if, for every $\epsilon>0$ there exists $d\in\bbbn$ such that
$$\adjustlimits\liminf_{n\rightarrow\infty}\sup_{v_n\in A_n}
\nu_{\mathbf A_n}(\ball[d]{\mathbf A_n}(v_n))>1-\epsilon.$$
\end{definition}

In other words, a sequence $\seq{A}$ is non-dispersive if, for every $\epsilon>0$, $\epsilon$-almost all elements of $\mathbf A_n$ are included in some ball of radius at most $d$, for some fixed $d$.

\begin{definition}
A non-trivial cluster $\Xseq{X}$ of $\seq{A}$ is a 
{\em globular cluster} 
of $\seq{A}$ if
$\seq{A}[\Xseq{X}]$ is non-dispersive.
\end{definition}

Every globular cluster is clearly strongly atomic, but the converse does not hold as witnessed, for instance, by sequence of expanders. The strongly atomic clusters that are not globular are called {\em open clusters}.

Opposite to globular clusters are residual clusters:
\begin{definition}
	A cluster $\Xseq{X}$ of $\seq{A}$ is {\em residual} if 
for every $d\in\bbbn$ it holds
$$
\adjustlimits\limsup_{n\rightarrow\infty}\sup_{v_n\in A_n}
\nu_{\mathbf A_n}(\ball[d]{\mathbf A_n}(v_n))=0.
$$
\end{definition}

The case of bounded degree graphs is particularly interesting.
Recall that a sequence $\seq{G}$ of graphs is a {\em vertex expander} if there exists  $\alpha>0$ such that 
$$
\liminf h_{\rm out}(G_n)\geq\alpha.
$$
\begin{lemma}
	Let $\seq{G}$ be a sequence of graphs with maximum degree at most $\Delta$ and let $\Xseq{C}\not\approx\Xseq{0}$ be a cluster of $\seq{G}$. The following are equivalent:
	\begin{itemize}
		\item $\Xseq{C}$ is a strongly atomic cluster;
		\item for every $\epsilon>0$ there exists $\Xseq{X}\subseteq \Xseq{C}$ such that for every $n\in\bbbn$ it holds $|X_n|<\epsilon |C_n|$ and $\seq{G}[\Xseq{C}\setminus \Xseq{X}]$ is a vertex expander.
	\end{itemize}
\end{lemma}
\begin{proof}
	This is a direct consequence of Lemma~\ref{lem:expand}.
\end{proof}

\begin{lemma}
\label{lem:expi0}
Let $\Xseq{X}$  be an expanding cluster of $\seq{A}$, and let 
$\Xseq{Y}$ be any cluster of $\seq{A}$.

Then any convergent subsequence of
$\left(\frac{\nu_{\seq{A}}(\Xseq{X}\cap\Xseq{Y})}{\nu_{\seq{A}}(\Xseq{X})}\right)$ has limit either $0$ or $1$.
\end{lemma}
\begin{proof}
Let $\Xseq{Z}=\Xseq{X}\cap\Xseq{Y}$.
Assume there exists an increasing function $f:\bbbn\rightarrow\bbbn$ and a positive real $\alpha\in (0,1)$ such that
$\lim\nu_{\seq{A}_f}(\Xseq{Z}_f)/\nu_{\seq{A}_f}(\Xseq{X}_f)=\alpha$.

According to Fact~\ref{fact:bord}, for every integer $d\in\bbbn$ it holds
$$\ball[d]{\mathbf{A}_{f(n)}}(Z_{f(n)})\subseteq
\ball[d+1]{\mathbf{A}_{f(n)}}(X_{f(n)})\cup \ball[d+1]{\mathbf{A}_{f(n)}}(Y_{f(n)}).$$
It follows that $\partial_{\seq{A}_f}\Xseq{Z}_f$ is negligible in $\seq{A}_f$. 

By compactness, there exists an increasing function $g:\bbbn\rightarrow\bbbn$ such that $(\seq{A}[Z])_{g\circ f}$ is local convergent. It follows that $\Xseq{Z}_{g\circ f}$ is a cluster 
of $\seq{A}_{g\circ f}$ which is neither equivalent to $\seq{0}$ nor to $\Xseq{A}_{g\circ f}$, thus $\Xseq{X}$ is not strongly atomic, what contradicts the hypothesis, according to Lemma~\ref{lem:satom}.
\end{proof}

\begin{lemma}
\label{lem:expmix}
Let $\Xseq{X}$ and $\Xseq{Y}$ be expanding clusters of $\seq{A}$.

Then 
\begin{itemize}
\item either $\Xseq{X}\cap\Xseq{Y}\approx\Xseq{0}$ (i.e. $\Xseq{X}$ and $\Xseq{Y}$ are {\em essentially disjoint}); 
\item or $\Xseq{X}\between\Xseq{Y}$ and then every $\Xseq{Z}$ with $Z_n\in\{X_n,Y_n\}$ is an expanding cluster of $\seq{A}$.
\end{itemize}
\end{lemma}
\begin{proof}
If $\nu_{\mathbf A_n}(X_n\cap Y_n)=o(1)$ then $\Xseq{X}$ and $\Xseq{Y}$ are essentially disjoint (as $\partial_{\seq{A}}(\Xseq{X}\cap\Xseq{Y})$ is negligible in $\seq{A}$ (see proof of Lemma~\ref{lem:expi0}).

Otherwise, there exists, according to Lemma~\ref{lem:expi0}, an increasing function $f:\bbbn\rightarrow\bbbn$ such that
$\Xseq{X}_f\approx \Xseq{Y}_f$. It follows that $\seq{A}_f[\Xseq{X}_f]$ and $\seq{A}_f[\Xseq{Y}_f]$ (thus $\seq{A}[\Xseq{X}]$ and $\seq{A}[\Xseq{Y}]$) have the same local limit. Let $\Xseq{Z}$ be such that $Z_n\in\{X_n,Y_n\}$. Then $\seq{A}[\Xseq{Z}]$ is local-convergent, $\partial_{\seq{A}}\Xseq{Z}$ is negligible in $\seq{A}$, and $\lim\nu_{\seq{A}}(\Xseq{Z})$ exists (and 
$\lim\nu_{\seq{A}}(\Xseq{Z})=\lim\nu_{\seq{A}}(\Xseq{X})=\lim\nu_{\seq{A}}(\Xseq{Y})$). It follows that $\Xseq{Z}$ is a non-trivial cluster.
Thus $\Xseq{X}$ and $\Xseq{Y}$ are interweaving (i.e. $\Xseq{X}\between\Xseq{Y}$).
 That $\Xseq{Z}$ is strongly atomic (hence expanding) follows from the hypothesis that both $\Xseq{X}$ and $\Xseq{Y}$ are expanding (hence strongly atomic): any cluster included in a subsequence of $\Xseq{Z}$ has a subsequence which is a cluster included in a subsequence of $\Xseq{X}$ or in a subsequence of $\Xseq{Y}$.
\end{proof}

It is possible that a local-convergent sequence $\seq{A}$ has arbitrarily many pairwise intersecting non equivalent expanding clusters but not two essentially disjoint ones: 
\begin{example}
Consider a local-convergent sequence $\seq{E}$ of connected $d$-regular high-girth expanders with $|E_n|=cn(1+o(1))$ (and uniform probability measure), for some constant $c>1$. Let $\mathbf{A}_n$ be defined as three copies of $\mathbf{E}_n$ if $n$ is odd, and the union of $\mathbf{E}_n$ and
$\mathbf{E}_{2n}$ if $n$ is even. Selecting a copy of $\mathbf{E}_n$ into each $\mathbf{A}_n$ leads to uncountably many pairwise intersecting non-equivalent expanding clusters. However, no two essentially disjoint expanding clusters exist in $\seq{A}$.
Note that we could have made $\mathbf{A}_n$ connected by adding paths of length $\sqrt{n}$ linking the connected components without changing the conclusion.
	\end{example}

\section{Clustering and the Cluster Comb Lemma}
\label{sec:comb}
The notion of clustering intuitively covers the idea of partitioning the structures in a local convergent sequence as well as the limit into disjoint clusters.

\begin{definition}
	Let $\seq{A}$ be a local-convergent sequence of $\sigma$-structures.
	A lifted sequence $\lift{\seq{A}}$ of $\seq{A}$ obtained by extending the signature $\sigma$ into $\sigma^+$ by adding countably many unary relations $M_1,M_2,\dots$ is a {\em clustering} of $\seq{A}$ if, denoting
	$$\Xseq{S}=\Xseq{A}\setminus\bigcup M_i(\seq{A})$$
	 the following conditions hold:
	\begin{enumerate}
		\item\label{cl:1} The sequence $\lift{\seq{A}}$ is local convergent;
		\item\label{cl:2} The sequence $\Xseq{S}$ is negligible and  $\bigcup_i\partial_{\seq{A}}M_i(\seq{A})\subseteq \Xseq{S}$ ;
		\item\label{cl:3} For every $n\in\bbbn$, the non empty sets among $S_n$,
	$M_1(\mathbf A_n), M_2(\mathbf A_n), \dots$  form a partition of $A_n$;
		\item\label{cl:4} The partition induced by the $M_i$'s is stable in the sense that
$$
\sum_i\lim\langle M_i,\seq{A}\rangle=\lim \sum_i\langle M_i,\seq{A}\rangle=1.
$$
\end{enumerate}
\end{definition}

\begin{definition}
	We say that two clusters $\Xseq{C_1}$ and $\Xseq{C_2}$ are 
	\begin{itemize}
		\item {\em weakly disjoint} if $\Xseq{C_1}\Delta\Xseq{C_2}\approx\Xseq{0}$;
		\item {\em disjoint} if $\Xseq{C_1}\Delta\Xseq{C_2}=\Xseq{0}$;
		\item {\em strongly disjoint} if $(\ball{\seq{A}}(\Xseq{C_1})\cap\Xseq{C_2})\cup(\Xseq{C_1}\cap\ball{\seq{A}}(\Xseq{C_2}))=\Xseq{0}$.
	\end{itemize} 
\end{definition}

\begin{remark}
\label{rem:clust}
Conditions \eqref{cl:1} and \eqref{cl:2} imply that each sequence $M_i(\seq{A})$ is a cluster of $\seq{A}$ hence a clustering of $\seq{A}$ induce a ``partition'' into countably many disjoint clusters, and that the clusters defined by the marks $M_i$ are pairwise strongly disjoint.
\end{remark}

A simple idea to construct a clustering of a local convergent sequence $\seq{A}$ is as follows:
assume $\seq{A}$ has a cluster $\Xseq{X_1}\not\approx\Xseq{0}$. Then let $\seq{A}_1$ be the lift of $\seq{A}$ with $\Xseq{X_1}$ marked $M_1$. Then look for a cluster $\Xseq{X_2}\not\approx\Xseq{0}$ of $\seq{A}_1$ disjoint from $M_1$ and mark it $M_2$, thus obtaining $\seq{A}_2$. Repeat the process until no cluster can be found.
There are two main problems with this process:
\begin{itemize}
	\item In general we do not obtain a clustering, as the obtained partition needs not to be stable and the global outer boundary $\bigcup_i\partial_{\seq{A}}M_i(\seq{A})$ needs not to be negligible;
	\item The partition is essentially not unique (and it is not clear which clusters of $\seq{A}$ may appear simultaneously in the partition).
\end{itemize}

The first point is exemplified by the fact that we do not have the converse of Remark~\ref{rem:clust} does not holds in general: partitioning into disjoint clusters do not define a clustering in general.

For instance, consider the following sequence of star forests.
\begin{example}
Consider  the sequence $\seq{G}$ where $\mathbf G_n$ is the union of $2^n$ stars
$\mathbf H_{n,1},\dots,\mathbf H_{n,2^n}$, where the $i$-th star $\mathbf H_{n,i}$
has order $2^{2^n}(2^{-i}+2^{-n})/2$. 
 Let $\Xseq{C}^i$ be the sequence such that $C^i_n$ is the vertex set $H_{n,i}$ of the $i$th biggest connected component of $\mathbf G_n$ (or the empty subset if $i>2^n$). It is easily checked that each $\Xseq{C}^i$ is a cluster and that for each $n$ the (non-empty) subsets $C_n^i$ form a partition of $G_n$. Assume that we mark each $C_n^i$ by mark $M_i$. Then, asymptotically, only one half of the vertices will be marked.
\end{example}

Nevertheless, we shall prove that the converse of Remark~\ref{rem:clust} is almost true.
In order to do so, 
we consider countably many disjoint clusters 
 $\Xseq{C}^1,\dots,\Xseq{C}^i,\dots$ 
of a local convergent sequence $\seq{A}$. For each $i\in\bbbn$ we define
\begin{align*}
\lambda_i&=\lim\nu_{\seq{A}}(\Xseq{C}^i)
\intertext{and}
\lambda_0&=1-\sum_{i\geq 1}\lambda_i.
\end{align*}

The next lemma shows how powerful the stability assumption~\eqref{cl:4} can be:

\begin{lemma}
\label{lem:stable}
	Assume that there exists negligible $\Xseq{S}\supseteq\bigcup_i\partial_{\seq{A}}\Xseq{C}^i$ and that it holds
	\begin{equation}
	\label{eq:stable}
		\sum_i\lim\nu_{\seq{A}}(\Xseq{C}^i)=\lim\sum_i\nu_{\seq{A}}(\Xseq{C}^i).
	\end{equation}
	Then $\Xseq{R}=\Xseq{A}\setminus\Xseq{S}\setminus\bigcup_i\Xseq{C}^i$
	is a cluster, and the lifted sequence $\lift{\seq{A}}$ obtained by marking 
	$\Xseq{R},\Xseq{C}^1,\Xseq{C^2},\dots$ by (say) marks $M_0,M_1,M_2,\dots$ is a clustering of $\seq{A}$.
\end{lemma}
\begin{proof}
		First note that~\eqref{eq:stable} easily implies that $\nu_{\seq{A}}(\Xseq{C}^i)$ converges to $(\lambda_i)_{i\in\bbbn}$ in $\ell^p$-norm for $p\geq 1$.
	Let $\phi_1,\phi_2,\dots$ be  strongly $r$-local formulas with $p$ free variables in the language of $\sigma$. Then for any fixed $N\in\bbbn$ it holds
		\begin{align*}
		&\left|\biggl(\sum_{i\geq 1}\nu_{\mathbf A_n}(C_n^i)^p\langle\phi^i,\mathbf A_n[C_n^i]\rangle\biggr)-\biggl(\sum_{i\geq 1}\lambda_i^p\lim\langle\phi^i,\seq{A}[\Xseq{C}^i]\rangle\biggr)\right|\leq\\
		&\qquad\qquad\qquad \sum_{i\geq 1}|\nu_{\mathbf A_n}(C_n^i)^p-\lambda_i^p|+\sum_{i\geq 1}\lambda_i^p\bigl|\lim\langle\phi^i,\seq{A}[\Xseq{C}^i]\rangle-\langle\phi^i,\mathbf A_n[C_n^i]\rangle\bigr|\\
		&\qquad\qquad\|\nu_{\seq{A}}(\Xseq{C}^i)-(\lambda_i)_{i\in\bbbn}\|_p+
		\sum_{i=1}^N \bigl|\lim\langle\phi^i,\seq{A}[\Xseq{C}^i]\rangle-\langle\phi^i,\mathbf A_n[C_n^i]\rangle\bigr|+\sum_{i>N}\lambda_i^p.\\		
	\end{align*}
	It follows that
	\begin{equation}
	\label{eq:convs}
	\lim_{n\rightarrow\infty}\sum_{i\geq 1}\nu_{\mathbf A_n}(C_n^i)^p\langle\phi^i,\mathbf A_n[C_n^i]\rangle=\sum_{i\geq 1}\lambda_i^p\lim\langle\phi^i,\seq{A}[\Xseq{C}^i]\rangle.	
	\end{equation}

	Let $\psi$ be a strongly $r$-local formula with $p$ free variables in the language of $\sigma^+=\sigma\cup\{M_0,M_1,M_2,\dots\}$.
	For $\zeta$ non negative integer, let $\psi^\zeta$ be the formula obtained from $\psi$ by replacing each term $M_i(t)$ by {\tt true} if $i=\zeta$ and 
	{\em false} otherwise.
	According to Lemma~\ref{lem:approx1} it holds
\begin{align*}
\langle\psi,\lift{\mathbf A_n}\rangle&=\nu_{\mathbf A_n}(R_n)^p\langle\psi^0,\mathbf A_n[R_n]\rangle+\sum_{i\geq 1}\nu_{\mathbf A_n}(C_n^i)^p\langle\psi^i,\mathbf A_n[C_n^i]\rangle,\\
\langle\psi^0,\mathbf A_n\rangle&=\nu_{\mathbf A_n}(R_n)^p\langle\psi^0,\mathbf A_n[R_n]\rangle+\sum_{i\geq 1}\nu_{\mathbf A_n}(C_n^i)^p\langle\psi^0,\mathbf A_n[C_n^i]\rangle.
\end{align*}
Thus, according to \eqref{eq:convs} it holds
$$
	\lim_{n\rightarrow\infty}\sum_{i\geq 1}\nu_{\mathbf A_n}(C_n^i)^p\langle\psi^0,\mathbf A_n[C_n^i]\rangle=\sum_{i\geq 1}\lambda_i^p\lim\langle\psi^0,\seq{A}[\Xseq{C}^i]\rangle.	
$$
Hence $\lim_{n\rightarrow\infty} \nu_{\mathbf A_n}(R_n)^p\langle\psi^0,\mathbf A_n[R_n]\rangle$ exists and
$$
\lim_{n\rightarrow\infty} \nu_{\mathbf A_n}(R_n)^p\langle\psi^0,\mathbf A_n[R_n]\rangle=\lim \langle\psi^0,\seq{A}\rangle-\sum_{i\geq 1}\lambda_i^p\lim\langle\psi^0,\seq{A}[\Xseq{C}^i]\rangle.
$$
It follows that $\lim \langle\psi,\lift{\seq{A}}\rangle$ exists and
$$
\lim \langle\psi,\lift{\seq{A}}\rangle=\lim \langle\psi^0,\seq{A}\rangle-\sum_{i\geq 1}\lambda_i^p\lim\langle\psi^0,\seq{A}[\Xseq{C}^i]\rangle+\sum_{i\geq 1}\lambda_i^p\lim\langle\psi^i,\seq{A}[\Xseq{C}^i].
$$
Hence $\lift{\seq{A}}$ is a clustering of $\seq{A}$.
\end{proof}

To handle cases where~\eqref{eq:stable} does not hold, we need to introduce the notion of clip:

\begin{definition}
A {\em clip} us a non-decreasing function $F:\bbbn\rightarrow\bbbz^+$ such that $F\gg 1$ (i.e.
$\lim_{n\rightarrow\infty}F(n)=\infty$) and such that for every integers $n\leq n'$ it holds
\begin{equation}
\sum_{i=1}^{F(n)}\bigl|\nu_{\mathbf A_{n'}}(C_{n'}^i)-\lambda_i\bigr|\leq\sum_{i>F(n)}\lambda_i.
\end{equation}
\end{definition}
First, a few remarks are in order. The function $F:\bbbn\rightarrow\bbbz^+$ defined by
$$
F(n)=\min\biggl(n,\max\Bigl\{t\leq n:\ \forall n'\geq n\ 
\sum_{i=1}^t\bigl|\nu_{\mathbf A_{n'}}(C_{n'}^i)-\lambda_i\bigr|\leq\sum_{i>t}\lambda_i\Bigr\}\biggr)
$$
is a clip, as for $t=0$ the inequality holds and as for every $k\in\bbbn$ there exists $n\in\bbbn$ such that $F(n)\geq\min(n,k)$ (as the left-hand side of the inequality tends to $0$ as $n'\rightarrow\infty$). Thus clips always exist.

Secondly, remark that if $1\ll G\leq F$ and $F$ is a clip then $G$ is a clip as well, as 
$$
\sum_{i=1}^{G(n)}\bigl|\nu_{\mathbf A_{n'}}(C_{n'}^i)-\lambda_i\bigr|
\leq
\sum_{i=1}^{F(n)}\bigl|\nu_{\mathbf A_{n'}}(C_{n'}^i)-\lambda_i\bigr|\leq\sum_{i>F(n)}\lambda_i\leq\sum_{i>G(n)}\lambda_i.
$$
\begin{lemma}
\label{lem:clip1}
Let $F$ be a clip. Then
$$
\lim_{n\rightarrow\infty}\sum_{i=1}^{F(n)}\nu_{\mathbf A_n}(C_n^i)=1-\lambda_0.
$$
\end{lemma}
\begin{proof}
It follows directly from the definition of a clip that for every integer $n$ it holds
$$\sum_{i=1}^{F(n)}\lambda_i-\sum_{i>F(n)}\lambda_i\leq \sum_{i=1}^{F(n)} \nu_{\mathbf A_n}(C_n^i)\leq \sum_{i=1}^{F(n)}\lambda_i+\sum_{i>F(n)}\lambda_i.$$
A $\lim_{n\rightarrow\infty}\sum_{i\geq F(n)}\lambda_i=0$, we deduce that 
$$\lim_{n\rightarrow\infty}\sum_{i=1}^{F(n)}\nu_{\mathbf A_n}(C_n^i)=\sum_{i\geq 1}\lambda_i=1-\lambda_0.$$
\end{proof}
Given a clip $F$, we define  $\Xseq{R}$ by
$$
R_n=A_n\setminus\bigcup_{i=1}^{F(n)}C_n^i.
$$
As for every integers $i$ and $d$ it holds
$$
\lim\nu_{\seq{A}}(\ball[d]{\seq{A}}(\partial_{\seq{A}} \Xseq{C^i}))=0
$$
there exists a function $T:\bbbn\times\bbbn\rightarrow\bbbn$ such that for every integers $i,d$ and $n\geq T(i,d)$ it holds
$$
\nu_{\mathbf A_n}(\ball[d]{\mathbf A_n}(\partial_{\mathbf A_n} C_n^i))\leq\frac{2^{-i}}{d}.
$$
Define
$$
M(a)=\max_{1\leq i\leq a}\max_{1\leq d\leq a}T(i,j).
$$
 Define also $G:\bbbn\rightarrow\bbbz^+$ by
$$
G(n)=\min(F(n),\max\{i\in\bbbn:\ M(i)\leq n\}). 
$$
Obviously, $1\ll G\leq F$ thus $G$ is a clip.
This clip has the following property:
\begin{lemma}
The sequence $\Xseq{S}$ defined by
$$
S_n=\bigcup_{i=1}^{G(n)}\partial_{\mathbf A_n} C_n^i
$$
is negligible.
\end{lemma}
\begin{proof}
Let $d\in\bbbn$. For sufficiently large $n$ it holds
\begin{align*}
\nu_{\mathbf A_n}(\ball[d]{\mathbf A_n}(S_n))
&\leq\nu_{\mathbf A_n}(\ball[G(n)]{\mathbf A_n}(S_n))\\
&\leq\sum_{i=1}^{G(n)}\nu_{\mathbf A_n}(\ball[G(n)]{\mathbf A_n}(\partial C_n^i))\\
&\leq\sum_{i=1}^{G(n)}\frac{2^{-i}}{G(n)}<\frac{1}{G(n)}.
\intertext{Hence}
\lim\nu_{\seq{A}}(\ball[d]{\seq{A}}(\Xseq{S}))&=0,
\end{align*}
that is: $\Xseq{S}$ is negligible.
\end{proof}

Define the subset sequences $\Xseq{D}^i$ by
$$
D^i_n=\begin{cases}
C^i_n&\text{if }n\geq G(i)\\
\emptyset&\text{otherwise}
\end{cases}
$$ 
and let $\Xseq{R}$ be defined by $R_n=A_n\setminus\bigcup_{i\geq 1} D_n^i$.
\begin{lemma}
Either $\lambda_0=0$ and $\Xseq{R}$ is negligible, or
$\lambda_0>0$ and $\Xseq{R}$ is a cluster.
\end{lemma}
\begin{proof}
Note that 
$$\Xseq{S}=\bigcup_{i\geq 1}\partial\Xseq{D}^i\supseteq \partial\Xseq{R}.$$
In particular, $\partial\Xseq{R}$ is negligible. According to Lemma~\ref{lem:clip1}, we have
$$
\lim_{n\rightarrow\infty}\sum_{i=1}^{G(n)}\nu_{\mathbf A_n}(D_n^i)=1-\lambda_0,
$$
thus
$\lim\nu_{\seq{A}}(\Xseq{R})=\lambda_0$.
Consider a strongly $r$-local formula $\phi$ with $p$ free variable.
For every $\epsilon>0$ there exists $n_0$ such that for every $n\geq n_0$ it holds $\nu_{\mathbf A_n}(\ball[r]{\mathbf A_n}(S_n))<\epsilon$. It follows that
$$
\Bigl|\langle\phi,\mathbf{A}_n\rangle-\nu_{\mathbf{A}_n}(R_n)^p \langle\phi,\mathbf{A}_n[R_n]\rangle-\sum_{i\geq 1}
\nu_{\mathbf A_n}(D_n^i)^p\langle\phi,\mathbf{A}_n[D_n^i]\rangle\Bigr|<3p\epsilon.
$$
Thus, if $\lambda_0>0$ we have
\begin{align*}
\lim\langle\phi,\seq{A}[\Xseq{R}]\rangle&=
\frac{1}{\lambda_0^p}\lim_{n\rightarrow\infty}\Bigl(
\langle\phi,\mathbf{A}_n\rangle-
\sum_{i\geq 1}\lambda_i^p \langle\phi,\mathbf{A}_n[{D}^i_n]\rangle
\Bigr)\\
&=\frac{1}{\lambda_0^p}\Bigl(
\lim\langle\phi,\seq{A}\rangle-
\sum_{i\geq 1}\lambda_i^p \lim\langle\phi,\seq{A}[\Xseq{D}^i]\rangle
\Bigr).
\end{align*}
(Note that we can safely exchange limit and sum here because the partition into 
$\Xseq{R}$ and the $\Xseq{D}^i$'s is stable, see Lemma~\ref{lem:stable}.)
It follows that either $\lambda_0$ and $\Xseq{R}$ is negligible, or
$\lambda_0>0$ and $\Xseq{R}$ is a cluster.
\end{proof}
\begin{great}
\begin{lemma}[Cluster Comb Lemma]
\label{lem:comb}
Let $\seq{A}$ be a local convergent sequence of $\sigma$-structures, and let $\Xseq{C}^1,\dots,\Xseq{C}^i,\dots$ be countably many strongly disjoint clusters of $\seq{A}$.

Let $\sigma^+$ be the signature $\sigma$ augmented by unary relations $M_i$ ($i\in\{0,1,2,\dots\}$). Then there exist a clustering $\seq{A}^+$ of $\seq{A}$ with the property
that for $i=1,\dots$, the marks $M_i$ comb the clusters $\Xseq{C}^i$  in the sense that there exists a non decreasing function $G:\bbbn\rightarrow\bbbn$ with $g\gg 1$ with
\begin{equation}
M_i(\mathbf A_n)=\begin{cases}
C_n^i&\text{if }n\geq G(i)\\
\emptyset&\text{otherwise}
\end{cases}
\end{equation}
(In particular $M_i(\seq{A})\approx\Xseq{C}^i$).
\end{lemma}	
\end{great}

\begin{proof}
Denote $\Xseq{S}=\Xseq{A}\setminus\bigcup_i M_i(\seq{A})$, $\Xseq{D}^i=M_i(\seq{A})$ and $\Xseq{R}=M_0(\seq{A})$.

	Remark that we have the property that $\seq{A}\approx \seq{A}-\Xseq{S}$, which is the disjoint union of $\seq{A}[R]$ and all the
$\seq{A}[\Xseq{D}^i]$. Mark vertices of $\Xseq{R}$ by $M_0$,
vertices of $\Xseq{D}^i$ by $M_i$, and further mark vertices in $N_n$ by mark $M_S$.
It is easily checked that the proportion of $A_n$ marked by some 
mark $M_i$ for $0\leq i\leq k$ tends to $\sum_{i=0}^k\lambda_i$ as $n\rightarrow\infty$, and that this value tends to $1$ as $k\rightarrow\infty$. Consider the signature $\sigma^+$ extended by these marks, and let $\seq{A}^+$ be the  sequence of marked structures. Let 
$\phi$ be an $r$-local strongly local formula with $p$ free variables. Then
$$
\langle\phi,\mathbf{A}_n^+\rangle=\lambda_0^p\langle\phi,\mathbf{A}_n^+[R_n]\rangle+\sum_{i\geq 1}\lambda_i^p\langle\phi,\mathbf{A}_n^+[D_n^i]\rangle.
$$
Denote by $\phi_i$ the formula derived from $\phi$ by replacing each $M_i$ with {\tt true} and every $M_j$ with $j\neq i$ with {\rm false}.
Notice that $\phi_i$ is an $r$-local strongly local formula in the language of the original signature $\sigma$,  that
$\phi(\mathbf{A}_n^+[R_n])=\phi_0(\mathbf{A}_n[R_n])$ and that 
$\phi(\mathbf{A}_n^+[D_n^i])=\phi_i(\mathbf{A}_n[D_n^i])$ (for $i\in\bbbn$). Hence
$$
\langle\phi,\mathbf{A}_n^+\rangle=\lambda_0^p\langle\phi_0,\mathbf{A}_n[R_n]\rangle+\sum_{i\geq 1}\lambda_i^p\langle\phi_i,\mathbf{A}_n[D_n^i]\rangle.
$$
Thus $\seq{A}^+$ is a local convergent sequence.
\end{proof}

Remark: if one only assumes that the clusters $\Xseq{C}^i$ are almost disjoint (meaning $\Xseq{C}^i\Delta\Xseq{C}^j$ negligible if $i\neq j$) then we get the same conclusion, except that the second item is weakened to $\Xseq{D}^i\approx\Xseq{C}^i$. The idea is to define the clusters $\Xseq{Z}^i=\Xseq{C}^i\setminus\bigcup_{j<i}\ball{\seq{A}}(\Xseq{C}^j)$ that are strongly disjoints and equivalent to the original clusters.

\section{The Clustering Problem}
It is not clear which clusters of $\seq{A}$ can be ``captured'' in general from the only information available from local convergence, and whether it is possible to mark these  clusters in a constructive way.

The answer to this question is that we can always capture all the (countably many) globular clusters and that we can explicitly define the marking based on the knowledge of some of the limit Stone pairing and basic Fourier analysis, and a subtle cut method to handle the non-commutativity of countable sums and limits.
 This will be the motivation of the final part of this paper. This part demonstrates pleasing mathematical paradox: in order to achieve a more concrete result we first have to generalize.
 
\part{Effective Construction of the Globular Clusters}
\section{The Representation Theorem and Some Consequences}
\label{sec:rep}

Let $\mathcal B$ be the Lindenbaum--Tarski algebra defined by ${\rm FO}^{\rm local}(\sigma)$ and let $S_\sigma$ be the Stone dual of $\mathcal B$, which is a Polish space, whose topology is generated by its clopen subsets. Recall that the duality of $\mathcal B$ and $S_\sigma$ is expressed by the existence of a mapping $K$ from ${\rm FO}^{\rm local}(\sigma)$ to the family of all the clopen subsets of $S$ such that $K(\phi\vee\psi)=K(\phi)\cup K(\psi)$, $K(\phi\wedge\psi)=K(\phi)\cap K(\psi)$, $K(\neg\phi)=S\setminus K(\phi)$, and $K(\phi)=K(\psi)$ if and only if $\phi$ and $\psi$ are logically equivalent. For a local formula $\phi$, we further denote by $k(\phi)$ the indicator function of $K(\phi)$, which is obviously continuous on $S$. Note that the $\sigma$-algebra of Borel subsets of $S_\sigma$ turns $S_\sigma$ into a standard Borel space. 

The following representation theorem has been proved in \cite{CMUC} (in the case where finite structures are only considered with uniform measures).
The extension to the general case (finite structures endowed with a probability measure) is easy, and we do not prove it here.

\begin{theorem}
\label{thm:rep}
For every finite structure $\mathbf A$ there is a unique 
probability measure $\mu_{\mathbf A}$ on $S_\sigma$ such that for every finite $\sigma$-structure $\mathbf A$ and every local formula $\phi$ it holds
$$
\langle\phi,\mathbf A\rangle=\int_{S_\sigma} k(\phi)\,{\rm d}\mu_{\mathbf{A}}.
$$

Moreover, for every two finite structures $\mathbf A$ and $\mathbf B$, it holds $\mu_{\mathbf A}=\mu_{\mathbf B}$ if and only if the structures obtained from $\mathbf A$ and $\mathbf B$ by removing connected components without non-zero weight elements are isomorphic as weighted structures.

Denote by $\mathfrak M_\sigma$ the closure of the space of all the probability measures $\mu_{\mathbf A}$ for finite $\mathbf A$ (with respect to weak topology). 

Then, a sequence $\seq{A}=(\mathbf{A}_n)_{n\in\bbbn}$ of finite $\sigma$-structure is local-convergent if and only if the sequence 
$(\mu_{\mathbf{A}_n})_{n\in\bbbn}$ of probability measures
converges weakly, and then the limit probability measure is the unique 
probability measure $\mu_{\lim\seq{A}}$ 
such that for every local formula $\phi$ it holds
$$
\int_{S_\sigma} k(\phi)\,{\rm d}\mu_{\lim\seq{A}}=\langle\phi,\lim\seq{A}\rangle.
$$
\end{theorem}

Recall that a bounded sequence of positive finite measures $\mu_n$ on a metric space $S$ {\em converges weakly} to the finite positive measure $\mu$ if for any bounded continuous function $f:S\rightarrow\bbbr$ it holds $\int f\,{\rm d}\mu_n\rightarrow\int f\,{\rm d}\mu$. This is denoted by $\mu_n\Rightarrow \mu$,

Thus for every continuous function $f:S_\sigma\rightarrow\bbbr$, and for every local convergent sequence $\seq{A}$ it holds
$\mu_{\mathbf{A}_n}\Rightarrow\mu_{\lim\seq{A}}$ and thus
\begin{equation}
\label{eq:lim}
\int_{S_\sigma} f\,{\rm d}\mu_{\lim\seq{A}}=\lim_{n\rightarrow\infty}\int_S f\,{\rm d}\mu_{\mathbf{A}_n}.
\end{equation}
(Note, however that \eqref{eq:lim} does not hold for general Borel functions $f:S\rightarrow\bbbr$.)
When considering random variables, one equally uses the terms {\em convergence in distribution}, {\em weak convergence}, or {\em convergence in law}.
In our setting, we will use the term  ``weak convergence'' when referring to convergence of probability measures on a Stone space, and we then use the notation $\mu_n\Rightarrow\mu$; we will use the term ``convergence in distribution'' (or ``convergence in law'') when referring to convergence random variables with values in $\bbbr^k$, and then we use the notation $X_n\xrightarrow{\mathcal D}X$. In this latter case, we use the term {\em distribution} (or {\em law}) of $X$ for the related probability function on $\bbbr^k$. In the case of a (scalar) random variable $X$, the distribution can be alternatively described by means of its {\em cumulative distribution function} $F_X$ defined by $F_X(x)=\mathbb[X\leq x]$.  

One of the important aspects of the study of local convergence is to determine (or even characterize) those parameters $F$ that are {\em local-continuous} in the sense that if a sequence $\seq{A}=(\mathbf A_n)_{n\in\bbbn}$ of finite structures is local convergent then so is the sequence $(F(\mathbf A_n))_{n\in\bbbn}$ of the associated parameters. Of course, every continuous real function $f\in C(S_\sigma)$ defines a local continuous parameter $\mathbf A\mapsto \int_{S_\sigma} f\,{\rm d}\mu_{\mathbf A}$. But we shall explicit some 
local continuous parameters that are not of this form. As we shall see such parameters will be of prime importance for clustering structures in a local convergent sequence.

\begin{definition}
Let $\mathbf A$ be a $\sigma$-structure and let $\phi$ be a first order formula with free variables $x_1,\dots,x_p$ (with $p\geq 1$).
Denote by $\phi^v(\mathbf A)$ the set
$$\phi^v(\mathbf A)=\Bigl\{
(u_1,\dots,u_{p-1})\in A^{p-1}:\ \mathbf A\models(v,u_1,\dots,u_{p-1})\Bigr\}.$$

The {\em local Stone pairing} of $\phi$ and $\mathbf A$ at $v$ is
\begin{align*}
\langle\psi,\mathbf{A}\rangle_{v}&={\rm Pr}(\mathbf A\models\psi(v,X_2,\dots,X_p))\\
&=\nu_{\mathbf A}^{\otimes (p-1)}(\phi^v(\mathbf A))
\end{align*}
Hence if $\nu_{\mathbf A}(\{v\})\neq 0$ we get that 
the local Stone pairing of $\phi$ and $\mathbf A$ at $v$ is nothing but the conditional probability
${\rm Pr}(\mathbf A\models\psi(X_1,X_2,\dots,X_p)|X_1=v)$.
\end{definition}

In our setting, every finite structure is considered as a probability space and thus the local Stone pairing of a formula $\phi$ and finite structure $\mathbf A$  defines a random variable
$$
\langle \phi,\mathbf A\rangle_\bullet: v\mapsto \langle \phi,\mathbf A\rangle_v.
$$

The (admittedly technical) Lemma~\ref{lem:locSP} will be the key tool for our estimation of clustering parameters. As it proceeds by means of Fourier analysis, we take time to recall some basics.

Given a random variable $\mathbf X$ with values in $\bbbr^k$ and law $P$, the {\em characteristic function} 
of $\mathbf X$ or $P$ is 
$$\gamma(\mathbf t)=\mathbb E[e^{i\mathbf t\cdot\mathbf X}]=\int e^{i\mathbf t\cdot\mathbf x}\,{\rm d}P(\mathbf x)\qquad\text{for every $\mathbf t\in\bbbr^k$},$$
where $\mathbf t\cdot\mathbf x$ denotes the usual inner product of two vectors $\mathbf x$ and $\mathbf t$ in $\bbbr^k$.

A standard Taylor expansion of $E[e^{i\mathbf t\cdot\mathbf X}]$ gives the following expression of the characteristic function as an infinite series:
$$
\gamma(\mathbf t)=\sum_{w_1\geq 0}\dots\sum_{w_k\geq 0}\mathbb E[X_1^{w_1}\dots X_k^{w_k}]\frac{it_1^{w_1}\dots t_k^{w_k}}{w_1!\dots w_k!}.
$$
A main property of characteristic functions is that they fully characterise distribution laws, and that they relate convergence in law of distributions to pointwise convergence of characteristic functions. Precisely, we have:
\begin{theorem}[L\'evy's continuity theorem]
	If $P_n$ are random laws on $\bbbr^k$ whose characteristic
functions $\gamma_n(\mathbf t)$ converge for all $\mathbf t$ to some $\gamma(\mathbf t)$, where $f$ is continuous at $0$
along each coordinate axis, then $P_n$ converges in law to a law $P$ with characteristic
function $f$.
\end{theorem}

Note that there is a one-to-one correspondence between cumulative distribution functions and characteristic functions. If $X$ is a (scalar) random variable we have
\begin{theorem}[L\'evy]
\label{thm:Levy}
	If $\gamma$ is the characteristic function of a scalar random variable with cumulative distribution function $F$, then for two points $a<b$ such that $F$ is continuous at $a$ and $b$ it holds
	$$
	F(b)-F(a)=\frac{1}{2\pi}\lim_{T\rightarrow\infty}\int_{-T}^T\frac{e^{-ita}-e^{-itb}}{it}\gamma(t)\,{\rm d}t.
	$$
	Moreover, if $a$ is an atom of $X$ (that is a discontinuity point of $F$) then
	$$
	F(a)-F(a-0)=\lim_{T\rightarrow\infty}\frac{1}{2T}\int_{-T}^Te^{-ita}\gamma(t)\,{\rm d}t.
	$$
	\end{theorem}
Note that this inversion theorem extends to the case of random vectors.

\begin{lemma}[Continuity of joint distribution of local Stone pairing]
\label{lem:locSP}

	Let  $\phi_1,\dots,\phi_d$ be local formulas (with $p_1,\dots,p_d$ free variables).
	
	For $\mu\in \mathfrak M_\sigma$ and
	 $\mathbf t\in\bbbr^d$ define
	\begin{align*}
	\gamma(\mu,\mathbf t)&=\sum_{w_1\geq 0}\dots\sum_{w_d\geq 0}\left(\int_{S_\sigma} k(\psi_{\mathbf w})\,{\rm d}\mu\right)\,\prod_{j=1}^d\frac{(i t_j)^{w_j}}{w_j!},
	\intertext{where $\psi_{(0,\dots,0})$ is true statement, and for $\mathbf w\neq(0,\dots,0)$ we define} 
\psi_{\mathbf w}&:=\bigwedge_{i=1}^d\bigwedge_{j=1}^{w_i}\phi_i(x_1,x_{n_{i,j}+1},\dots,x_{n_{i,j}+p_i-1})
\intertext{with}
n_{i,j}&=\left(\sum_{\ell=1}^{i-1}w_\ell p_\ell\right)+(j-1)p_i+1.
	\end{align*}
	Then, the following properties hold:
	\begin{enumerate}
		\item for every $\mu\in\mathfrak M_\sigma$, the mapping $\mathbf t\mapsto\gamma(\mu,\mathbf t)$ is the characteristic function of a $d$-dimensional random variable $\mathbf D(\mu)$;
		\item the mapping $\mu\mapsto \mathbf D(\mu)$ is continuous in the sense that if $\mu_n$ converges weakly to $\mu$ then $D(\mu_n)$ converges in distribution to $D(\mu)$, that is:
		$$\mu_n\Rightarrow\mu\qquad\Longrightarrow\qquad\mathbf D(\mu_n)\xrightarrow{\mathcal D}\mathbf D(\mu);$$
		\item for every finite structure $\mathbf A$ (with associated probability measure $\mu_{\mathbf A}\in\mathfrak M_\sigma$)
	the $d$-dimensional random variable $$\mathbf D_{\mathbf A}=(\langle\phi_1,\mathbf A\rangle_\bullet,\dots,\langle\phi_d,\mathbf A\rangle_\bullet)$$ has the same distribution as $\mathbf D(\mu_{\mathbf A})$.
	\end{enumerate}
\end{lemma}

\begin{proof}
	We shall prove the three items in reverse order.
	
	Let us prove (3). For any finite structure $\mathbf A$ and any vector $\mathbf w=(w_1,\dots,w_d)\in\bbbn^d$,  
let $N=n_{d,w_d}+p_d-1$. Then it holds
\begin{align*}
\psi_{\mathbf w}(\mathbf A)&=\Bigl\{\mathbf x\in A^N:\ (\forall i\in [d]\ \forall j\in [w_i])\  (x_1,x_{n_{i,j}+1},\dots,x_{n_{i,j}+p_i-1})\in\phi_i(\mathbf A)\Bigr\}\\
&=\bigcup_{v\in A}\{v\}\times\bigl\{\mathbf x\in A^{N-1}:\ (\forall i\in [d]\ \forall j\in [w_i])\  (x_{n_{i,j}},\dots,x_{n_{i,j}+p_i-2})\in\phi_i^v(\mathbf A)\bigr\}\\
&=\bigcup_{v\in A}\{v\}\times\overbrace{\phi_1^v(\mathbf A)\times\dots\times\phi_1^v(\mathbf A)}^{w_1\text{ times}}\times\dots\times\overbrace{\phi_d^v(\mathbf A)\times\dots\times\phi_d^v(\mathbf A)}^{w_d\text{ times}}.
\intertext{Thus}
	\langle\psi_{\mathbf w},\mathbf A\rangle&=\nu_{\mathbf A}^{\otimes N}(\psi_{\mathbf w}(\mathbf A))\\
	&=\sum_{v\in A}\nu_{\mathbf A}(\{v\})\ \bigl(\nu_{\mathbf A}^{\otimes (p_1-1)}(\phi_1^v(\mathbf A))\bigr)^{w_1}\dots\bigl(\nu_{\mathbf A}^{\otimes (p_d-1)}(\phi_d^v(\mathbf A))\bigr)^{w_d}\\
	&=\mathbb E_v[\langle\phi_1,\mathbf A\rangle_v^{w_1}\dots \langle\phi_d,\mathbf A\rangle_v^{w_d}].
\end{align*}

If follows that the characteristic function $\gamma_{\mathbf A}$ of $\mathbf D_{\mathbf A}$
is equal to:
$$
\gamma_{\mathbf A}(\mathbf t)=\mathbb E[e^{i\mathbf t\cdot \mathbf D_{\mathbf A}}]=\sum_{w_1\geq 0}\dots\sum_{w_d\geq 0}
\langle\psi_{\mathbf w},\mathbf A\rangle
\prod_{j=1}^d \frac{(it_j)^{w_j}}{w_j!}=\gamma(\mu_{\mathbf A},\mathbf t).
$$
(Note that as all the moments $\langle\psi_{\mathbf w},\mathbf A\rangle$ are bounded by $1$ the above series converges for every (complex) vector $\mathbf t$.) As they have the same characteristic functions, the random variables $\mathbf D_{\mathbf A}$ and $\mathbf D(\mu_{\mathbf A})$ have the same distribution.

Let us now prove (1) and (2). 
It is sufficient to consider the case where $\mu_n=\mu_{\mathbf A_n}$ for some local convergent sequence $\seq{A}$. As $\mu_n\Rightarrow\mu$, the fonctions $\gamma(\mu_n,\mathbf t)$ converge pointwise to the function
$\gamma(\mu,\mathbf t)$. Moreover, $\gamma(\mu,\mathbf t)$ is clearly continuous at $\mathbf t=\mathbf 0$ hence, according to L\'evy's continuity theorem, the random variables $\mathbf D_{\mathbf A_n}$ converge in distribution to a random variable $\mathbf D$ with characteristic function $\gamma(\mu,\mathbf t)$.
\end{proof}

\begin{remark}[for an interested reader]
	The formula defining $\psi_{\mathbf w}$ and the equality
	of $\langle\psi_{\mathbf w},\mathbf A\rangle$ and $\mathbb E_v[\langle\phi_1,\mathbf A\rangle_v^{w_1}\dots \langle\phi_d,\mathbf A\rangle_v^{w_d}]$ are generalization of the following simple fact:
	For a graph $G$ and a vertex $v$ of $G$, $\langle x_1\sim x_2,G\rangle_v$ (where $\sim$ denotes adjacency) is the probability that a random vertex $x_2$ is adjacent to $x_1=v$, that is ${\rm deg}(v)/|G|$, and $\langle x_1\sim x_2,G\rangle$ is the average of $\langle x_1\sim x_2,G\rangle_v$ over all vertices of $G$, that is 
	$\langle x_1\sim x_2,G\rangle=\mathbb E[\langle x_1\sim x_2,G\rangle_v]$.
	Similarly,  $\langle (x_1\sim x_2)\wedge (x_1\sim x_3),G\rangle_v$ is the probability that random $x_2$ is adjacent to $v$ and random $x_3$ is adjacent to $v$. As $x_2$ and $x_3$ are independent random vertices, this is nothing but the square of ${\rm deg}(v)/|G|$. Hence
	$\langle (x_1\sim x_2)\wedge (x_1\sim x_3),G\rangle=(\langle x_1\sim x_2,G\rangle_v)^2]$. The same way, for every integer $k$, it holds
	$$
	\langle (x_1\sim x_2)\wedge\dots\wedge (x_1\sim x_{k+1}),G\rangle=(\langle x_1\sim x_2,G\rangle_v)^k].
	$$
\end{remark}

In this paper, we shall be interested in random variables that are a bit more complicated, but definable as a limit of local Stone pairing.
In this context we will need the following complement to Lemma~\ref{lem:locSP}.
\begin{lemma}
\label{lem:locSP2}
Let	$\mu\in\mathfrak M_\sigma$ and let $(\phi_{\ell,1})_{\ell\in\bbbn}$, \dots, $(\phi_{\ell,d})_{\ell\in\bbbn}$ be sequences of local formulas (with $p_1,\dots,p_d$ free variables, respectively) such that for every integer $1\leq i\leq d$ it holds
	$$
	\phi_{1,i}\rightarrow\phi_{2,i}\rightarrow\dots\rightarrow\phi_{\ell,i}\rightarrow\dots
	$$
	(where $\rightarrow$ stands for logical implication).

Let $\mathbf D_\ell(\mu)$ be a $d$-dimensional random variable with characteristic function
  $\gamma_\ell(\mu,\mathbf t)$, which is the function associated to $\phi_{\ell,1},\dots,\phi_{\ell,d}$ as in 
Lemma~\ref{lem:locSP}.

 Then, as $\ell\rightarrow\infty$, the random variables $\mathbf D_\ell(\mu)$ converge in distribution to a random variable $\mathbf D_\infty(\mu)$, whose characteristic function $\gamma_\infty(\mu,\mathbf t)$ is the pointwise limit of the functions $\gamma_\ell(\mu,\mathbf t)$.
\end{lemma}	
\begin{proof}
Let 
$$\psi_{\ell,\mathbf w}:=\bigwedge_{i=1}^d\bigwedge_{j=1}^{w_i}\phi_{\ell,i}(x_1,x_{n_{i,j}+1},\dots,x_{n_{i,j}+p_i-1})$$
(as in Lemma~\ref{lem:locSP}).

	For each vector $\mathbf w\in\bbbn^d$ the sequence 
$\bigl(\int_S k(\psi_{\ell,\mathbf w})\,{\rm d}\mu\bigr)_{\ell\in\bbbn}$ is non-decreasing and bounded by $1$ hence converging. It follows that the functions $\gamma_\ell(\mu,\mathbf t)$ converge pointwise as $\ell\rightarrow\infty$ to 
$$
	\gamma_\infty(\mu,\mathbf t)=\sum_{w_1\geq 0}\dots\sum_{w_d\geq 0}\left(\lim_{\ell\rightarrow\infty}\int_S k(\psi_{\ell,\mathbf w})\,{\rm d}\mu\right)\,\prod_{j=1}^d\frac{(i t_j)^{w_j}}{w_j!},
$$
which is continuous at $\mathbf t=\mathbf 0$. Thus the theorem follows from L\'evy's continuity theorem.
\end{proof}

Note that if $\seq{A}$ is a local convergent sequence of finite structures, it holds
\begin{align*}
\mathbf D_{\ell}(\mu_{\mathbf A_n})&\xrightarrow{\mathcal D}\mathbf D_{\ell}(\mu_{\lim\seq{A}})&\text{as $n\rightarrow\infty$}\\
\mathbf D_{\ell}(\mu_{\lim\seq{A}})&\xrightarrow{\mathcal D}\mathbf D_{\infty}(\mu_{\lim\seq{A}})&\text{as $\ell\rightarrow\infty$} 	
\end{align*}
However, although  $D_{\ell}(\mu_{\mathbf{A}_n})\xrightarrow{\mathcal D}\mathbf D_{\infty}(\mu_{\mathbf{A}_n})$ as $\ell\rightarrow\infty$, it is not true in general that $\mathbf D_{\infty}(\mu_{\mathbf{A}_n})$ converges in distribution to $\mathbf D_{\infty}(\mu_{\lim\seq{A}})$ as $n\rightarrow\infty$.

\begin{definition}
Assume $\mathbf{A}$ be a $\sigma$-structures. 
The {\em $1$-point random lift distribution} of $\mathbf A$ is the probability distribution over (isomorphism classes of) $\sigma^\bullet$-structures (where $\sigma^\bullet$ is the signature obtained from $\sigma$ by adding a unary symbol $M_1$), corresponding to the marking a random elements $X_1$ of $\mathbf A$, drawn from $A$  according to probability distribution $\nu_{\mathbf A}$. We denote by $\Pi$ the map from 
the space ${\rm Rel}(\sigma)$ of isomorphism classes of finite $\sigma$-structures (with domain endowed with a probability measure) to the space $P({\rm Rel}(\sigma))$ of probability distributions over ${\rm Rel}(\sigma)$, which maps a $\sigma$-structure $\mathbf A$ to its
     $1$-point random lift distribution $\Pi(\mathbf A)$.
\end{definition}

Recall that in the context of structures with a domain endowed with a probability measure, the notion of isomorphism  is more involved than standard isomorphism of structures with no associated probability measure.

Let $\mathbf A$ and $\mathbf B$ be $\sigma$-structures, and let $N_A$ (resp. $N_B$) be the union of all the connected components of $\mathbf A$ (resp. $\mathbf B$) without any element of positive measure. Then $\mathbf A$ and $\mathbf B$ are {\em isomorphic} if there exists a bijective mapping $f:A\setminus N_A\rightarrow B\setminus N_B$ preserving the measure (i.e. such that
$\nu_{\mathbf A}=\nu_{\mathbf B}\circ f$) and all the relations both ways
(i.e. $\mathbf A\models R(v_1,\dots,v_n)\iff \mathbf B\models R(f(v_1),\dots,f(v_n))$).

\begin{remark}
The $1$-point random lift corresponds to marking a vertex by $M_1$. Thus the obtained structure is a ``rooted'' structure. We choose this terminology in view of generalization to multiple and iterated random rooting.	
\end{remark}

The space ${\rm Rel}(\sigma)$, endowed with topology defined by local convergence, can be identified (via the continuous injection $\iota^\sigma:\mathbf A\mapsto\mu_{\mathbf A}$ of the representation theorem)
to an open subspace of the Polish space $P(S_\sigma)$, the space of all probability measures on $S_\sigma$ (with weak-$\ast$ topology). We denote by
$\mathfrak M_\sigma$ the closure of $\iota^\sigma({\rm Rel}(\sigma))$.  Similarly, the space ${\rm Rel}(\sigma^\bullet)$ can be identified via injection $\iota^{\sigma^\bullet}$ to an open subspace of $P(S_{\sigma^\bullet})$ with closure $\mathfrak M_{\sigma^\bullet}$. The pushfoward
$\iota^{\sigma^\bullet}_\ast:{\rm P}({\rm Rel}(\sigma))\rightarrow {\rm P}(\mathfrak M_\sigma)$
of $\iota^{\sigma^\bullet}$, defined by 
$$\iota^{\sigma^\bullet}_\ast(\zeta)=\zeta\circ(\iota^{\sigma^\bullet})^{-1}$$
is a continuous injection from ${\rm P}({\rm Rel}(\sigma))$ to ${\rm P}(\mathfrak M_\sigma)$.

The following result makes possible to transfer results about unrooted structures to $1$-point random lifts (i.e. randomly rooted structures). It is a non-trivial refining of Representation Theorem~\ref{thm:rep} and it is the main result of this section.
\begin{great}
\begin{theorem}[1-point random lift theorem]
\label{thm:1lift}
There exists a (unique) continuous function $\widetilde{\Pi}:\mathfrak M_\sigma\rightarrow{\rm P}(\mathfrak M_{\sigma^\bullet})$ such that the following diagram commutes:
\begin{center}
\includegraphics{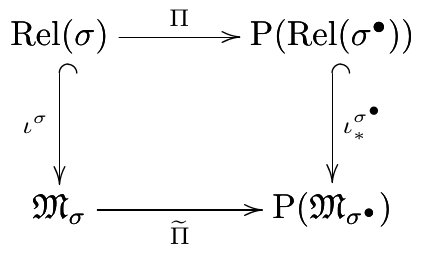}
\end{center}
\end{theorem}
\end{great}
\begin{proof}
	Consider an enumeration
	$\phi_1,\dots,\phi_d,\dots$ of local formulas with respect to signature $\sigma^\bullet$. To each formula $\phi_i$ with $p\geq 0$ free variables we associate the local formula	$\psi_i$ (with respect to signature $\sigma$) with $p+1$ free variables by replacing each free variable $x_i$ by $x_{i+1}$, and then each	term $M_1(t)$ by the term $t=x_1$. Consider $\sigma^\bullet$-structures $\mathbf A^+$ obtained by marking a single element $v\in A$ in a $\sigma$-structure $\mathbf A$. Then it holds
	$$
	\langle\phi_i,\mathbf A^+\rangle=\langle\psi_i,\mathbf A\rangle_v.
	$$
	
	In order to prove Theorem~\ref{thm:1lift}, it is sufficient to prove that
	if $(\mathbf A_n)_{n\in\bbbn}$ is a local convergent	 sequence, then the measures $\rho_\ast^{\sigma^\bullet}\circ\,\Pi(\mathbf A_n)$ converge weakly. This is sufficient as for every $\mu\in\mathfrak M_\sigma$ we can then define $\widetilde{\Pi}(\mu)$ as the weak limit of the measures $\rho_\ast^{\sigma^\bullet}\circ\,\Pi(\mathbf A_n)$, where $(\mathbf A_n)_{n\in\bbbn}$ is any sequence of finite $\sigma$-structures such that $\mu_{\mathbf A_n}\Rightarrow\mu$. This proves Theorem~\ref{thm:1lift}.
	
	Thus let $(\mathbf A_n)_{n\in\bbbn}$ be a local convergent	 sequence and let $\zeta_n=\rho_\ast^{\sigma^\bullet}\circ\,\Pi(\mathbf A_n)$. 	
	 The topology of $\mathfrak M_{\sigma^\bullet}$ can be metrized by
	 means of the following metric:
	  for $\mu_1,\mu_2\in \mathfrak M_{\sigma^\bullet}$, we define the distance ${\rm d}(\mu_1,\mu_2)$ by
	  $$
{\rm d}(\mu_1,\mu_2)=\inf\left\{\epsilon>0:\ \forall 1\leq i\leq 1/\epsilon\ \left|\int_{S_{\sigma^\bullet}}k(\phi_i)\,{\rm d}\mu_1-  
\int_{S_{\sigma^\bullet}}k(\phi_i)\,{\rm d}\mu_2\right|\leq\epsilon\right\}.$$
Let $F:\mathfrak M_{\sigma^\bullet}\rightarrow[0,1]$ be continuous, and let $\epsilon>0$.
	As $\mathfrak M_{\sigma^\bullet}$ is compact there is $\alpha>0$ such that for every $\mu_1,\mu_2\in\mathfrak M_{\sigma^\bullet}$ it holds 
	$${\rm d}(\mu_1,\mu_2)\leq\alpha\quad\Longrightarrow\quad |F(\mu_1)-F(\mu_2)|\leq\epsilon.$$
		
Let $d=\lceil 1/\alpha\rceil$.
Consider the continuous map $p:\mathfrak M_{\sigma^\bullet}\rightarrow [0,1]^d$ defined by 
$$
p(\mu)=\left(\int_{S_{\sigma^\bullet}}k(\phi_1)\,{\rm d}\mu,\dots,
\int_{S_{\sigma^\bullet}}k(\phi_d)\,{\rm d}\mu\right).
$$
Consider a partition $B_1,\dots,B_{d^d}$ of $[0,1]^d$ into $d^d$ boxes of side $1/d$ (``Rubik cube type partition''). Let $C_j=p^{-1}(B_j)$, and let 
$t_j=F(\mu_j)$ for an arbitrary (fixed) choice of $\mu_j\in C_j$.
According to Lemma~\ref{lem:locSP}, the sequence of tuples $(\langle\psi_1,\mathbf A_n\rangle_\bullet,\dots,\langle\psi_d,\mathbf A_n\rangle_\bullet)$ converges in distribution.
Thus for every box $C_j$ the value
\begin{align*}
\int_{C_j}{\rm d}\zeta_n(\mu)&={\rm Pr}[\mu\in C_j]&\text{($\mu$ dist. wrt $\zeta_n$)}\\
&={\rm Pr}[(\langle\phi_1,\mathbf A_n^+\rangle,\dots,\langle\phi_d,\mathbf A_n^+\rangle)\in B_j]&\text{($\mathbf A^+_n$ dist. wrt $\Pi(\mathbf A_n)$)}\\
&={\rm Pr}[(\langle\psi_1,\mathbf A_n\rangle_v,\dots,\langle\psi_d,\mathbf A_n\rangle_v)\in B_j]&\text{($v$ dist. wrt $\nu_{\mathbf A_n}$)}
\end{align*}
converges as $n\rightarrow\infty$.
For every $1\leq j\leq d^d$ and for every $\mu_1,\mu_2\in C_j$ it holds ${\rm d}(\mu_1,\mu_2)\leq 1/d$ hence for every $\mu\in C_j$ it holds $|F(\mu_1)-t_j|\leq \epsilon$.
Thus it holds
$$
\left|\int_{C_j}F(\mu)\,{\rm d}\zeta_n(\mu)-t_j \int_{C_j}{\rm d}\zeta_n(\mu)\right|\leq \epsilon  \int_{C_j}{\rm d}\zeta_n(\mu).
$$
As the sets $C_j$ form a partition of $S_{\sigma^\bullet}$ and as $\zeta_n$ is a probability measure, we get
$$
\left|\int_{S_{\sigma^\bullet}}F(\mu)\,{\rm d}\zeta_n(\mu)-\sum_{j=1}^{d^d}t_j \int_{C_j}{\rm d}\zeta_n(\mu)\right|\leq \epsilon.
$$
Hence for sufficiently large $n$, $\int_{S_{\sigma^\bullet}}F(\mu)\,{\rm d}\zeta_n(\mu)$ concentrates in an interval of size at most $2\epsilon$. By letting $\epsilon\rightarrow 0$ we conclude that $\int_{S_{\sigma^\bullet}}F(\mu)\,{\rm d}\zeta_n(\mu)$ converges, hence (as this holds for every continuous function $F$) that $\zeta_n$ is weakly convergent.
\end{proof}

\begin{remark}
Actually, along the same lines we could prove more: for the linear operator $\widehat\Pi:{\rm P}({\rm Rel}(\sigma))\rightarrow{\rm P}({\rm Rel}(\sigma^\bullet))$ defined by
$$\widehat{\Pi}(\zeta)=\int_{{\rm Rel}(\sigma)}\Pi(\mathbf A)\,{\rm d}\zeta(\mathbf A),$$
there exists a (unique) continuous linear map $\widetilde{\widehat{\Pi}}$ such that the following diagram commutes:
\begin{center}
\includegraphics{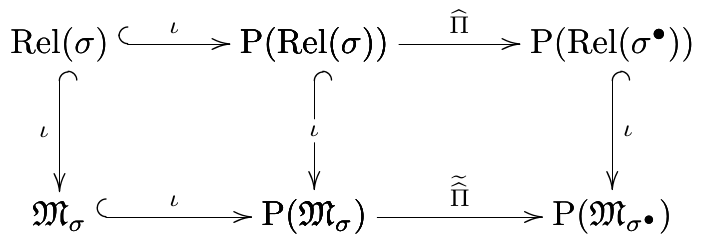}
\end{center}
\end{remark}

\section{Spectrum Driven Clustering}
\label{sec:glob}
We shall now make use of the abstract results of Section~\ref{sec:rep} to compute the globular clusters of a local convergent sequence.

In some sense globular clusters corresponds to the non-zero measure connected components at the limit. Although we do not have, in general, a nice limit structure for a local-convergent sequence of structures, we shall see that nevertheless we can track globular clusters and give an explicit formula for their limit size. 

To achieve this, we shall first show that the moments of the distribution of the limit sizes of the globular clusters may be computed from Stone pairing, and then we shall deduce the distribution of the limit sizes of the globular clusters by standard Fourier analysis.

\subsection{Spectrum}
\label{sec:spectrum}
We start our analysis by the study of the limit sizes of the globular clusters.

Let $\phi_d$ be the formula ${\rm dist}(x_1,x_2)\leq d$.
Let $\seq{A}$ be a local convergent sequence of $\sigma$-structures, and let $D_{d,n}:A_n\rightarrow[0,1]$ be the random variable
$$
D_{d,n}(v)=\langle\phi_d,\mathbf A_n\rangle_v=\nu_{\mathbf A_n}(\ball[d]{\mathbf A_n}(v)).
$$
As obviously $\phi_d$ implies $\phi_{d+1}$ it follows from Lemma~\ref{lem:locSP2} that there exists random variables $D_d$ and $D$ such that $D_{d,n}\xrightarrow{\mathcal D}D_d$ and $D_d\xrightarrow{\mathcal D} D$ (which are limits in distribution, for $n\rightarrow\infty$ and $d\rightarrow\infty$, respectively).

\begin{remark}
The random variables $D_{d,n}$ have here a concrete meaning, as the measure of the radius $d$ ball centered at a random vertex. However, there is no particular meaning for the sample space of random variables $D_d$ and $D$ (as the existence of these were simply derived from convergence of characteristic functions). 

Even if we intuitively interpret the random variables $D_d$ and $D$ as if they were built on a similar limit sample space, we have to take care in our argumentation that this interpretation is not {\em a priori} justified.
\end{remark}

We denote respectively by $F_{d,n}, F_d$ and $F$ the cumulative distribution functions of $D_{d,n}, D_d$ and $D$. 
According to Froda's theorem, each $F_d$ (and $F$) has at most countably many discontinuities. 
As $D_d\xrightarrow{\mathcal D}D$ (as $d\rightarrow\infty$) the functions $F_d$ converge pointwise to $F$ at every continuity point of $F$. Similarly, for each $d\in\bbbn$,
 as $D_{n,d}\xrightarrow{\mathcal D}D_d$ (as $n\rightarrow\infty$) the functions $F_{n,d}$ converge pointwise to $F_d$ at every continuity point of $F_d$.
We define $\Lambda_d$ (resp. $\Lambda$) as the (at most countable) set of discontinuities of $F_d$ (resp. $F$), and let $\mathcal R=[0,1]\setminus\bigl(\Lambda\cup\bigcup_{d\in\bbbn}\Lambda_d\bigr)$. In other words, $\mathcal R$ is the (cocountable) set of points where  all the considered limit cumulative functions ($F$ and $F_d$ for $d\in\bbbn$) are continuous.

\begin{remark}
If $d_1<d_2$ then for every integer $n$ and ever real $t$ it holds
$$F_{d_1,n}(t)={\rm Pr}(D_{d_1,n}\leq t)\geq {\rm Pr}(D_{d_2,n}\leq t)=F_{d_2,n}(t)$$ 
and thus also $F_{d_1}\geq F_{d_2}\geq F$.	
\end{remark}

We now take time for a useful simple lemma.
\begin{lemma}
\label{lem:interv}
Let $t_1<t_2$ be in $\mathcal R$, and let $n$ and $d_1<d_2$ be integers.

 Then
$$F_{d_2,n}(t_2)-F_{d_1,n}(t_1)\leq {\rm Pr}(t_1<D_{d_1,n}\leq D_{d_2,n}\leq t_2)\leq F_{d_1,n}(t_2)-F_{d_1,n}(t_1).$$
Hence
if $|F(t_j)-F_{d_1}(t_j)|<\epsilon$ and $|F_{d_i}(t_j)-F_{d_i,n}(t_j)|<\epsilon$ hold for $i,j\in\{1,2\}$ then
$$|{\rm Pr}(t_1<D_{d_1,n}\leq D_{d_2,n}\leq t_2)-(F(t_2)-F(t_1))|<4\epsilon.$$
\end{lemma}
\begin{proof}
	The right inequality is obvious as ${\rm Pr}(t_1<D_{d_1,n}\leq D_{d_2,n}\leq t_2)\leq {\rm Pr}(t_1<D_{d_1,n}\leq t_2)$. For the left inequality, note that 
	\begin{align*}
	{\rm Pr}(t_1<D_{d_1,n}\leq D_{d_2,n}\leq t_2)&={\rm Pr}(t_1<D_{d_1,n}\leq t_2)-{\rm Pr}(t_1<D_{d_1,n}\leq t_2<D_{d_2,n})\\
	&\leq {\rm Pr}(t_1<D_{d_1,n}\leq t_2)-({\rm Pr}(D_{d_1,n}\leq t_2)-{\rm Pr}(D_{d_1,n}\leq t_2))\\
	&=(F_{d_1,n}(t_2)-F_{d_1,n}(t_1))-(F_{d_1,n}(t_2)-F_{d_2,n}(t_2))
	\end{align*}
\end{proof}

We are now approaching the final steps of our cluster analysis. This is admittedly technical and we shall
need further several lemmas in order to prove Theorem~\ref{thm:main}.

However the intuition for our proof is easy and can be outlined as follows: if we would have a proper explicit limit structure, the random variable $D$ would intuitively correspond to the measure of the connected component of a random element. Thus we expect $D$ to be a discrete random variable, and that the probability that $D=\lambda$ is the measure of the union of all connected components of measure $\lambda$ hence an integral multiple of $\lambda$.   
The aim of this part is to show that this intuitive notion of limit connected components is captured by the concept of globular clusters. This setting will not only ground the above intuition, but will also allow to track the formation of the limit connected components down to the structures in the sequence.

Hence our first step is to prove that $D$ is a purely discrete random variable, that is that its cumulative distribution function $F$ is constant except at its (at most countably many) discontinuity points. This we shall do now.

\begin{lemma}
\label{lem:Fdisc}
The spectrum distribution of a local-convergent sequence of finite structures is discrete and its associated mass probability function
$p:[0,1]\rightarrow[0,1]$ defined by

$$
p(x)=F(x)-\lim_{\epsilon\rightarrow 0}F(x-\epsilon).
$$
is such that that for every $x\in[0,1]$, either $p(x)=0$ or $p(x)\geq x$.
\end{lemma}
\begin{proof}
We shall prove that for $t_1<t_2$ in $\mathcal R$, either $F(t_1)=F(t_2)$ or $F(t_2)-F(t_1)\geq t_1$.
It will follow,
by cutting the interval $[t_1, t_2]$  recursively,   that $F$ is constant except at its discontinuity points, and that the mass probability function $p$ satisfies $p(x)\geq x$ at every point $x$ where $p(x)\neq 0$.

So let $t_1<t_2$ be in $\mathcal{R}$ and such that $F(t_1)<F(t_2)$, and
 let $0<\epsilon<(F(t_2)-F(t_1))/4$.
As $D_d\xrightarrow{\mathcal D}D$ (as $d\rightarrow\infty$) there exists $d$ such that 
 $|F_d(t_1)-F(t_1)|<\epsilon$ and $|F_d(t_2)-F(t_2)|<\epsilon$. 
 Moreover, as $D_{n,d}\xrightarrow{\mathcal D}D_d$ (for fixed $d$ and $n\rightarrow\infty$)
  there exists $n$ such that
 $|F_{d,n}(t_1)-F_d(t_1)|<\epsilon$,
 $|F_{d,n}(t_2)-F_d(t_2)|<\epsilon$,
  $|F_{2d,n}(t_1)-F_{2d}(t_1)|<\epsilon$
  and
 $|F_{2d,n}(t_2)-F_{2d}(t_2)|<\epsilon$.
 
 According to Lemma~\ref{lem:interv} it holds
 \begin{align*}
 	{\rm Pr}(t_1<D_{d,n}\leq D_{2d,n}\leq t_2)-(F(t_2)-F(t_1))&<4\epsilon
 \end{align*}
 thus ${\rm Pr}(t_1<D_{d,n}\leq D_{2d,n}\leq t_2)>0$.
Hence there exists $v\in A_n$ such that $t_1<D_{d,n}(v)\leq D_{2d,n}(v)\leq t_2$.
For every $x\in\ball[d]{\mathbf A_n}(v)$ it holds
$\ball[d]{\mathbf A_n}(x)\subseteq \ball[2d]{\mathbf A_n}(v)$ hence $D_{d,n}(x)\leq D_{2d,n}(v)\leq t_2$. Also, $\ball[2d]{\mathbf A_n}(x)\supseteq \ball[d]{\mathbf A_n}(v)$ thus $D_{2d,n}(x)\geq D_{d,n}(v)>t_1$.
As this holds for every $x\in \ball[d]{\mathbf A_n}(v)$, we get $\ball[d]{\mathbf A_n}(v)\subseteq \{x: t_1<D_{2d,n}(x)\text{ and }D_{d,n}(x)\leq t_2\}$. 
Thus we have
\begin{align*}
\nu_{\mathbf A_n}(\ball[d]{\mathbf A_n}(v))-{\rm Pr}(t_1<D_{2d,n}\leq t_2)&\leq {\rm Pr}(t_1<D_{2d,n}\text{ and }D_{d,n}\leq t_2)
-{\rm Pr}(t_1<D_{2d,n}\leq t_2) \\
&\leq  {\rm Pr}(D_{d,n}\leq t_2)-{\rm Pr}(D_{2d,n}\leq t_2) \\
&=F_{d,n}(t_2)-F_{2d,n}(t_2)\\
&<4\epsilon.
\end{align*}

Hence
${\rm Pr}(t_1<D_{2d,n}\leq t_2)>\nu_{\mathbf A_n}(\ball[d]{\mathbf A_n}(v))-4\epsilon>t_1-4\epsilon$. Hence $F(t_2)-F(t_1)>t_1-8\epsilon$. By letting  $\epsilon\rightarrow 0$, we get $F(t_2)-F(t_1)\geq t_1$ as claimed.
\end{proof}

Recall that $\Lambda$ is the set of discontinuities of $F$, that is the set of $x\in[0,1]$ such that $p(x)\neq 0$. Note that
it follows from Lemma~\ref{lem:Fdisc} that 
for every integer $z$ there exists at most $z$ values
 $\lambda\in\Lambda$ with $\lambda\geq 1/z$.

The next lemma will ground our intuition that $p(\lambda)$ should be an integral multiple of $\lambda$. Indeed, we will prove later that $p(\lambda)/\lambda$ is the number of disjoint globular clusters with limit measure $\lambda$.

\begin{lemma}
\label{lem:integ}
Let $\lambda\in\Lambda$. Then 
$p(\lambda)/\lambda\in\bbbn$.
\end{lemma}
\begin{proof}
Let $0<\epsilon<\lambda^2/11$.
Fix $t_1,t_2\in\mathcal R$ with $0<t_1<\lambda<t_2$, $t_2-t_1<\epsilon$, and such that $\lambda$ is the only discontinuity point of $F$ on $[t_1,t_2]$ (hence $p(\lambda)=F(t_2)-F(t_1)$).

Then there exist $\delta=\delta(\epsilon,t_1,t_2)$ such that for every $d\geq \delta$ it holds
$|F(t_1)-F_{kd}{t_i}|<\epsilon$  for every $1\leq k\leq 4$ and every $i\in\{1,2\}$, and
there exists 
$\eta=\eta(\epsilon,t_1,t_2,d)$ such that for every $n\geq \eta$
it holds $|F_{kd,n}(t_i)-F_{kd}(t_i)|<\epsilon$ for every $1\leq k\leq 4$ and every $i\in\{1,2\}$.

We prove by contradiction that no two vertices $v_1,v_2\in A_n$ exist such that
\begin{gather*}
	t_1<D_{d,n}(v_1)\leq D_{4d,n}(v_1)\leq t_2,\\
	t_1<D_{d,n}(v_2)\leq D_{4d,n}(v_2)\leq t_2,\\
	2d<{\rm dist}(v_1,v_2)\leq 3d.
\end{gather*}
Assume the contrary.
Then $\ball[4d]{\mathbf A_n}(v_1)$ contains the disjoint union
of $\ball[d]{\mathbf A_n}(v_1)$ and $\ball[d]{\mathbf A_n}(v_2)$ thus
$$\nu_{\mathbf A_n}(\ball[4d]{\mathbf A_n}(v_1))> 2t_1>t_1+\lambda-\epsilon=(t_1+\epsilon)+(\lambda-2\epsilon)>t_2,$$
 contradicting $D_{4d,n}(v_1)\leq t_2$.

Let $S=S_{t_1,t_2,d}(n)$ be a maximal set of vertices $v\in A_n$, 
pairwise at distance greater than $3d$, and such that 
$$	t_1<D_{d,n}(v)\leq D_{4d,n}(v)\leq t_2.$$
First note that for $v,v'\in S$ the balls
$\ball[d]{\mathbf A_n}(v)$ and $\ball[d]{\mathbf A_n}(v')$ do not intersect hence
$$
1\geq\nu_{\mathbf A_n}\bigl(\bigcup_{v\in S}\ball[d]{\mathbf A_n}(v)\bigr)=\sum_{v\in S}D_{d,n}(v)>t_1|S|>(\lambda-\epsilon)|S|.
$$
Thus $|S|<1/(\lambda-\epsilon)$.

Also every vertex $w$ such that 
$t_1<D_{d,n}(w)\leq D_{4d,n}(w)\leq t_2$ belongs to
$\bigcup_{v\in S}\ball[2d]{\mathbf A_n}(v)=\ball[2d]{\mathbf A_n}(S)$. 
It follows that 
${\rm Pr}(t_1<D_{d,n}\leq D_{4d,n}\leq t_2)\leq t_2 |S|$.
Also, $${\rm Pr}(t_1<D_{2d,n}\leq t_2)\geq \nu_{\mathbf A_n}(\ball[d]{\mathbf A_n}(S))=\sum_{s\in S}D_{d,n}(s)> t_1|S|>\lambda |S|-\epsilon/(\lambda-\epsilon).$$

As
$${\rm Pr}(t_1<D_{2d,n}\leq t_2)
\leq {\rm Pr}(t_1<D_{d,n}\leq D_{4d,n}\leq t_2)
\leq t_2|S|
<\lambda |S|+\epsilon/(\lambda-\epsilon),$$
we get
$$|{\rm Pr}(t_1<D_{2d,n}\leq t_2)-\lambda|S||<\epsilon/(\lambda-\epsilon).$$
As
$$|{\rm Pr}(t_1<D_{2d,n}\leq t_2)-p(\lambda)|<4\epsilon$$
we deduce
$$|p(\lambda)-\lambda|S||<(4+1/(\lambda-\epsilon))\epsilon.$$

As $\epsilon<\lambda^2/11$, it holds  
$|p(\lambda)-|S|\lambda|<\lambda/2$, thus 
 $|S|=|S_{t_1,t_2,d}(n)|$ is constant for all the values
$t_1,t_2,d,n$
consistent with $0<\epsilon<\lambda^2/11$. Denoting $m(\lambda)$ this common value of $|S_{t_1,t_2,d}(n)|$, and by 
 letting $\epsilon\rightarrow 0$, we get 
 $p(\lambda)=m(\lambda)\lambda$ thus $p(\lambda)/\lambda\in\bbbn$.
\end{proof}

We now define several functions, which will be of key  importance in our precise definition and analysis of the globular clusters.

Let us fix $\lambda\in\Lambda$.

\begin{trivlist}
\item {\bf Definition of $\epsilon_z$.}
For $z\in\bbbn$, we define
\begin{equation}
\epsilon_z=2^{-z}.	
\end{equation}
\item {\bf Definition of $z_0(\lambda)$.}
We define 
\begin{equation}
	z_0(\lambda)=\lceil 5-2\log_2\lambda\rceil.
\end{equation}
(Thus $\epsilon_{z_0(\lambda)}\leq \lambda^2/32$.)
\item {\bf Definition of $\alpha_z(\lambda)$ and $\beta_z(\lambda)$.}
We define
$$\alpha_1(\lambda)<\alpha_2(\lambda)<\dots<\lambda<\dots<\beta_2(\lambda)<\beta_1(\lambda),$$
such that $\Lambda\cap[\alpha_1(\lambda),\beta_1(\lambda)]=\{\lambda\}$, every $\alpha_z(\lambda)$ and $\beta_z(\lambda)$ belong to $\mathcal R$, and such that for every $z\in\bbbn$ it holds
\begin{equation}
|\beta_z(\lambda)-\alpha_z(\lambda)|<\epsilon_z.
\end{equation}

\item {\bf Definition of $\delta_z(\lambda)$.}
As $D_d\xrightarrow{\mathcal D}D$ (as $d\rightarrow\infty$) we can define integers $\delta_1(\lambda)<\delta_2(\lambda)<\dots$ such that for every $z\in\bbbn$ and every $d\geq\delta_z(\lambda)$ it holds
\begin{align}
	|F_d(\alpha_z(\lambda))-F(\alpha_z(\lambda))|&<\epsilon_z\\
	|F_d(\beta_z(\lambda))-F(\beta_z(\lambda))|&<\epsilon_z
\end{align}

\item {\bf Definition of $\eta_z(\lambda)$.}
As $D_{n,d}\xrightarrow{\mathcal D}D_d$ (for fixed $d$ and as $n\rightarrow\infty$) we can define integers
$\eta_1(\lambda)<\eta_2(\lambda)<\dots$ such that for every $z\in\bbbn$, every 
$n\geq\eta_z(\lambda)$ and every integer $k\in\{1,\dots 8\}$ it holds
\begin{align}
|F_{k\delta_{z}(\lambda),n}(\alpha_{z}(\lambda))-F_{k\delta_{z}(\lambda)}(\alpha_{z}(\lambda))|&<\epsilon_z\\
|F_{k\delta_{z}(\lambda),n}(\beta_{z}(\lambda))-F_{k\delta_{z}(\lambda)}(\beta_{z}(\lambda))|&<\epsilon_z	
\end{align}
\end{trivlist}

We now define some sequences of sets.
The sets $Z_n^{\lambda,z}$ will anticipate our construction of globular clusters, by giving a rough approximate of them. Then the set $S_n^\lambda$ will collect a ``center'' for each of the ``component'' of size $\lambda$.

\begin{trivlist}
\item {\bf Definition of $Z_n^{\lambda,z}$.}
For $n,z\in\bbbn$ we define subset
$Z^{\lambda,z}_n$ as follows:
\begin{itemize}
	\item If $n<\eta_z$  then $Z^{\lambda,z}_n=\emptyset$;
	\item Otherwise, 
$Z^{\lambda,z}_n$ is the set of all elements of $A_n$ such that 
\begin{align}
D_{8\delta_z,n}(v)&\leq \beta_z(\lambda)\label{eq:Z1}\\	
D_{\delta_{z'},n}(v)&>\alpha_{z'}(\lambda)&(\forall z'\in\{z_0(\lambda),\dots,z\})\label{eq:Z2}
\end{align}
\end{itemize}
\item {\bf Definition of $S_n^{\lambda}$.}
We define 
$S_n^\lambda$ as a maximal set of vertices $v\in Z_n^{\lambda,z}$, pairwise at distance at least $7\delta_z$, where
$z$ is (implicitly) defined by  $\eta_z\leq n<\eta_{z+1}$.
\end{trivlist}

We take time for few remarks:
\begin{remark}
Note that \eqref{eq:Z1} implies
$D_{8\delta_{z'},n}(v)\leq \beta_{z'}(\lambda)$ for every $1\leq z'\leq z$. Also \eqref{eq:Z2} becomes clearly more and more restrictive as $z$ grows.
 Hence for every $z\geq z_0(\lambda)$ and every $n\in\bbbn$ such that
$\eta_z\leq n<\eta_{z+1}$ it holds
\begin{equation}
	Z^{\lambda,z_0(\lambda)}_n\supseteq Z^{\lambda,z_0(\lambda)+1}_n\supseteq\dots\supseteq Z^{\lambda,z}_n\supsetneq Z^{\lambda,z+1}_n=Z^{\lambda,z+2}_n=\dots=\emptyset.
\end{equation}
\end{remark}

\begin{remark}
According to the definitions of $\delta_z$ and $\eta_z$, it holds
$$|F_{8\delta_z(\lambda),n}(\beta_z(\lambda)- F(\beta_z(\lambda)|<2\epsilon_z$$
\end{remark}

\begin{remark}
If $z'<z''$ then, according to Lemma~\ref{lem:interv} it holds
$$|{\rm Pr}(\alpha_{z'}(\lambda)<D_{\delta_{z'}(\lambda),n}\leq D_{\delta_{z''}(\lambda),n}\leq \alpha_{z''}(\lambda))|<4\epsilon_{z'}$$
Thus $${\rm Pr}(D_{\delta_z(\lambda),n}>\alpha_z(\lambda))-{\rm Pr}\bigl(\bigwedge_{z'=z_0}^zD_{\delta_{z'}(\lambda),n}>\alpha_{z'}(\lambda)\bigr)<4\sum_{z'=z_0}^z\epsilon_{z'}=2^{2-z_0}.$$
It follows that
$$
\nu_{\mathbf A_n}(Z_n^{\lambda,z})\geq p(\lambda)-4\epsilon_z-2^{2-z_0}.
$$
	
\end{remark}
We now prove that, as wanted, the number of elements of $S_n^\lambda$ is (for sufficiently large $\lambda$) the anticipated number of globular clusters of size $\lambda$.
\begin{lemma}
	For every $\lambda\in\mathcal R$ and every $n\geq \eta_{\lceil\lambda^{-1}\rceil}$ it holds
	$|S_n^\lambda|=p(\lambda)/\lambda$.
\end{lemma}
\begin{proof}
Note that obviously, as $\nu_{\mathbf A_n}(\ball[7\delta_z(\lambda)]{\mathbf A_n}(s))<\lambda+\epsilon_z$ holds for every $s\in S_n^\lambda$, we get
$$|S_n^\lambda|\geq\frac{|Z_n^{\lambda,z}|}{\lambda+\epsilon_z}\geq \frac{p(\lambda)-2^{2-z}-2^{2-z_0(\lambda)}}{\lambda+2^{-z}}=\frac{p(\lambda)}{\lambda}-\frac{2^{-z}/\lambda-2^{2-z}-2^{2-z_0(\lambda)}}{\lambda-2^{-z}}>\frac{p(\lambda)}{\lambda}-1,$$
hence $|S_n^\lambda|\geq p(\lambda)/\lambda$. On the other hand, for every $s\in S_n^\lambda$ it holds
$D_{\delta_z(\lambda),n}(s)>\alpha_z(\lambda)$ and $D_{3\delta_z(\lambda),n}(s)\leq\beta_z(\lambda)$ thus for every $v\in\ball[\delta_z(\lambda)]{\mathbf A_n}(S_n^\lambda)$ it holds
$\alpha_z(\lambda)<D_{2\delta_z(\lambda),n}(s)\leq\beta_z(\lambda)$ thus
\begin{align*}
	(\lambda-\epsilon_z)|S_n^\lambda|&<\nu_{\mathbf A_n}(\ball[\delta_z(\lambda)]{\mathbf A_n}(S_n^\lambda))\\
	&\leq {\rm Pr}(\alpha_z(\lambda)<D_{2\delta_z(\lambda),n}(s)\leq\beta_z(\lambda))\\
	&<p(\lambda)+4\epsilon_z
	\intertext{hence}
|S_n^\lambda|&<\frac{p(\lambda)}{\lambda}+1.
\end{align*}
Altogether, it follows that $|S_n^\lambda|=p(\lambda)/\lambda$.
\end{proof}

We are now ready to define sets gathering all the ``components'' with limit measure $\lambda$. We will prove that they define universal clusters.
\begin{trivlist}
	\item {\bf Definition of $C^\lambda_n$.}
For $\lambda\in\Lambda$ and $n\in\bbbn$ we define
\begin{equation}
C^\lambda_n=\begin{cases}
\emptyset,&\text{if }n<\eta_{z_0(\lambda)}\\
\ball[2\delta_z]{\mathbf A_n}(S^\lambda_n),&\text{otherwise, if $z$ is such that }\eta_z\leq n<\eta_{z+1}
\end{cases}
\end{equation}
\end{trivlist}

The sets $C^\lambda_n$ will be the building block for the construction of our clusters. Lemmas~\ref{lem:ZC} to~\ref{lem:phiC} will be used to prove that the sequences 
$\Xseq{C^\lambda}$ define a clustering of $\seq{A}$ into countably many universal clusters plus a residual cluster.
The general aspects of the sets $Z_n^{\lambda,z},S_n^\lambda$, and $C_n^\lambda$ we tried to visualize by Fig.~\ref{fig:Zn}.

\begin{figure}
	\begin{center}
		\includegraphics[width=\textwidth]{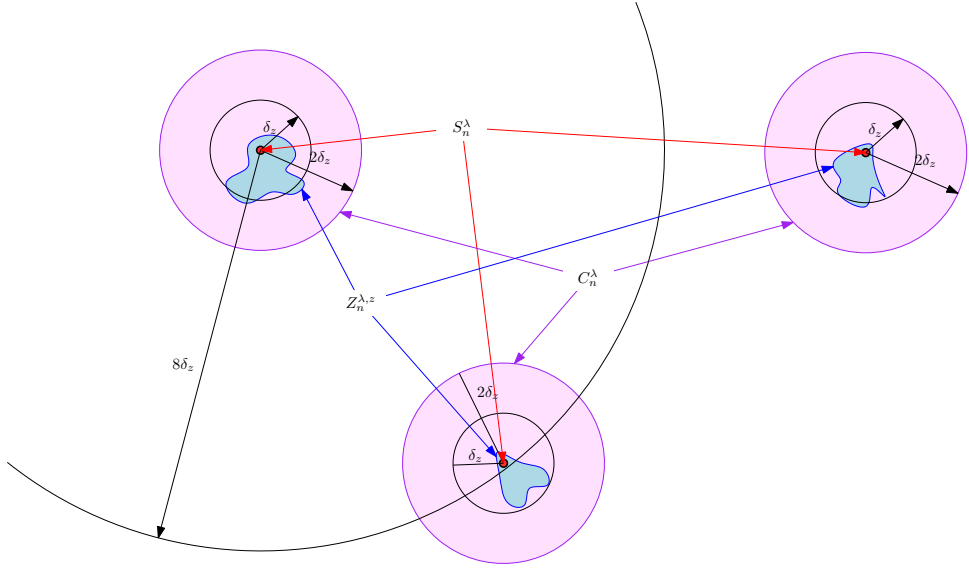}
	\end{center}
	\caption{General aspects of the sets $Z_n^{\lambda,z},S_n^\lambda$, and $C_n^\lambda$}
	\label{fig:Zn}
\end{figure}

Similarly to Lemma~\ref{lem:integ} we prove

\begin{lemma}
\label{lem:ZC}
	Let $\lambda\in\Lambda$, let $z\geq z_0(\lambda)$, and let $\eta_z\leq n<\eta_{z+1}$.
	
	Then $Z_n^{\lambda,z}\subseteq C_n^\lambda$.
\end{lemma}
\begin{proof}
	Assume for contradiction that there exists an element $v\in Z_n^\lambda\setminus C_n^\lambda$.
By the maximality of $S_n^\lambda$, we get that 
$v$ is at distance at most $7\delta_z$ from some element $u\in S_n^\lambda$. Moreover, ${\rm dist}(u,v)>2\delta_z$ as $v\notin\ball[2\delta_z]{\mathbf A_n}(S_n^\lambda)$.
Then $\ball[8\delta_z]{\mathbf A_n}(v)$ contains the disjoint union
of $\ball[\delta_z]{\mathbf A_n}(v)$ and $\ball[\delta_z]{\mathbf A_n}(u)$ thus
$$\nu_{\mathbf A_n}(\ball[8\delta_z]{\mathbf A_n}(v))> 2\alpha_z(\lambda)>\alpha_z(\lambda)+\lambda-\epsilon_2=(\alpha_z(\lambda)+\epsilon_z)+(\lambda-2\epsilon_z)>\beta_z(\lambda),$$
 contradicting $D_{8\delta_z,n}(v)\leq \beta_z(\lambda)$.
\end{proof}

We now prove that our sets $C_n^\lambda$ are pairwise disjoint.

\begin{lemma}
\label{lem:disjC}
	Let $\lambda<\lambda'$ be two elements of $\Lambda$, let $z\geq \max(z_0(\lambda),\lceil 1-\log_2(\lambda'-\lambda)\rceil)$, 
		and let $\eta_z\leq n<\eta_{z+1}$.
	
	Then $C_n^\lambda\cap C_n^{\lambda'}=\emptyset$.
\end{lemma}
\begin{proof}
	Assume for contradiction that there exists an element $v\in C_n^\lambda\cap  C_n^{\lambda'}$.
	Then there exists $u\in S_n^\lambda$ and $u'\in S_n^{\lambda'}$ such that
	$v\in \ball[2\delta_z]{\mathbf A_n}(u)\cap \ball[2\delta_z]{\mathbf A_n}(u')$ hence ${\rm dist}(u,u')\leq 4\delta_z$. It follows that $\ball[\delta_z]{\mathbf A_n}(u')\subseteq \ball[8\delta_z]{\mathbf A_n}(u)$ hence
	$\alpha_z(\lambda')\leq \beta_z(\lambda)$ thus $|\lambda-\lambda'|<2.2^{-z}$, contradicting our choice of $z$.
\end{proof}

We now prove that the measure of $C_n^\lambda$ is concentrated around the limit measure $p(\lambda)$.
\begin{lemma}
\label{lem:measC}
		Let $\lambda\in\Lambda$, let $z\geq z_0(\lambda)$, and let $\eta_z\leq n<\eta_{z+1}$.
		
		Then $|\nu_{\mathbf A_n}(C_n^\lambda)-p(\lambda)|<2^{-z}p(\lambda)/\lambda$.
\end{lemma}
\begin{proof}
	Note that $$\nu_{\mathbf A_n}(C_n^\lambda)=\sum_{s\in S_n^\lambda}\nu_{\mathbf A_n}(\ball[2\delta_z]{\mathbf A_n}(s)).$$
	Hence, as $|S_n^\lambda|=p(\lambda)/\lambda$) it holds
	$|\nu_{\mathbf A_n}(C_n^\lambda)-p(\lambda)|<2^{-z}p(\lambda)/\lambda$.
\end{proof}

The next lemma not only shows that the outer boundary of $\Xseq{C^\lambda}$ is negligible (what is required in order for $\Xseq{C^\lambda}$ to be a cluster) but also that the neighborhood of these outer boundaries are so small that their sum will also be small (what we will make use of in lemma~\ref{lem:negC}).

\begin{lemma}
\label{lem:bordC}
	Let $\lambda\in\Lambda$, let $n\geq\eta_{\lceil\lambda^{-1}\rceil}$, and let $z\in\bbbn$ be such that $\eta_z\leq n<\eta_{z+1}$. Then it holds
	$$\nu_{\mathbf A_n}(\ball[\delta_z]{\mathbf A_n}(\partial_{\mathbf A_n}C^\lambda_n))<2^{1-z}p(\lambda)/\lambda.$$
	In particular, $\partial_{\seq{A}}\Xseq{C}^\lambda\approx\Xseq{0}$.
\end{lemma}
\begin{proof}
	As elements of $S^\lambda_n$ are pairwise at distance at least $7\delta_z$, it holds
	$$
	\ball[\delta_z]{\mathbf A_n}(\partial_{\mathbf A_n}C^\lambda_n)=
	\biguplus_{v\in S^\lambda_n}\bigl(\ball[3\delta_z]{\mathbf A_n}(v)\setminus \ball[\delta_z]{\mathbf A_n}(v))\bigr)
	$$
	(where $\uplus$ denotes a disjoint union).
	As $v\in S_n^\lambda$ it holds
	\begin{align*}
	\nu_{\mathbf A_n}(\ball[\delta_z]{\mathbf A_n}(v))&=D_{\delta_z,n}(v)>\alpha_z(\lambda)\\	
	\nu_{\mathbf A_n}(\ball[3\delta_z]{\mathbf A_n}(v))&\leq D_{8\delta_z,n}(v)\leq \beta_z(\lambda)
	\intertext{Hence}
	\nu_{\mathbf A_n}\bigl(\ball[3\delta_z]{\mathbf A_n}(v)\setminus \ball[\delta_z]{\mathbf A_n}(v))\bigr)&<\epsilon_z
	\intertext{Thus}
	\nu_{\mathbf A_n}(\ball[\delta_z]{\mathbf A_n}(\partial_{\mathbf A_n}C^\lambda_n))&<|S^\lambda_n|\epsilon_z< 2\epsilon_zp(\lambda)/\lambda
	\end{align*}
\end{proof}

\begin{lemma}
\label{lem:negC}
	Let $n\in\bbbn$, let $\lambda\in\Lambda$ be minimum such that $n\geq\eta_{z_0(\lambda)}$.
	Let $z$ be defined by $\eta_z\leq n<\eta_{z+1}$, and let
	$$W_n=\{v: D_{\delta_z,n}(v)>\alpha_z(\lambda)\}\setminus\bigcup_{\alpha\in \Lambda} C_n^\alpha.$$
	Then $$\nu_{\mathbf A_n}(\ball[\delta_z]{\mathbf A_n}(W_n))\leq 2^{-z}(1+3/\lambda).$$
	In particular, $\Xseq{W}\supseteq \partial_{\seq{A}}\bigl(\bigcup_{\lambda	\in\Lambda}\Xseq{C}^\lambda\bigr)$ and 
	$\Xseq{W}\approx\Xseq{0}$.
\end{lemma}
\begin{proof}
	Let $F_n=\{v: D_{\delta_z,n}(v)>\alpha_z(\lambda)\}$.
	Then $\ball[\delta_z]{\mathbf A_n}(F_n)\subseteq\{v: D_{\delta_2z,n}(v)>\alpha_z(\lambda)\}$.
	Hence $\nu_{\mathbf A_n}(\ball[\delta_z]{\mathbf A_n}(F_n))\leq 1-F_{2\delta_z,n}(\alpha_z(\lambda))$.
	It follows that
	\begin{align*}
		\nu_{\mathbf A_n}(\ball[\delta_z]{\mathbf A_n}(W_n))&\leq \nu_{\mathbf A_n}(\ball[\delta_z]{\mathbf A_n}(F_n))-\sum_{\lambda'\geq\lambda}\nu_{\mathbf A_n}(C_n^{\lambda'})+\sum_{\lambda'\geq\lambda}\nu_{\mathbf A_n}(\ball[\delta_z]{\mathbf A_n}(\partial_{\mathbf A_n}C_n^{\lambda'}))\\
		&\leq\epsilon_z+\frac{2^{-z}}{\lambda}+\frac{2^{1-z}}{\lambda}.
	\end{align*}
\end{proof}

We now are ready for our last lemma needed to prove that the sequences $\Xseq{C^\lambda}$ define a clustering of $\seq{A}$ into countably many universal clusters plus a residual cluster.
\begin{lemma}
	\label{lem:phiC}
	For each $\lambda\in\Lambda$ the sequence $\Xseq{C}^\lambda=(C^\lambda_n)_{n\in\bbbn}$ is a cluster.
\end{lemma}
\begin{proof}
Let $\phi$ be an $r$-local strongly local formula with free variables $x_1,\dots,x_q$.
For $d\in\bbbn$ let $\Psi_d$ be the following formula with $q+1$ free variables
$$
\Psi_d:\quad\phi(x_2,\dots,x_{q+1})\wedge\bigwedge_{i=2}^{q+1}{\rm dist}(x_1,x_i)\leq d.
$$
Note that if $d_1<d_2$ and $v\in A_n$ then
$$	0\leq \langle\Psi_{d_2},\mathbf A_n\rangle_v-\langle\Psi_{d_1},\mathbf A_n\rangle_v\leq q \bigl(D_{d_2,n}(v)-D_{d_1,n}(v)\bigr).
$$

For $\lambda\in\Lambda$
we consider an integer $z_1$ such that $\delta_{z_1}>r$ and $z_1\geq z_0(\lambda)$, an integer $z\geq z_1$ and $\eta_z\leq n<\eta_{z+1}$.
Then the following holds:
for every $s\in S_n^\lambda$ and every $x\in\ball[\delta_{z_1}]{\mathbf A_n}(s)$, it holds
$
\ball[d]{\mathbf A_n}(x)\subseteq \ball[d+\delta_{z_1}]{\mathbf A_n}(s)
$ and $
\ball[d]{\mathbf A_n}(s)\subseteq \ball[d+\delta_{z_1}]{\mathbf A_n}(x)
$ we get
$$
\langle\Psi_{d-\delta_{z_1}},\mathbf A_n\rangle_s\leq \langle\Psi_{d},\mathbf A_n\rangle_x\quad\text{and}\quad
\langle\Psi_{d},\mathbf A_n\rangle_x\leq \langle\Psi_{d+\delta_{z_1}},\mathbf A_n\rangle_s.
$$
It follows that 
\begin{align*}
\langle\Psi_{2\delta_{z_1}},\mathbf A_n\rangle_x-
\langle\Psi_{2\delta_{z_1}},\mathbf A_n\rangle_s&\leq
\langle\Psi_{3\delta_{z_1}},\mathbf A_n\rangle_s-
\langle\Psi_{2\delta_{z_1}},\mathbf A_n\rangle_s\\
&\leq q\, {\rm Pr}(2\delta_{z_1}\leq {\rm dist}(x,s)\leq 3\delta_{z_1})\\
&\leq q\, (D_{8\delta_{z},n}(s)-D_{\delta_{z_1},n}(s))\\
&< q (\beta_{z}(\lambda)-\alpha_{z_1}(\lambda))\\
&< q (\epsilon_{z_1}+\epsilon_z)
\intertext{and}
\langle\Psi_{2\delta_{z_1}},\mathbf A_n\rangle_s-
\langle\Psi_{2\delta_{z_1}},\mathbf A_n\rangle_x&\leq
\langle\Psi_{2\delta_{z_1}},\mathbf A_n\rangle_s-
\langle\Psi_{\delta_{z_1}},\mathbf A_n\rangle_s\\
&\leq q\, {\rm Pr}(\delta_{z_1}\leq {\rm dist}(x,s)\leq 2\delta_{z_1})\\
&< q (\epsilon_{z_1}+\epsilon_z).
\intertext{Thus}
|\langle\Psi_{2\delta_{z_1}},\mathbf A_n\rangle_x-
\langle\Psi_{2\delta_{z_1}},\mathbf A_n\rangle_s|&<q (\epsilon_{z_1}+\epsilon_z).
\end{align*}
Also,
\begin{align*}
|\langle\Psi_{2\delta_{z}},\mathbf A_n\rangle_s-
\langle\Psi_{2\delta_{z_1}},\mathbf A_n\rangle_s|
&\leq q (D_{\delta_z,n}(s)-D_{\delta_{z_1},n}(s))\\
&\leq q (D_{8\delta_z,n}(s)-D_{\delta_{z_1},n}(s))\\
&\leq q (\beta_z(\lambda)-\alpha_{z_1}(\lambda))\\
&<q (\epsilon_{z_1}+\epsilon_{z}).
\end{align*}
Moreover, 
$$
|\langle\Psi_{2\delta_{z}},\mathbf A_n\rangle_s-\langle\Psi_{2\delta_{z}},\mathbf A_n-\partial_{\mathbf A_n}C_n^\lambda\rangle_s|<4q\epsilon_{z}p(\lambda)/\lambda.
$$
and 
$$
\langle\phi,\mathbf A_n[C_n^\lambda]\rangle=\frac{\sum_{s\in S_n^\lambda}\langle\Psi_{2\delta_z},\mathbf A_n-\partial_{\mathbf A_n}C_n^\lambda\rangle_s}{\nu_{\mathbf A_n}(C_n^\lambda)^p}.
$$

Thus, as
$$
|\nu_{\mathbf A_n}(C_n^\lambda)-\lambda|<\epsilon_z p(\lambda)/\lambda,
$$
it holds

\begin{align*}
	\mathbb E[\langle\Psi_{2\delta_{z_1}},A_n\rangle_v\,{\mathbf 1}_{Z_n^{\lambda,z_1}}(v)]
	&= \frac{\sum_{v\in Z_n^{\lambda,z_1}}\nu_{\mathbf A_n}(v)\,\langle\Psi_{2\delta_{z_1}},\mathbf A_n\rangle_v}{\nu_{\mathbf A_n}(Z_n^{\lambda,z_1})}\\
	&\approx \frac{1}{\lambda}\sum_{v\in C_n^\lambda}\langle\Psi_{2\delta_{z}},\mathbf A_n\rangle_v\\
	&\approx \sum_{s\in S_n^\lambda}\langle\Psi_{2\delta_{z}},\mathbf A_n\rangle_s\\
	&\approx \lambda^p\langle\phi,\mathbf A_n[C_n^\lambda]\rangle
\end{align*}
Let $H_{z_1,n}$ be the (multivariate) cumulative distribution function of 
$$(\langle\Psi_{\delta_{z_1}},\mathbf A_n\rangle_\bullet,1-D_{\delta_{z_0(\lambda)}},\dots, 1-D_{\delta_{z_1}}, D_{8\delta_{z_1}}).$$
According to its definition we have
$v\in Z_n^{\lambda,z_1}$ if and only if 
$$(1-D_{\delta_{z_0(\lambda)}},\dots, 1-D_{\delta_{z_1}}, D_{8\delta_{z_1}})\in [0,1-\alpha_{z_0(\lambda)}(\lambda)]\times\dots\times [0,1-\alpha_{z_1}(\lambda)]\times [0,\beta_{z_1}(\lambda)].$$
Thus
$${\rm Pr}[\langle\Psi_{\delta_{z_1}},\mathbf A_n\rangle_v\leq x\text{ and }v\in Z_n^{\lambda,z_1}]=H_{n,z_1}(x,1-\alpha_{z_0(\lambda)}(\lambda),\dots,1-\alpha_{z_1}(\lambda),\beta_{z_1}(\lambda)).$$

It follows that
\begin{align*}
\mathbb E[\langle\Psi_{2\delta_{z_1}},A_n\rangle_v\,{\mathbf 1}_{Z_n^{\lambda,z_1}}(v)]&=\int_0^1 {\rm Pr}[\langle\Psi_{\delta_{z_1}},\mathbf A_n\rangle_v\leq x\text{ and }v\in Z_n^{\lambda,z_1}]\,{\rm d}x\\
&=\int_0^1 1-H_{n,z_1}(x,1-\alpha_{z_0(\lambda)}(\lambda),\dots,1-\alpha_{z_1}(\lambda),\beta_{z_1}(\lambda))\,{\rm d}x.
\end{align*}

According to Lemma~\ref{lem:locSP} there exists a (vector) random variable $\mathbf V_{z_1}$ such that 
$$(\langle\Psi_{\delta_{z_1}},\mathbf A_n\rangle_\bullet,1-D_{\delta_{z_0(\lambda)}},\dots, 1-D_{\delta_{z_1}}, D_{8\delta_{z_1}})\xrightarrow{\mathcal D}\mathbf V_{z_1}.$$
Let $H$ be the cumulative distribution function of $\mathbf V_{z_1}$. Then, as $n\rightarrow\infty$ it holds
$$
	\lim_{n\rightarrow\infty}\mathbb E[\langle\Psi_{2\delta_{z_1}},A_n\rangle_v\,{\mathbf 1}_{Z_n^{\lambda,z_1}}(v)]=\int_0^1 1-H(x,1-\alpha_{z_0(\lambda)}(\lambda),\dots,1-\alpha_{z_1}(\lambda),\beta_{z_1}(\lambda))\,{\rm d}x.
$$

As 
	$|\mathbb E[\langle\Psi_{2\delta_{z_1}},A_n\rangle_v\,{\mathbf 1}_{Z_n^{\lambda,z_1}}(v)]-\lambda^p\langle\phi,\mathbf A_n[C_n^\lambda]\rangle|$ goes to $0$ when $z_1$ goes to infinity (and $n$ grows in consequence),
	we get that $\langle\phi,\mathbf A_n[C_n^\lambda]\rangle$ converges hence
	$\seq{C}^\lambda$ is a cluster.
\end{proof}

We are now ready to prove our first clustering result:
\begin{lemma}
	 	Let $\seq{A}$ be a local convergent sequence of $\sigma$-structures. 
	 	Let $\sigma^+$ be the signature obtained from $\sigma$ by the addition of countably many unary symbols $M_R$ and $M_{i}$ ($i\in\bbbn$).
	 	Then marking by $M_i$ the cluster $C_n^{\lambda_i}$ (where
$\lambda_1>\lambda_2>\dots$ are the elements of $\Lambda$ order in decreasing order) and by $M_0$ the sequence of sets 
$$\Xseq{R}=\Xseq{A}\setminus\Xseq{W}\setminus\bigcup_{\lambda\in\Lambda}\Xseq{C}^\lambda$$
	we obtain clustering $\lift{\seq{A}}$ of $\seq{A}$ with the following properties:
 	\begin{itemize}
 		\item For every $i\in\bbbn$, $\bigl(M_{i}(\lift{\mathbf A_n})\bigr)_{n\in\bbbn}$  is a universal globular cluster, and $M_i(\lift{\mathbf A_n})$ asymptotically consists in a set inducing $p(\lambda_i)/\lambda_i$ 
 		 disjoint connected substructures, each  of measure $\lambda_i+o(1)$ in $\mathbf A_n$.
 		\item $\bigl(M_R(\mathbf A_n^+)\bigr)_{n\in\bbbn}$  is a residual cluster.
 	\end{itemize}
\end{lemma}
\begin{proof}	
That $\lift{A}$ is clustering follows from Lemma~\ref{lem:stable}.
That $\seq{C}^\lambda$ is a universal cluster is trivial as the constructions and proofs can be achieved the same way (with same result) in any conservative lift of $\seq{A}$.
The sequence $\Xseq{R}$ is obviously residual.
\end{proof}

We are now ready to prove Theorem~\ref{thm:main}, which we state now in the following more precise form.
\begin{great}
 \begin{theorem}
 \label{thm:bigmain}
  	Let $\seq{A}$ be a local convergent sequence of $\sigma$-structures. Then there exists a signature
  $\sigma^+$ (obtained from $\sigma$ by the addition of countably many unary symbols $M_{i,j,k}$ ($i\in\bbbn$, $1\leq j\leq a_i$, $1\leq k\leq b_{i,j}$), of a unary symbol $M_R$ and of unary symbol $M_S$), a sequence $\lambda_1>\lambda_2>\dots$ a positive reals and a clustering $\lift{\seq{A}}$ of $\seq{A}$ with the following properties:
 	\begin{itemize}
 		\item For every $i\in\bbbn$, $1\leq j\leq a_i$, and $1\leq k\leq b_{i,j}$,
 	   $\Xseq{G^{i,j,k}}=M_{i,j,k}(\lift{\seq{A}})$ is a 
 	   globular cluster of $\seq{A}$ such that $\lim\nu_{\seq{A}}(\Xseq{G}^{i,j,k})=\lambda_i$, that is a cluster such that for every positive real $\epsilon$ there is an integer $d$ which satisfies
 	   $$
 	   \lambda_i-\epsilon<\liminf_{n\rightarrow\infty}\max_{v_n\in G^{i,j,k}_n}\nu_{\mathbf A_n}(\ball[d]{v_n})\leq\lim_{n\rightarrow\infty}\nu_{\mathbf A_n}(G_n^{i,j,k})=\lambda_i.
 	   $$
 	   \item $\Xseq{R}=M_{R}(\lift{\seq{A}})$ is a 
 	   residual cluster of $\seq{A}$, that is a cluster such that
 	   for every integer $d$ it holds
 	   $$
 	  \limsup_{n\rightarrow\infty}\max_{v_n\in G^{i,j,k}_n}\nu_{\mathbf A_n}(\ball[d]{v_n})=0.
 	   $$
 	   \item The sequence $\Xseq{S}$ is negligible, that is such that for every integer $d$ it holds
 	  $$
 	  \limsup_{n\rightarrow\infty}\nu_{\mathbf A_n}(\ball[d]{S_n})=0.
 	  $$
 		\item The marks partition the sets $A_n$ is a stable way, that is
 	$$\lim\nu_{\seq{A}}(\Xseq{R})+\sum_{i\geq 1}\lim\nu_{\seq{A}}(\Xseq{G}^{i,j,k})=1.$$
 	\item Clusters $\Xseq{G}^{i,j,k}$ and $\Xseq{G}^{i',j',k'}$ are interweaving (i.e. $\Xseq{G}^{i,j,k}\between \Xseq{G}^{i',j',k'}$) if and only if $i=i'$ and $j=j'$.
 	 \item The clusters $\bigcup_{k=1}^{b_{i,j}}\Xseq{G}^{i,j,k}$ (grouping interweaving clusters) are universal.
 	 \item The number $N_i=\sum_{j=1}^{a_i}b_{i,j}$ of clusters with limit measure $\lambda_i$ is 
$$
N_i=\frac{1}{\lambda_i}
\lim_{T\rightarrow\infty}\frac{1}{2T}\int_{-T}^{+T}  
\biggl[\sum_{w\geq 1}\biggl(
\lim_{d\rightarrow\infty}\int_{S_\sigma}
k(\psi_{d,w})\,{\rm d}\mu\biggr)
\,\frac{(is)^w}{w!}\biggr] e^{-i\lambda_i s}
\,{\rm d}s,
$$
where $\psi_{d,w}$ is the formula
$$
\psi_{d,w}(x_1,\dots,x_{w+1}):=\bigwedge_{i=1}^w {\rm dist}(x_1,x_i)\leq d.
$$
 	\end{itemize}	
  \end{theorem}	
\end{great}

\begin{proof}
By construction, the number of connected components of $\mathbf{A}_n[C^\lambda_n]$ is asymptotically 
$p(\lambda)/\lambda$ and each of these connected components has asymptotically measure $\lambda$.
Let $\mathbf B_{n,1},\dots,\mathbf B_{n,k_n}$ be the connected components of $\mathbf A_n[C_n^\lambda]$.
If there is a local formula $\phi$ such that
$$\lim_{n\rightarrow\infty}\min_i\langle\phi, \mathbf B_i\rangle\neq
\lim_{n\rightarrow\infty}\min_i\langle\phi, \mathbf B_i\rangle$$
we can break $C^\lambda$ into smaller universal clusters.
At the end of the day, we get a clustering of $\seq{A}$ into countably many clusters, such that each cluster $\Xseq{C}^i$ has asymptotically $k_i$ connected components with same asymptotic measure and same asymptotic profile. It follows that $\Xseq{C}^i$ it the disjoint union of $k_i$ interweaving clusters.

The statement giving the number $N_i$ of clusters with measure $\lambda_i$ is due to the equality $N_i=p(\lambda_i)\lambda_i$ and the application of L\'evy's theorem (Theorem~\ref{thm:Levy}) for the computation of $p(\lambda_i)$ from the characteristic function $\gamma_\infty(\mu,t)$ associated to the formulas ${\rm dist}(x_1,x_2)\leq d$ by Lemma~\ref{lem:approx1}.
\end{proof}

A direct consequence of Theorem~\ref{thm:bigmain} stands in the following complete characterization of the globular clusters of a local convergent sequence.
\begin{great}
\begin{theorem}
We have the following complete characterization of the globular clusters of a local convergent sequence $\seq{A}$: For a sequence $\Xseq{X}$ of subsets of $\seq{A}$ the following are equivalent
 	 \begin{enumerate}
 	 	\item $\Xseq{X}$ is a globular cluster of $\seq{A}$;
 	 	\item there exists a negligible sequence $\Xseq{N}$ 
 	 and integers $i,j$ (with $1\leq j\leq a_i$) such that for every integer $n$ it holds
 	$$X_n\Delta N_n\in\{G_n^{i,j,1},G_n^{i,j,2},\dots,G_n^{i,j,b_{i,j}}\}.$$
 \end{enumerate}
\end{theorem}
\end{great}
\begin{proof}
If $\Xseq{X}$ is obtained by interweaving clusters from 
$\{G_n^{i,j,1},G_n^{i,j,2},\dots,G_n^{i,j,b_{i,j}}\}$, then $\Xseq{X}$ is a cluster, which is obviously globular. Hence (2)$\Rightarrow$(1).
Conversely, let $\Xseq{X}$ be a globular cluster. 
As the partition is stable there exists, for every $\epsilon>0$, integers $i_0$ and $n_0$ such that for every $n\geq n_0$ it holds 
$$\sum_{i>i_0}\sum_{j=1}^{a_i}\nu_{\mathbf A_n}(Z_n^{i,j})<\epsilon.$$
Then notice that $\Xseq{X}\cap\Xseq{R}\approx\Xseq{0}$ as $\Xseq{R}$ is residual and $\Xseq{X}$ is not. 
According to Lemma~\ref{lem:expmix}, for each $i,j,k$ it either holds
$\Xseq{X}\cap\Xseq{G}^{i,j,k}\approx\Xseq{0}$ or $\Xseq{X}\between\Xseq{G}^{i,j,k}$.
Let $n_1\geq n_0$ be such that for every $n\geq n_0$ 
and every integers $i,j$ with $i\leq i_0$ and $\Xseq{X}\cap \Xseq{Z}^{i,j}\approx 0$ it holds
$$
\nu_{\mathbf A_n}(R_n\cap X_n)<\epsilon
\text{ and }
\nu_{\mathbf A_n}(R_n\cap Z_n^{i,j})<\frac{\epsilon}{\sum_{i=1}^{i_0}a_i}.
$$
Then, letting $\epsilon<\lim\nu_{\seq{A}}(\Xseq{X})/4$ we get that there exists integers $i,j,k$ such that $\Xseq{X}\between\Xseq{G}^{i,j,k}$. It follows that $\Xseq{X}\between\Xseq{G}^{i',j',k'}$ if and only if $i=i'$ and $j=j'$. Thus
for every $(i',j')\neq (i,j)$ it holds $\Xseq{X}\cap \Xseq{Z}^{i,j}\approx 0$, and thus it holds $\liminf\nu_{\seq{A}}(\Xseq{X}\cap\Xseq{Z}^{i,j})\geq 1-3\epsilon$. Letting $\epsilon\rightarrow 0$ we get that $\Xseq{Z}^{i,j}\setminus\Xseq{X}$ is negligible.
As $\Xseq{X}$ is globular and as $\Xseq{Z}^{i,j}$ consists in connected components with same positive limit measure as $\Xseq{X}$ selecting from $Z^{i,j}_n$ a connected component with maximal intersection with $X_n$ we get a globular cluster $\Xseq{Y}$ such that $\Xseq{Y}\approx\Xseq{X}$.
\end{proof}

\addtocontents{toc}{\bigskip}
\section{Conclusion and Future Work}

In this paper we have shown that a the local convergence of a sequence of finite structures is enough to obtain properties that cannot be expressed directly by means of a first-order formula: one can cluster the sequence into countably many globular clusters and a residual cluster.
It is perhaps surprising that one can do so just from local convergence. The obtained clustering is natural and continuous. We believe that this analysis may be of interest in cluster analysis itself if only by the concepts that naturally arose in this study.

On the other hand, we feel that
 this is only the beginning of the story. Particularly because of their connection to expanders, we would like to further refine our clustering and find further expanding (non globular) clusters. However this will require to consider a stronger notion of convergence, such as generalized local-global convergence. Our generalization of local-global convergence extends the notion of local-global convergence based on the colored neighborhood metric of Bollob\'as and Riordan \cite{bollobas2011sparse}, which was introduced by Hatami, Lov\'asz, and Szegedy \cite{hatami2014limits}. This will be the subject of a forthcoming paper.

\providecommand{\noopsort}[1]{}\providecommand{\noopsort}[1]{}

\end{document}